\documentclass[11pt,a4paper]{article}
\usepackage[T1]{fontenc}
\usepackage{lmodern}
\usepackage[margin=27mm]{geometry}
\usepackage{amsmath,amssymb,amsthm,mathtools,microtype}
\usepackage{enumitem}
\usepackage{tikz-cd}
\usepackage{float}
\usepackage[colorlinks=true,linkcolor=blue,citecolor=blue,urlcolor=blue]{hyperref}
\usepackage{bookmark}
\numberwithin{equation}{section}
\newtheorem{theorem}{Theorem}[section]
\newtheorem{lemma}[theorem]{Lemma}
\newtheorem{proposition}[theorem]{Proposition}

\theoremstyle{definition}\newtheorem{assumption}{Assumption}[section]
\theoremstyle{remark}
\newcommand{\R}{\mathbb R}
\newcommand{\Z}{\mathbb Z}
\newcommand{\T}{\mathbb T}
\newcommand{\Torus}{\mathbb T}
\newcommand{\id}{\operatorname{id}}
\newcommand{\Lop}{\mathcal L}
\newcommand{\Dop}{\mathcal D}
\newcommand{\Kop}{\mathcal K}
\newcommand{\Sop}{\mathcal S}
\newcommand{\darc}{\partial_{s}}

\newcommand{\W}{W^{1,\infty}}

\newcommand{\F}{\mathcal F}

\allowdisplaybreaks[2]
\hypersetup{pdftitle={Tangential stability and fully discrete convergence of the classical BGN scheme for curve shortening flow},pdfauthor={Qiqi Rao}}
\title{Tangential stability and fully discrete convergence\\of the classical BGN scheme for curve shortening flow}
\author{Qiqi Rao\thanks{Corresponding author. Department of Mathematics,
The Chinese University of Hong Kong, Shatin, Hong Kong, China.
Email: \href{mailto:qiqirao@cuhk.edu.hk}{qiqirao@cuhk.edu.hk}.}}
\date{}
\begin{document}\maketitle
\begin{abstract}
We prove fully discrete convergence of the classical
Barrett--Garcke--N\"urnberg (BGN) scheme for curve-shortening flow of
smooth embedded closed planar curves. The main obstruction is that the
mass form controls only normal motion, whereas the tangential motion is
determined implicitly by the curvature equation and governs the
parametrization. The usual length-decay estimate therefore does not
control perturbations of the full position update. We separate temporal
and spatial errors through the time-semidiscrete BGN solution. A scalar
normal resolvent and exact curvature and length identities yield uniform
regularity and first-order time convergence. For the spatial analysis,
an adapted normal--tangential norm gives a near-contractive estimate for
the linearized update. Gauss--Lobatto cancellations and an exact
covariance identity for the assembled nodal normals produce an $H^1$
one-step defect of order $h^{k+1}$, while an exact difference identity
controls the nonlinear remainder. For every fixed $k\ge1$, including
the original piecewise linear method, we obtain the matched
$W^{1,\infty}$ error bound $C(\tau+h^k)$ on periodic quasi-uniform meshes
with $h\le c\tau^2$. All discrete steps are uniquely solvable, and the
numerical curves remain regular and embedded. To the best of our
knowledge, this is the first fully discrete convergence result for the
classical BGN curve-shortening scheme without additional stabilization.


\end{abstract}

\section{Introduction}

The parametric finite element methods of Barrett, Garcke and N\"urnberg
(BGN) couple a normal-velocity law to a weak curvature equation, which
also determines tangential node motion. These methods were introduced
for second- and fourth-order curve evolutions
\cite{BGN2007JCP,BGN2007} and extended to evolving surfaces
\cite{BGN2008}. For curve shortening, freezing the normal and metric
at the beginning of a time step gives a linear system and a discrete
length-decay inequality without a time-step restriction
\cite{BGN2007}. The tangential motion redistributes nodes along the
curve, with mesh properties reviewed in \cite{BGN2020}.
Applications of these weak curvature formulations include two-phase
Navier--Stokes flow \cite{BGN2015NavierStokes}, surfactant transport
\cite{BGN2015Surfactant}, fluidic membranes \cite{BGN2016Membranes},
solid-state dewetting \cite{ZhaoJiangBao2021}, and planar Willmore flow
\cite{GarckeNurnbergZhao2025Willmore}. Higher-order BGN extensions
include second-order and BDF time discretizations
\cite{JiangSuZhang2024SecondOrder,JiangSuZhang2024BDF}, Runge--Kutta
methods based on continuous equations on each time interval
\cite{MaRao2026}, and isoparametric spatial discretizations
\cite{GarckeEtAl2025}.

We study the classical BGN scheme for closed planar curves with
outward normal velocity $V_{\mathrm n}=-\kappa$, where $\kappa$ is the
signed curvature, on a fixed interval $[0,T]$ before the first singular time.
The objective is an error estimate for the computed positions while
retaining the original discrete equations. Three features enter this
analysis. The time mass form controls only the normal component,
whereas the implicitly selected tangential motion determines the
parametrization. Uniform derivative bounds and preservation of regularity
and embeddedness must be proved for the time-semidiscrete curves.
Finally, the assembled nodal normals and the changing metric enter
both the consistency error and the linearized update. Discrete length
decay alone does not supply these estimates. The central issue is
therefore stability rather than formal consistency with the normal
velocity law: the curvature coupling must recover control of the full
parametrized update from a mass form that sees only its normal component.

Error estimates are available for several related discretizations.
Li \cite{Li2020} analyzed Dziuk's linearly implicit curve-shortening
method, whose time term acts on the full vector velocity. Elliott and
Fritz \cite{ElliottFritz2017} proved semidiscrete estimates for a
DeTurck reparametrization with a fixed positive parameter. Their
estimate is not uniform as that parameter tends to zero.
For BGN-type methods, Bai and Li \cite{BaiLi2025} proved convergence
of a stabilized curve-shortening scheme with polynomial degree
$k\ge2$. They also characterized its limiting particle trajectories
through a minimal-deformation-rate system. Bai, Garcke and Veerapaneni
\cite{BaiGarckeVeerapaneni2026} analyzed a stabilized BGN-type method
for curves transported by a prescribed background velocity field.
Other constructions prescribe tangential motion through harmonic maps
\cite{DuanLi2024} or minimal deformation rates \cite{GaoEtAl2026MDR}.
Huang, Li and Tang \cite{HuangLiTang2026} proved convergence for a
closed-surface method that imposes this latter motion through an
additional elliptic velocity equation and uses an $L^2$-projected normal.

The distinction from the stabilized BGN method is already visible in
the discrete tangential equation. For test functions whose nodal
values are orthogonal to the assembled average normals, the normal
mass term vanishes. In the classical scheme, the stiffness of the
velocity then equals the negative stiffness residual of the current
position divided by the time step. This residual generally does not
vanish on a discrete curve. The stabilization in
\cite[(1.5)--(1.8)]{BaiLi2025} cancels it and enforces a homogeneous
tangential equation for the discrete velocity, which is used in that
convergence analysis. The transport result
\cite{BaiGarckeVeerapaneni2026} uses a related stabilization and does
not include the feedback from curvature to velocity. The additional
velocity equation in \cite{HuangLiTang2026} also defines a different
discrete update. The question addressed here is therefore how to
control perturbations of the original BGN update, including its
tangential residual. In particular, neither the full-vector coercivity
available for Dziuk's method nor a separately prescribed or stabilized
tangential equation is available for the scheme studied here.

We use the time-semidiscrete BGN solution as a reference trajectory
to separate temporal and spatial errors. This solution is obtained
from the mixed scheme by using continuous spaces and exact integration.
Eliminating the tangential velocity gives a scalar resolvent for the
normal component. Exact curvature and arclength identities then convert
the geometric update into a scalar diffusion recurrence for curvature
and length. A nonlocal cancellation combines the tangential transport,
the drift of the normalized-arclength coordinate, and the Schur
complement term into the nonlinearity of the exact curvature equation.
Discrete parabolic smoothing first closes a low-Sobolev continuation
argument. A near-contractive high-order estimate then gives uniform
regularity and improves the temporal defects to order $\tau$.
Normalized arclength, together with a marked point and its tangent,
fixes the remaining phase and rigid-motion freedom in the comparison
of parametrizations.

The spatial analysis addresses the missing tangential coercivity at the
level of the step map. Linearization of the time-semidiscrete update
shows that the tangential perturbation is coupled elliptically to the
normal perturbation. The adapted norm therefore measures the normal
perturbation together with the difference between the tangential
perturbation and its elliptically induced part. In this norm the
linearized step has stability
factor $1+C\tau$. This estimate is transferred to the finite element
linearization without replacing the assembled nodal normals. The
Gauss--Lobatto orthogonality and zero-mean cancellations, together with
an exact covariance identity for these normals, improve the local
$H^1$ consistency defect from the standard interpolation order $h^k$
to $h^{k+1}$. An exact second-order difference identity then controls
the nonlinear remainder around the interpolated reference trajectory.

To the best of our knowledge, this paper gives the first fully discrete
error analysis of the classical BGN scheme for smooth closed planar
curve-shortening flow without additional stabilization, including the
original piecewise linear scheme. We also cover its continuous $P^k$
extension for every fixed $k\ge1$, with Gauss--Lobatto mass lumping
and exact stiffness integration. Let $h$ be the size of a periodic
quasi-uniform mesh, and set $\tau=T/M$ and $t_m=m\tau$, with
$M\in\mathbb N$.
Write $\widehat x_h^m$ for the finite element position in the
normalized-arclength coordinate of the time-semidiscrete curve, and
$\widehat X(t_m)$ for the exact normalized-arclength parametrization
with the marked point specified in Section~\ref{sec:reference-maps}.
Under Assumption~\ref{ass:exact-flow}, with nodally interpolated initial
data, Theorem~\ref{un:thm:main} gives
\[
 \max_{0\le m\le M}
 \|\widehat x_h^m-\widehat X(t_m)\|_{W^{1,\infty}}
 \le C(\tau+h^k),\qquad
 0<h\le c\tau^2,\quad 0<\tau\le\tau_0,
\]
where $c,C,\tau_0>0$ are independent of $h,\tau,m$.
All mixed steps are uniquely solvable, and the numerical curves
remain regular and embedded. The same bound controls the Hausdorff
distance. For the time-semidiscrete solution,
Theorem~\ref{up:thm:time} gives first-order convergence of curvature
in $H^\sigma$, length, and the matched position in $C^1$, for
$3/2<\sigma<2$.

The temporal argument gives the matched $C^1$ error $C\tau$. For the
spatial error $e_h^m=x_h^m-I_hx^m$, the one-step consistency,
linearized stability, and nonlinear remainder estimates lead to
the recurrence
\[
 \|e_h^{m+1}\|_{x^{m+1}}
 \le (1+C\tau)\|e_h^m\|_{x^m}+Ch^{k+1}
       +Ch^{-1/2}\|e_h^m\|_{x^m}^2.
\]
The condition $h\le c\tau^2$ converts the derivative-comparison loss
$h/\tau$ into a per-step $O(\tau)$ perturbation. Under the stopping
bound $\|e_h^m\|_{x^m}\le Kh^{k+1}/\tau$, it also gives, since $k\ge1$,
\[
 h^{-1/2}\|e_h^m\|_{x^m}^2
 \le K\frac{h^{k+1/2}}{\tau}\|e_h^m\|_{x^m}
 \le Kc^{3/2}\tau^2\|e_h^m\|_{x^m}.
\]
Thus the quadratic term is included in the linear growth factor
$1+C\tau$. The resulting $H^1$ supercloseness to the reference
interpolant and one inverse estimate give the spatial
$W^{1,\infty}$ error $Ch^k$. Adding the temporal error proves the
estimate in Theorem~\ref{un:thm:main}.

These arguments separate the stability question created by a
normal-only mass form from the consistency estimates for the particular
geometric evolution. For analogous parametric schemes with implicitly
determined tangential motion, the proof identifies three ingredients
that can be checked independently: a scalar reduction of the normal
equation, a norm adapted to the normal--tangential coupling, and a
high-order consistency estimate for the assembled geometric mass form.

The paper is organized as follows.
Section~\ref{sec:setting} states the scheme and main result.
Sections~\ref{sec:time}--\ref{sec:upgrade} establish the time-semidiscrete
estimates, Section~\ref{sec:consistency} proves consistency and
linearized stability, and Section~\ref{sec:nonlinear} completes the
fully discrete error estimate.

\section{The BGN scheme and main results}\label{sec:setting}

\subsection{Geometric setting}\label{sec:geometric-setting}

Throughout, smooth means $C^\infty$ in the indicated variables.
Let $\Gamma(t)\subset\R^2$, $0\le t\le T$, denote the exact closed
curves. Fix $\Torus=\R/\Z$ and a counterclockwise initial
parametrization $x^0:\Torus\to\Gamma(0)$ satisfying
$|\partial_\rho x^0|=|\Gamma(0)|$, from a prescribed point $x^0(0)$.
The material label $\rho$ identifies $x^0(\rho)$ and is retained by
the BGN updates and the finite element mesh.

For a regular counterclockwise parametrization $y:\Torus\to\R^2$,
write $\kappa(\rho)$, $\mathbf n(\rho)$ and $\mathbf t(\rho)$ for the
signed curvature, outward unit normal and unit tangent at $y(\rho)$.
Arclength $s$ is measured from $y(0)$, with
\[
 J(a,b)=(b,-a),\qquad
 \mathbf t=\frac{\partial_\rho y}{|\partial_\rho y|},\qquad
 \mathbf n=J\mathbf t,\qquad
 \partial_s=|\partial_\rho y|^{-1}\partial_\rho.
\]
The curvature is determined by the Frenet identities
\begin{equation}\label{eq:orientation}
 \darc\mathbf{t}=-\kappa\mathbf{n},
 \qquad \darc\mathbf{n}=\kappa\mathbf{t}.
\end{equation}
Our curvature has the opposite sign to the scalar curvature in
\cite[Section~2.2, Lemma~13]{BGN2020}. The normal velocity of
curve-shortening flow is $-\kappa$.
For a parameter function $f=f(\rho)$, integration on
$\Gamma=y(\Torus)$ denotes the weighted parameter integral
\[
 ds=|\partial_\rho y|\,d\rho,\qquad
 \int_\Gamma f\,ds:=\int_0^1 f(\rho)|\partial_\rho y(\rho)|\,d\rho.
\]
Primes on parameter functions denote $\partial_\rho$ and on bilinear
forms denote differentiation with respect to the geometry.

Unless another domain or measure is specified, function spaces and norms
use the displayed coordinate on $\Torus$. We write $H^r=H^r(\Torus)$,
$\|\cdot\|_2=\|\cdot\|_{L^2(\Torus)}$ and
$\|\cdot\|_\infty=\|\cdot\|_{L^\infty(\Torus)}$.
For $f=f(\rho)$ and integers $r\ge0$,
\begin{equation}\label{eq:integer-sobolev-norm}
 \|f\|_{H^r}^2=\sum_{j=0}^r\int_0^1
                   |\partial_\rho^j f(\rho)|^2\,d\rho,
 \qquad \overline f=\int_0^1 f(\rho)\,d\rho.
\end{equation}
For vectors, squared norms are summed over components. This derivative-sum norm is
fixed in estimates with factor $1+C\tau$. Fractional norms use periodic
Fourier weights, with a fixed exponent $\sigma\in(3/2,2)$.
The norms $H^r(ds)$ use arclength derivatives and measure, and
$H^{-1}(ds)$ is dual to $H^1(ds)$ under the arclength pairing.
Constants $C>0$ may vary between estimates, with dependencies stated
in each result. We write $A\asymp B$ for $A\le CB$ and $B\le CA$.

For nonempty compact sets $A,B\subset\R^2$, the Hausdorff distance is
\[
 d_H(A,B)=\max\left\{
 \sup_{a\in A}\inf_{b\in B}|a-b|,
 \sup_{b\in B}\inf_{a\in A}|a-b|
 \right\}.
\]

\subsection{Finite element scheme}

The assembled mixed equations give the position form used in
Section~\ref{sec:consistency}.

Fix an integer $k\ge1$ and a family of periodic quasi-uniform partitions $\mathcal T_h$ of the material domain $\Torus$. If $h=\max_{K\in\mathcal T_h}|K|$, the quasi-uniformity constant is independent of $h$. Let $P^k(K)$ be the polynomials of degree at most
$k$ on $K$. Set
\begin{equation}\label{eq:space}
 V_h=\{\eta_h\in C^0(\Torus):\eta_h|_K\in P^k(K)
              \text{ for all }K\in\mathcal T_h\}.
\end{equation}
On each element we use the $(k+1)$-point Gauss--Lobatto rule. Its positive weights $w_{K,i}$ include the reference-element scale. The rule is exact for polynomials of degree at most $2k-1$. Global nodes, with periodic endpoints identified, are denoted by $\rho_i$.
In elementwise sums, $i\in K$ means that $\rho_i$ is a Gauss--Lobatto
node of $K$, and $K\ni i$ ranges over the elements incident to that node.
Let $I_h$ be interpolation at these nodes and set $x_h^0=I_hx^{0}$.
Thus $x_h^0(\rho_i)=x^0(\rho_i)$, with fixed nodal labels $\rho_i$.

Choose the number of time steps $M\in\mathbb N$ and set
$\tau=T/M$ and $t_m=m\tau$ for $0\le m\le M$. Discrete time levels are written
as superscripts. The subscript $h$ indicates the finite element
discretization, while $i$ and $K$ identify spatial nodes and elements.
Given $x_h^m\in V_h^2$, set $\Gamma_h^m=x_h^m(\Torus)$.
Its elementwise outward unit normal is the parameter function
$\mathbf n_h^m:\Torus\to\R^2$, with one-sided values at element
endpoints, given by
\begin{equation}\label{eq:discrete-geometry}
 \mathbf{n}_h^m=\frac{J(x_h^m)'}{|(x_h^m)'|}.
\end{equation}
The formulas are used when $|(x_h^m)'|>0$. Positivity throughout the computation is a conclusion of the proof. For scalar or vector quantities $f,g$ with one-sided nodal values,
and vector fields $u,\phi\in H^1(\Torus)^2$, define
\begin{align}
 (f,g)_h^m
  &=\sum_{K\in\mathcal T_h}\sum_{i\in K}
      w_{K,i}\bigl|(x_h^m|_K)'(\rho_i)\bigr|
      f|_K(\rho_i)\cdot g|_K(\rho_i),\label{eq:lumped-inner}\\
 a_h^m(u,\phi)
  &=\int_{\Torus}\frac{u'\cdot\phi'}{|(x_h^m)'|}\,d\rho.
       \label{eq:stiffness}
\end{align}
The form $a_h^m$ uses the finite element geometry and exact integration.

For $0\le m<M$, given $x_h^m$, the classical BGN step seeks
$(x_h^{m+1},\kappa_h^{m+1})\in V_h^2\times V_h$.
Here $\kappa_h^{m+1}$ is the scalar unknown of the mixed scheme.
The equations are
\begin{subequations}\label{eq:bgn}
\begin{align}
 v_h^m&=\frac{x_h^{m+1}-x_h^m}{\tau},\label{eq:bgn-velocity}\\
 (v_h^m\cdot\mathbf{n}_h^m,\chi_h)_h^m
  &=-(\kappa_h^{m+1},\chi_h)_h^m
                 &&\text{for all }\chi_h\in V_h,\label{eq:bgn-normal}\\
 (\kappa_h^{m+1}\mathbf{n}_h^m,\phi_h)_h^m
  &=a_h^m(x_h^{m+1},\phi_h)
                 &&\text{for all }\phi_h\in V_h^2.\label{eq:bgn-curvature}
\end{align}
\end{subequations}
The normal and metric are evaluated at time $t_m$.

At a shared node define the total mass and the weighted average normal by
\begin{align}
 \mu_i^m&=\sum_{K\ni i}w_{K,i}\bigl|(x_h^m|_K)'(\rho_i)\bigr|,\label{eq:nodal-mass}\\
 \omega_i^m&=\frac{1}{\mu_i^m}
       \sum_{K\ni i}w_{K,i}\bigl|(x_h^m|_K)'(\rho_i)\bigr|
                    \mathbf{n}_h^m|_K(\rho_i).
                    \label{eq:averaged-normal}
\end{align}
The vector $\omega_i^m$ need not have unit length. Equation~\eqref{eq:bgn-normal} gives
\begin{equation}\label{eq:curvature-elimination}
 \kappa_h^{m+1}(\rho_i)=-v_h^m(\rho_i)\cdot\omega_i^m.
\end{equation}
Consequently, with
\begin{equation}\label{eq:normal-form}
 b_h^m(u,\phi)=\sum_i\mu_i^m
        (u(\rho_i)\cdot\omega_i^m)
        (\phi(\rho_i)\cdot\omega_i^m),
\end{equation}
the equivalent velocity and position equations, for every
$\phi_h\in V_h^2$, are
\begin{align}
 b_h^m(v_h^m,\phi_h)+\tau a_h^m(v_h^m,\phi_h)
     &=-a_h^m(x_h^m,\phi_h),\label{eq:velocity-eliminated}\\
 a_h^m(x_h^{m+1},\phi_h)
       +\tau^{-1}b_h^m(x_h^{m+1}-x_h^m,\phi_h)&=0.
       \label{eq:position-eliminated}
\end{align}
The nodal matrix in $b_h^m$ is $\mu_i^m\omega_i^m\otimes\omega_i^m$. We retain this matrix in the analysis rather than replace it by an average of one-sided normal projectors.

\subsection{Time semidiscretization and parametrizations}\label{sec:reference-maps}

Let $x^m:\Torus\to\R^2$ denote the time-semidiscrete BGN solution
in the material coordinate, starting from $x^0$, and set
$\Gamma^m=x^m(\Torus)$. Thus $x^m(\rho)$ is the position at step $m$
of the particle labelled by $x^0(\rho)$.
Its curvature and frame at $x^m(\rho)$ are denoted by
$\kappa^m(\rho),\mathbf n^m(\rho),\mathbf t^m(\rho)$.
For $0\le m<M$, the velocity and its components in the frame
$(\mathbf n^m,\mathbf t^m)$ are
\[
\begin{aligned}
 v^m(\rho)&=\frac{x^{m+1}(\rho)-x^m(\rho)}{\tau},\\
 v_{\mathrm n}^m&=v^m\cdot\mathbf n^m,\qquad
 v_{\mathrm t}^m=v^m\cdot\mathbf t^m,\\
 v^m&=v_{\mathrm n}^m\mathbf n^m+v_{\mathrm t}^m\mathbf t^m.
\end{aligned}
\]
\paragraph{The position update.}
Curve integrals and $\partial_s$ use the
conventions of Section~\ref{sec:geometric-setting} with $y=x^m$,
including when $\partial_s$ acts on $x^{m+1}$.
The BGN step seeks
$x^{m+1}\in H^1(\Torus;\R^2)$ such that, for every
$\phi\in H^1(\Torus;\R^2)$,
\begin{equation}\label{time:bgn}
 \int_{\Gamma^m}(v^m\cdot\mathbf n^m)
                    (\phi\cdot\mathbf n^m)\,ds
 +\int_{\Gamma^m}\partial_s x^{m+1}\cdot\partial_s\phi\,ds=0.
\end{equation}

For a regular input $x$ for which this problem is uniquely solvable,
$\mathcal F_\tau(x)$ denotes its solution with $x^m$ replaced by $x$.
Likewise, $\mathcal F_{h,\tau}$ is the finite element position solution
operator defined by \eqref{eq:position-eliminated}. Thus
\begin{equation}\label{eq:step-operators}
 x^{m+1}=\mathcal F_\tau(x^m),\qquad
 x_h^{m+1}=\mathcal F_{h,\tau}(x_h^m).
\end{equation}
Section~\ref{sec:consistency} specifies their solvability neighborhoods
and estimates their derivatives with respect to the input curve.

\paragraph{Material and normalized representations.}
For each regular $x^m$, define $\chi^m:\Torus\to\Torus$ by its
lift on $[0,1]$:
\begin{equation}\label{eq:normalized-arclength-coordinate}
 \begin{aligned}
 |\Gamma^m|&=\int_0^1|\partial_\rho x^m(\rho)|\,d\rho,\quad s=\int_0^\rho|\partial_\eta x^m(\eta)|\,d\eta,
 \quad \xi=\chi^m(\rho)=\frac{s}{|\Gamma^m|}.
 \end{aligned}
\end{equation}
Here $s$ is arclength from $x^m(0)$ and $\xi$ is normalized arclength.
Since $(\chi^m)'>0$ and $\chi^m(0)=0$, $\chi^m(1)=1$, periodic
extension gives an orientation-preserving diffeomorphism.
The inverse $(\chi^m)^{-1}:\Torus\to\Torus$ sends $\xi$ to $\rho$.

A hat denotes a normalized-arclength representation. For any scalar
or vector field $g^m$ defined in the material parameter, set
\[
 \widehat g^m:=g^m\circ(\chi^m)^{-1},\qquad
 \widehat g^m(\chi^m(\rho))=g^m(\rho).
\]
Thus $\widehat\kappa^m(\xi)$ is the curvature at
$\widehat x^m(\xi)=x^m(\rho)$ when $\xi=\chi^m(\rho)$.
The same convention applies to the frame, velocity and its components.
Unhatted BGN fields always retain their material representation, even
when the argument is omitted. By
\eqref{eq:normalized-arclength-coordinate},
\begin{equation}\label{eq:reference-parametrization}
 \widehat x^m(0)=x^m(0),\qquad
 |\partial_\xi\widehat x^m|=|\Gamma^m|.
\end{equation}
Norms and means of $g^m$ use $\rho$, whereas those of $\widehat g^m$
use $\xi$. These coordinate norms generally differ. Arclength
integration and differentiation satisfy
\[
 ds=|\partial_\rho x^m|\,d\rho=|\Gamma^m|\,d\xi,\qquad
 \widehat{\partial_s g^m}
 =|\Gamma^m|^{-1}\partial_\xi\widehat g^m.
\]
On hatted fields we therefore write
$\partial_s=|\Gamma^m|^{-1}\partial_\xi$.
The convention also applies to finite element fields $g_h^m$ with the
same map $\chi^m$ defined by the time-semidiscrete solution.
In particular, $\widehat x_h^m$ need not
have constant parameter speed. Initially $\chi^0=\id$, so
$\widehat g^0=g^0$.

\paragraph{The exact parametrization.}
Let $\widehat X(\cdot,t):\Torus\to\Gamma(t)$ be the positively oriented
normalized-arclength parametrization which is proved to be the limit of BGN scheme.
Define $\widehat\kappa(\xi,t)$, $\widehat{\mathbf n}(\xi,t)$ and
$\widehat{\mathbf t}(\xi,t)$ directly as the curvature, outward unit
normal and unit tangent at $\widehat X(\xi,t)$.
Here $s=|\Gamma(t)|\xi$ is measured from $\widehat X(0,t)$, with
$ds=|\Gamma(t)|\,d\xi$ and $\partial_s=|\Gamma(t)|^{-1}\partial_\xi$.

The limiting BGN velocity is
$-\widehat\kappa\widehat{\mathbf n}+\widehat\alpha\widehat{\mathbf t}$,
where $\widehat\alpha(\cdot,t)\in H^1(\Torus)$ minimizes the velocity
Dirichlet energy \cite[(1.9) and Section~4.4]{BaiLi2025}:
\[
 \widehat\alpha=\operatorname*{argmin}_{\beta\in H^1(\Torus)}
 \frac12\int_{\Gamma(t)}
 \bigl|\partial_s(-\widehat\kappa\widehat{\mathbf n}
                 +\beta\widehat{\mathbf t})\bigr|^2\,ds.
\]
Tangential variation and \eqref{eq:orientation} give
\begin{equation}\label{eq:exact-tangential-velocity}
 -\partial_s^2\widehat\alpha+\widehat\kappa^2\widehat\alpha
 =-3\widehat\kappa\,\partial_s\widehat\kappa
 \qquad\text{for }\xi\in\Torus.
\end{equation}
Since $\widehat\kappa\not\equiv0$ on a regular closed curve, this periodic
problem uniquely determines $\widehat\alpha$. We fix $\widehat X$ by
\begin{equation}\label{eq:exact-marked-point}
 \partial_t\widehat X(0,t)
 =(-\widehat\kappa\widehat{\mathbf n}
     +\widehat\alpha\widehat{\mathbf t})(0,t),\quad
 \widehat X(0,0)=x^0(0),\quad
 |\partial_\xi\widehat X(\xi,t)|=|\Gamma(t)|.
\end{equation}
Time derivatives are taken at fixed $\xi$. To preserve normalized
arclength, $\partial_t\widehat X$ has the tangential component
\eqref{time:alphaH}, which agrees with $\widehat\alpha$ at the marked point.
Sections~\ref{time:intrinsic-section} and~\ref{up:sec:first-order}
justify this choice.
We abbreviate $\widehat X(t)=\widehat X(\cdot,t)$, and similarly for
the exact fields.

We impose the following standing assumption on the exact solution
for the convergence results.
\begin{assumption}[Exact flow]\label{ass:exact-flow}
The family $\Gamma(t)$ consists of regular embedded closed curves
evolving by curve-shortening flow on $[0,T]$, with $T$ before the
first singular time. Its counterclockwise normalized-arclength
parametrization $\widehat X$, fixed by \eqref{eq:exact-marked-point},
satisfies
\[
 \widehat X\in C^\infty(\Torus\times[0,T];\R^2),\qquad
 x^0=\widehat X(\cdot,0).
\]
\end{assumption}

The spatial error is $x_h^m-x^m$, the temporal error is
$\widehat x^m-\widehat X(t_m)$, and the matched total error is
$\widehat x_h^m-\widehat X(t_m)$.
Figure~\ref{fig:parametrizations} shows the corresponding maps.
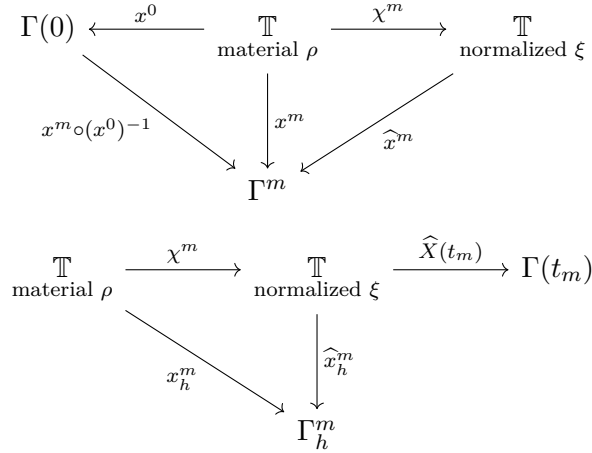
\begin{figure}[H]
\centering
\begin{tikzcd}[column sep=4em,row sep=3.3em]
 \Gamma(0)
   \arrow[dr,"x^m\circ(x^0)^{-1}"']
 & \underset{\text{material }\rho}{\Torus}
   \arrow[l,"x^0"'] \arrow[r,"\chi^m"] \arrow[d,"x^m"]
 & \underset{\text{normalized }\xi}{\Torus}
   \arrow[dl,"\widehat x^m"] \\
 & \Gamma^m &
\end{tikzcd}
\par\medskip
\begin{tikzcd}[column sep=4em,row sep=3.3em]
 \underset{\text{material }\rho}{\Torus}
   \arrow[r,"\chi^m"] \arrow[dr,"x_h^m"']
 & \underset{\text{normalized }\xi}{\Torus}
   \arrow[r,"\widehat X(t_m)"] \arrow[d,"\widehat x_h^m"]
 & \Gamma(t_m) \\
 & \Gamma_h^m &
\end{tikzcd}
\caption{Time-semidiscrete parametrizations and inherited material labels (top),
and the maps used for numerical--exact comparison (bottom).}
\label{fig:parametrizations}
\end{figure}

\subsection{Main result}

The main estimate combines temporal and spatial errors.

\begin{theorem}[Fully discrete error estimate]\label{un:thm:main}
Under Assumption~\ref{ass:exact-flow}, for every fixed $k\ge1$,
there exist positive constants
$c,C,\tau_0$, independent of $h,\tau,m$, such that if
\begin{equation}\label{un:eq:coupling}
 0<\tau=T/M\le\tau_0,\qquad 0<h\le c\tau^2,
\end{equation}
all original mixed BGN steps exist uniquely and their curves remain
regular and embedded. The matched error satisfies
\begin{equation}\label{un:eq:total}
 \max_{m\le M}
 \|\widehat x_h^m-\widehat X(t_m)\|_{\W}\le C(\tau+h^k).
\end{equation}
The same bound controls the matched $H^1$, $L^2$, and $L^\infty$
errors and the Hausdorff distance of the curves.
Constants may depend
on the prescribed smooth flow, $T$, the fixed degree, and the mesh family.
\end{theorem}

The proof is given in Section~\ref{sec:nonlinear}.
The mesh condition suffices for the derivative comparison and induction
argument. Decreasing $\tau_0$ enforces all fixed small-mesh thresholds.
\section{Analysis of the time-semidiscrete solution}\label{sec:time}

Under Assumption~\ref{ass:exact-flow}, we
prove existence and preliminary error estimates for the time-semidiscrete
solution of \eqref{time:bgn} on $[0,T]$ by controlling its curvature
and length in a low Sobolev norm. These bounds supply the higher
regularity and first-order convergence proof in Section~\ref{sec:upgrade}.

Write $\ell^m=|\Gamma^m|$ for the length of the time-semidiscrete
curve whenever $x^m$ is defined, $0\le m\le M$.

\begin{theorem}[Continuation and preliminary convergence]\label{time:theorem}
Suppose Assumption~\ref{ass:exact-flow} holds, and let
\[
 \frac32<\sigma<2.
\]
There exist $\tau_0>0$ and $C>0$, depending on the exact solution,
$T$, and $\sigma$, such that, for $0<\tau\le\tau_0$, every update
\eqref{time:bgn} with $(m+1)\tau\le T$ has a unique solution
$x^{m+1}\in C^\infty(\Torus;\R^2)$, and $x^{m+1}$ is a regular
embedding. 

For the curvature profiles and lengths defined in
Section~\ref{sec:reference-maps}, we have
\begin{equation}\label{time:curvature-convergence}
 \max_{m\tau\le T}
 \bigl(
   \|\widehat\kappa^m-\widehat\kappa(t_m)\|_{H^\sigma(\Torus)}
   +\bigl|\ell^m-|\Gamma(t_m)|\bigr|
 \bigr)
 \le C\tau^{(\sigma-1)/2}.
\end{equation}
\end{theorem}

The proof in Section~\ref{time:continuation-section} proceeds by
induction on the time index, using the velocity and curvature bounds
established in Lemma~\ref{time:local-estimates} and its proof,
the defect bounds in Lemma~\ref{time:intermediate-equation}, and the
stability estimate in Lemma~\ref{time:prefix-error}.

\subsection{Velocity representation and geometric identities}\label{time:intrinsic-section}

We introduce the scalar operators through the time-semidiscrete velocity equations,
then derive the curvature and length identities used in
Section~\ref{time:local-estimates-section}.

For each available length $\ell^m$, $0\le m\le M$, and each nonzero
scalar profile $f\in H^\sigma(\Torus;\R)$, define
\begin{align}
 \Lop_f^m&=-(\ell^m)^{-2}\partial_\xi^2+f^2,
 &\Dop_f^m&=(\ell^m)^{-1}(\partial_\xi f+2f\partial_\xi),
 \label{time:LD}\\
 \Kop_f^m&=\Dop_f^m(\Lop_f^m)^{-1}\Dop_f^m,
 &\Sop_f^m&=\Lop_f^m+\Kop_f^m.
 \label{time:KS}
\end{align}
The subscript specifies the coefficient profile, and the superscript
specifies the length $\ell^m$. In $\Dop_f^m$, $\partial_\xi f$ acts by
multiplication. These operators arise from the normal and tangential
components of the BGN equation. Eliminating the tangential component
gives the following representation. The required invertibility and
operator bounds are proved in Lemma~\ref{time:operator-bounds} below.

\begin{lemma}\label{time:exact-step}
Let $x^m\in C^\infty(\Torus;\R^2)$ be a regular closed parametrization
and let $\tau>0$. Then \eqref{time:bgn} has a unique
position solution $x^{m+1}\in C^\infty(\Torus;\R^2)$.
Its velocity components satisfy
\begin{equation}\label{time:normal-resolvent}
 \widehat v_{\mathrm n}^m=-(I+\tau\Sop_{\widehat\kappa^m}^m)^{-1}\widehat\kappa^m,\quad \widehat v_{\mathrm t}^m=(\Lop_{\widehat\kappa^m}^m)^{-1}\Dop_{\widehat\kappa^m}^m \widehat v_{\mathrm n}^m.
\end{equation}
\end{lemma}

\begin{proof}
The strong form of \eqref{time:bgn} is
\begin{equation}\label{time:strong-bgn}
 \partial_s^2x^{m+1}=v_{\mathrm n}^m\mathbf n^m.
\end{equation}
Substitute $x^{m+1}=x^m+\tau v^m$ and use \eqref{eq:orientation}.
The normal and tangential components satisfy
\[
 (I+\tau\Lop_{\widehat\kappa^m}^m)\widehat v_{\mathrm n}^m
   =-\widehat\kappa^m-\tau\Dop_{\widehat\kappa^m}^m \widehat v_{\mathrm t}^m,\quad \Lop_{\widehat\kappa^m}^m \widehat v_{\mathrm t}^m
   =\Dop_{\widehat\kappa^m}^m \widehat v_{\mathrm n}^m.
\]
The curvature of a regular closed curve is not identically zero.
Lemma~\ref{time:operator-bounds} below therefore gives invertibility
of $\Lop_{\widehat\kappa^m}^m$ and $I+\tau\Sop_{\widehat\kappa^m}^m$.
Elimination gives \eqref{time:normal-resolvent}. Conversely, pulling
these components back by $\chi^m$ gives a material velocity $v^m$
for which $x^{m+1}=x^m+\tau v^m$ satisfies
\eqref{time:strong-bgn}, hence \eqref{time:bgn}.
The velocity representation also gives uniqueness.
To verify the claimed spatial regularity, write
\eqref{time:strong-bgn} as
\[
 \partial_\rho^2x^{m+1}
 =\frac{\partial_\rho|\partial_\rho x^m|}{|\partial_\rho x^m|}
       \partial_\rho x^{m+1}
  +\frac{|\partial_\rho x^m|^2}{\tau}
       \bigl((x^{m+1}-x^m)\cdot\mathbf n^m\bigr)\mathbf n^m.
\]
Since $x^m$ is $C^\infty$ and regular, all coefficients are smooth.
Starting from $x^{m+1}\in H^1$, the right-hand side belongs to $L^2$,
so periodic elliptic regularity gives $x^{m+1}\in H^2$.
More generally, for every integer $r\ge1$,
\[
 x^{m+1}\in H^r
 \quad\Longrightarrow\quad
 \partial_\rho^2x^{m+1}\in H^{r-1}
 \quad\Longrightarrow\quad
 x^{m+1}\in H^{r+1}.
\]
Iteration and Sobolev embedding give
$x^{m+1}\in C^\infty(\Torus;\R^2)$.
This argument holds for each fixed $\tau>0$, with regularity
constants that may depend on $\tau$ and $x^m$.
\end{proof}

The following identities express the curvature and length after one step.
Differentiating the velocity decomposition gives
\begin{equation}\label{time:ABCq}
 \partial_s v^m
 =(\partial_s v_{\mathrm t}^m+\kappa^m v_{\mathrm n}^m)
        \mathbf t^m+(\partial_s v_{\mathrm n}^m-\kappa^m v_{\mathrm t}^m)
        \mathbf n^m.
\end{equation}
The position increment and \eqref{time:strong-bgn} imply
\begin{equation}\label{time:strong-identities}
\begin{aligned}
 \partial_s x^{m+1}&=\mathbf t^m+\tau\partial_s v^m,\\
 \partial_s(\mathbf t^m\cdot\partial_s v^m)
   &=-\kappa^m(\mathbf n^m\cdot\partial_s v^m),\\
 \tau\partial_s(\mathbf n^m\cdot\partial_s v^m)
   &=v_{\mathrm n}^m+\kappa^m
        (1+\tau\mathbf t^m\cdot\partial_s v^m).
\end{aligned}
\end{equation}
If $|\mathbf t^m+\tau\partial_s v^m|>0$ everywhere, then $x^{m+1}$
is regular and its relative metric gives
\begin{equation}\label{time:arclength-change}
\begin{aligned}
 \ell^{m+1}&=\int_0^1|(x^m)'|
              |\mathbf t^m+\tau\partial_s v^m|\,d\rho,\\
 \chi^{m+1}(\rho)&=\frac{1}{\ell^{m+1}}
    \int_0^\rho|(x^m)'(\zeta)|
              |\mathbf t^m+\tau\partial_s v^m|(\zeta)\,d\zeta.
\end{aligned}
\end{equation}
The updated normal is
\[
 \mathbf n^{m+1}
 =\frac{J[\mathbf t^m+\tau\partial_s v^m]}
        {|\mathbf t^m+\tau\partial_s v^m|}.
\]
Taking its component in \eqref{time:strong-bgn} and dividing by the
square of the relative metric yields
\begin{equation}\label{time:exact-curvature}
 \kappa^{m+1}
 =-v_{\mathrm n}^m
       \frac{1+\tau\mathbf t^m\cdot\partial_s v^m}
            {|\mathbf t^m+\tau\partial_s v^m|^3}.
\end{equation}
If $1+\tau\mathbf t^m\cdot\partial_s v^m>0$, the identity
\begin{equation}\label{time:q-expansion-exact}
 |\mathbf t^m+\tau\partial_s v^m|
 =1+\tau\mathbf t^m\cdot\partial_s v^m+\frac{\tau^2(\mathbf n^m\cdot\partial_s v^m)^2}
 {|\mathbf t^m+\tau\partial_s v^m|
       +1+\tau\mathbf t^m\cdot\partial_s v^m}
\end{equation}
and periodicity give
\begin{equation}\label{time:exact-length}
 \ell^{m+1}=\ell^m
 +\tau\int_0^1\kappa^m v_{\mathrm n}^m|(x^m)'|\,d\rho+\tau^2\int_0^1
 \frac{(\mathbf n^m\cdot\partial_s v^m)^2|(x^m)'|}
      {|\mathbf t^m+\tau\partial_s v^m|
          +1+\tau\mathbf t^m\cdot\partial_s v^m}\,d\rho.
\end{equation}

\subsection{Uniform bounds for the velocity operators}\label{time:operators-section}

The velocity representation requires invertible operators for each
fixed curve. The error estimates also require uniform bounds when
the coefficient profile and length vary. Throughout this subsection,
functions use $\xi\in\Torus$, norms and pairings use $d\xi$, and
$H^{-1}$ is the dual of $H^1$.

For a real coefficient profile
$f\in H^\sigma(\Torus;\R)\setminus\{0\}$ and a length $\ell>0$,
let $\Lop_{f,\ell},\Dop_{f,\ell},\Kop_{f,\ell},\Sop_{f,\ell}$
be the operators in \eqref{time:LD}--\eqref{time:KS} with $\ell^m$
replaced by $\ell$. Thus $\Lop_f^m=\Lop_{f,\ell^m}$, and similarly
for the other operators. The BGN velocity formula uses
$(f,\ell)=(\widehat\kappa^m,\ell^m)$, while comparison with the exact
flow uses $(\widehat\kappa(t),|\Gamma(t)|)$.
We also apply these operators to intermediate scalar states, so no
curve closure condition is imposed on $f$.

\paragraph{The set for uniform estimates.}
Fix positive constants $\ell_{\min},\ell_{\max},c_0$ such that
\[
\begin{gathered}
 0<\ell_{\min}<\min_{0\le t\le T}|\Gamma(t)|,\quad
 \ell_{\max}>\max_{0\le t\le T}|\Gamma(t)|,\quad 0<c_0<\min_{0\le t\le T}\|\widehat\kappa(t)\|_{L^2}.
\end{gathered}
\]
The choice of $c_0$ is possible because
$\int_{\Gamma(t)}\widehat\kappa\,ds=2\pi$ and the exact lengths are bounded.
For $R>0$, define
\begin{equation}\label{time:local-set}
 \mathcal U(R)=\bigl\{(f,\ell)\in
 H^\sigma(\Torus;\R)\times[\ell_{\min},\ell_{\max}]:
 \|f\|_{H^\sigma}\le R,\ \|f\|_{L^2}\ge c_0\bigr\}.
\end{equation}
On this set, $C_R$ denotes a constant depending only on
$R,\sigma,\ell_{\min},\ell_{\max},c_0$, independent of the particular
pair $(f,\ell)$. The length bounds control the principal coefficient,
and the lower bound on $\|f\|_{L^2}$ gives uniform coercivity of
$\Lop_{f,\ell}$, as shown below.
Fix $R>\max_{0\le t\le T}\|\widehat\kappa(t)\|_{H^\sigma}$ for the
continuation argument. The exact curvature--length pair satisfies
$(\widehat\kappa(t),|\Gamma(t)|)\in\operatorname{int}\mathcal U(R)$
for $0\le t\le T$. For $R_{\mathrm{ext}}>R$,
$\mathcal U(R)\subset\mathcal U(R_{\mathrm{ext}})$, with the same
constants $\ell_{\min},\ell_{\max},c_0$.

The following lemma supplies solvability for Lemma~\ref{time:exact-step},
spectral estimates for Lemma~\ref{time:local-estimates}, and coefficient
comparison for Lemma~\ref{time:intermediate-equation}.
\begin{lemma}\label{time:operator-bounds}
For each $f\in H^\sigma(\Torus;\R)\setminus\{0\}$ and $\ell>0$,
$\Lop_{f,\ell}:H^1\to H^{-1}$ is coercive and
\begin{equation}\label{time:D-bounds}
 \Dop_{f,\ell}:L^2\to H^{-1},
 \qquad
 \Dop_{f,\ell}:H^1\to L^2
\end{equation}
are bounded. On $L^2$, $\Kop_{f,\ell}$ is bounded, self-adjoint and
nonpositive, and $\Sop_{f,\ell}$ is self-adjoint and nonnegative. Also,
\begin{equation}\label{time:S-domain}
 \operatorname{Dom}(\Sop_{f,\ell})=H^2,\qquad
 \|g\|_{H^2}\asymp
 \|g\|_{L^2}+\|\Sop_{f,\ell}g\|_{L^2}
\end{equation}
with constants depending on $f$ and $\ell$.
In particular, $I+\tau\Sop_{f,\ell}$ is invertible for every $\tau>0$.

The coercivity, operator bounds and norm equivalence above are uniform
for $(f,\ell)\in\mathcal U(R)$. For any two pairs
$(f,\ell),(\widetilde f,\widetilde\ell)\in\mathcal U(R)$,
\begin{equation}\label{time:K-lipschitz}
 \|\Kop_{f,\ell}-\Kop_{\widetilde f,\widetilde\ell}\|_{\mathcal L(L^2,L^2)}
 \le C_R\bigl(\|f-\widetilde f\|_{H^1}
                    +|\ell-\widetilde\ell|\bigr).
\end{equation}
\end{lemma}

\begin{proof}
For $b\in H^1$, integration by parts gives
\[
 \langle\Lop_{f,\ell}b,b\rangle
 =\ell^{-2}\|\partial_\xi b\|_{L^2}^2+\|f b\|_{L^2}^2.
\]
Writing $\bar b=\int_0^1b\,d\xi$, periodic Poincar\'e gives
$\|b-\bar b\|_{L^2}\le C\|\partial_\xi b\|_{L^2}$, while
\[
 |\bar b|\,\|f\|_{L^2}
 \le \|f b\|_{L^2}
     +\|f\|_{L^\infty}\|b-\bar b\|_{L^2}.
\]
Since $\|f\|_{L^2}>0$ and $\ell>0$, these inequalities control
$\|b\|_{H^1}^2$ by $\langle\Lop_{f,\ell}b,b\rangle$.
The coercivity constant is uniform on $\mathcal U(R)$ because
$\ell$, $\|f\|_{L^\infty}$ and $\|f\|_{L^2}^{-1}$ are uniformly bounded.

For $g\in L^2$ and $w\in H^1$, integration by parts yields
\[
 |\langle\Dop_{f,\ell}g,w\rangle|
 =
 \ell^{-1}\left|
 \int_0^1 g\bigl(-2f\partial_\xi w-(\partial_\xi f)w\bigr)\,d\xi
 \right|
 \le C_{f,\ell}\|g\|_{L^2}\|w\|_{H^1}.
\]
The embedding $H^1\hookrightarrow L^\infty$ also gives the second
bound in \eqref{time:D-bounds}. Thus
\[
 \Kop_{f,\ell}:
 L^2\xrightarrow{\Dop_{f,\ell}}H^{-1}
 \xrightarrow{\Lop_{f,\ell}^{-1}}H^1
 \xrightarrow{\Dop_{f,\ell}}L^2
\]
is bounded. Since $\Dop_{f,\ell}^*=-\Dop_{f,\ell}$,
$\Kop_{f,\ell}=\Dop_{f,\ell}\Lop_{f,\ell}^{-1}\Dop_{f,\ell}$
is self-adjoint and nonpositive.

For smooth $g$, a further integration by parts gives
\begin{equation}\label{time:S-form}
\begin{aligned}
 (\Sop_{f,\ell}g,g)_{L^2}
 &=\int_0^1
   \bigl(\ell^{-1}\partial_\xi(\Lop_{f,\ell}^{-1}\Dop_{f,\ell}g)+fg\bigr)^2\,d\xi+\int_0^1
   \bigl(\ell^{-1}\partial_\xi g-f\Lop_{f,\ell}^{-1}\Dop_{f,\ell}g\bigr)^2\,d\xi
 \ge0.
\end{aligned}
\end{equation}
Since $\Sop_{f,\ell}=\Lop_{f,\ell}+\Kop_{f,\ell}$ is a bounded
self-adjoint perturbation of $\Lop_{f,\ell}$ on $L^2$, it is
self-adjoint with domain $H^2$. Together with \eqref{time:S-form},
this proves that $I+\tau\Sop_{f,\ell}$ is invertible for $\tau>0$.
The equation
\[
 -\ell^{-2}\partial_\xi^2 g
 =\Sop_{f,\ell} g-f^2g-\Kop_{f,\ell} g
\]
and the boundedness of $\Kop_{f,\ell}$ prove \eqref{time:S-domain}.
On $\mathcal U(R)$, the bounds for $\ell$, $\ell^{-1}$,
$\|f\|_{H^\sigma}$, and $\|f\|_{L^2}^{-1}$ make all these
constants uniform.

For two pairs in $\mathcal U(R)$, expand the difference of the three
factors in $\Kop_{f,\ell}=\Dop_{f,\ell}\Lop_{f,\ell}^{-1}\Dop_{f,\ell}$.
For the inverse factors use
\[
 \Lop_{f,\ell}^{-1}
 -\Lop_{\widetilde f,\widetilde\ell}^{-1}
 =
 \Lop_{f,\ell}^{-1}
 (\Lop_{\widetilde f,\widetilde\ell}-\Lop_{f,\ell})
 \Lop_{\widetilde f,\widetilde\ell}^{-1}
\]
between $H^{-1}$ and $H^1$.
The difference $\Dop_{f,\ell}-\Dop_{\widetilde f,\widetilde\ell}$
is bounded in the two norms in
\eqref{time:D-bounds} by
$C_R(\|f-\widetilde f\|_{H^1}+|\ell-\widetilde\ell|)$.
The difference $\Lop_{f,\ell}-\Lop_{\widetilde f,\widetilde\ell}$
has the same bound as an operator from $H^1$ to $H^{-1}$, by its
formula in \eqref{time:LD}.
Combining these estimates proves \eqref{time:K-lipschitz}.
\end{proof}

\subsection{Consistency estimates for the normal velocity and curvature}\label{time:local-estimates-section}

For the time-semidiscrete solution, the following lemma estimates
$\widehat v_{\mathrm n}^m+\widehat\kappa^m$ and
$\widehat v_{\mathrm n}^m+\widehat\kappa^{m+1}$ in $H^r(\Torus)$.
These bounds are used to estimate the defects and prove continuation.

\begin{lemma}\label{time:local-estimates}
For $(\widehat\kappa^m,\ell^m)\in\mathcal U(R)$ and sufficiently small $\tau$,
\begin{gather}
 \|\widehat\kappa^m+\widehat v_{\mathrm n}^m\|_{H^r}
   \le C_R\tau^{(\sigma-r)/2},\qquad 0\le r\le\sigma,
 \label{time:resolvent-estimates}\\
 \|\widehat\kappa^{m+1}+\widehat v_{\mathrm n}^m\|_{H^r}
 \le C_R\tau^{(\sigma+1-r)/2},\qquad 1\le r\le\sigma.
\label{time:curvature-increment}
\end{gather}
\end{lemma}

\begin{proof}
By Lemma~\ref{time:operator-bounds} and interpolation,
$\|(I+\Sop_{\widehat\kappa^m}^m)^{r/2}f\|_{L^2}$ is uniformly equivalent
to $\|f\|_{H^r}$ for $0\le r\le2$.
The resolvent representation \eqref{time:normal-resolvent} gives
\[
 \widehat\kappa^m+\widehat v_{\mathrm n}^m
 =\tau\Sop_{\widehat\kappa^m}^m
       (I+\tau\Sop_{\widehat\kappa^m}^m)^{-1}\widehat\kappa^m.
\]
For $0<\tau\le1$, the scalar bounds
\[
\begin{aligned}
 \sup_{\lambda\ge0}
 \frac{\tau\lambda(1+\lambda)^{(r-\sigma)/2}}{1+\tau\lambda}
 &\le C\tau^{(\sigma-r)/2}, &&0\le r\le\sigma,\\
 \sup_{\lambda\ge0}
 \frac{(1+\lambda)^{(r-\sigma)/2}}{1+\tau\lambda}
 &\le C\tau^{-(r-\sigma)/2}, &&\sigma\le r\le2
\end{aligned}
\]
and spectral calculus prove \eqref{time:resolvent-estimates} and yield
\[
 \|\widehat v_{\mathrm n}^m\|_{H^\sigma}\le C_R,\qquad
 \|\widehat v_{\mathrm n}^m\|_{H^2}\le C_R\tau^{-(2-\sigma)/2}.
\]

Since $H^{\sigma-1}$ is an algebra, the formula for
$\Dop_{\widehat\kappa^m}^m$ gives
\[
 \|\Dop_{\widehat\kappa^m}^m\widehat v_{\mathrm n}^m\|_{H^{\sigma-1}}
 \le C_R\|\widehat v_{\mathrm n}^m\|_{H^\sigma}\le C_R.
\]
Uniform coercivity of $\Lop_{\widehat\kappa^m}^m$ and periodic elliptic
regularity, applied to the tangential equation in
\eqref{time:normal-resolvent}, therefore imply
\[
 \|\widehat v_{\mathrm t}^m\|_{H^{\sigma+1}}
 \le C_R\|\Dop_{\widehat\kappa^m}^m\widehat v_{\mathrm n}^m\|_{H^{\sigma-1}}
 \le C_R.
\]
Applying $\Dop_{\widehat\kappa^m}^m$ once more bounds
$\Kop_{\widehat\kappa^m}^m\widehat v_{\mathrm n}^m
=\Dop_{\widehat\kappa^m}^m\widehat v_{\mathrm t}^m$ in $H^{\sigma-1}$.
Write the resolvent equation as
\[
 (I-\tau(\ell^m)^{-2}\partial_\xi^2)\widehat v_{\mathrm n}^m=-\widehat\kappa^m-\tau\bigl[(\widehat\kappa^m)^2\widehat v_{\mathrm n}^m
                                +\Kop_{\widehat\kappa^m}^m \widehat v_{\mathrm n}^m\bigr].
\]
The Fourier multiplier
$[1+\tau(\ell^m)^{-2}(2\pi n)^2]^{-1}$, $n\in\Z$, gives
\[
\begin{aligned}
 \|(I-\tau(\ell^m)^{-2}\partial_\xi^2)^{-1}f\|_{H^{\sigma+1}}
 &\le C_R\tau^{-1/2}\|f\|_{H^\sigma},\\
 \|\tau(I-\tau(\ell^m)^{-2}\partial_\xi^2)^{-1}f\|_{H^{\sigma+1}}
 &\le C_R\|f\|_{H^{\sigma-1}}.
\end{aligned}
\]
Apply the first bound to $\widehat\kappa^m$ and the second to the
bracketed term, which is uniformly bounded in $H^{\sigma-1}$.
Consequently,
\begin{equation}\label{time:additional-derivative}
 \|\widehat v_{\mathrm t}^m\|_{H^{\sigma+1}}
       +\|\Kop_{\widehat\kappa^m}^m \widehat v_{\mathrm n}^m\|_{H^{\sigma-1}}\le C_R,\qquad
 \|\widehat v_{\mathrm n}^m\|_{H^{\sigma+1}}\le C_R\tau^{-1/2}.
\end{equation}

The velocity bounds also control the change of arclength.
Using \eqref{time:ABCq} and the velocity bounds gives
\[
\begin{aligned}
 \|\widehat{\mathbf t}^m\cdot\partial_s\widehat v^m\|_{H^\sigma}
 +\|\widehat{\mathbf n}^m\cdot\partial_s\widehat v^m\|_{H^{\sigma-1}}
 &\le C_R,\\
 \|\widehat{\mathbf n}^m\cdot\partial_s\widehat v^m\|_{H^\sigma}
 &\le C_R\tau^{-1/2}.
\end{aligned}
\]
Both are uniformly bounded pointwise because $\sigma-1>1/2$.
The product estimate gives
\[
 \|\tau^2(\widehat{\mathbf n}^m\cdot\partial_s\widehat v^m)^2\|_{H^\sigma}\le C\tau^2\|\widehat{\mathbf n}^m\cdot\partial_s\widehat v^m\|_{L^\infty}
              \|\widehat{\mathbf n}^m\cdot\partial_s\widehat v^m\|_{H^\sigma}
 \le C_R\tau^{3/2}.
\]
The pointwise bound also gives
$1+\tau\widehat{\mathbf t}^m\cdot\partial_s\widehat v^m
\ge1-C_R\tau\ge1/2$ for small $\tau$.
Thus $x^{m+1}$ is regular. The identity
\[
 |\widehat{\mathbf t}^m+\tau\partial_s\widehat v^m|^2
 =(1+\tau\widehat{\mathbf t}^m\cdot\partial_s\widehat v^m)^2
   +\tau^2(\widehat{\mathbf n}^m\cdot\partial_s\widehat v^m)^2,
\]
implies
\[
 \bigl\||\widehat{\mathbf t}^m+\tau\partial_s\widehat v^m|^2-1\bigr\|_{H^\sigma}
 \le C_R\tau.
\]
Smooth composition with square-root and reciprocal functions gives
the first two bounds in \eqref{time:metric-estimates} below.
For the coordinate change, \eqref{time:arclength-change} gives
\[
 \partial_\xi\bigl(\chi^{m+1}\circ(\chi^m)^{-1}\bigr)
 =\frac{|\widehat{\mathbf t}^m+\tau\partial_s\widehat v^m|}
 {\displaystyle\int_0^1|\widehat{\mathbf t}^m+\tau\partial_s\widehat v^m|\,d\xi}.
\]
Its derivative differs from $1$ by $O_{H^\sigma}(\tau)$.
Since this map fixes $0$, Poincar\'e's inequality gives its difference
from $\id$ in $H^{\sigma+1}$. Together, these estimates yield
\begin{equation}\label{time:metric-estimates}
 \bigl\||\widehat{\mathbf t}^m+\tau\partial_s\widehat v^m|-1\bigr\|_{H^\sigma}
 +\left\|\frac{1+\tau\widehat{\mathbf t}^m\cdot\partial_s\widehat v^m}
 {|\widehat{\mathbf t}^m+\tau\partial_s\widehat v^m|^3}-1\right\|_{H^\sigma}+\|\chi^{m+1}\circ(\chi^m)^{-1}-\id\|_{H^{\sigma+1}}
 \le C_R\tau.
\end{equation}

To compare $\widehat\kappa^{m+1}$ with $-\widehat v_{\mathrm n}^m$ at
the same $\xi$, we need the inverse coordinate estimate
\[
 \|\chi^m\circ(\chi^{m+1})^{-1}-\id\|_{H^{\sigma+1}}
 \le C_R\tau.
\]
Indeed, \eqref{time:metric-estimates} and
$H^{\sigma+1}\hookrightarrow W^{2,\infty}$ give uniform
$W^{2,\infty}$ bounds and positive lower derivative bounds for the
coordinate change and its inverse. The chain rule and change of
variables give uniform composition bounds on $H^1$ and $H^2$, hence
on $H^\sigma$ by interpolation. Applying these bounds to the inverse
derivative identity, and using that the inverse fixes $0$, proves
the displayed estimate. We use the lifts fixing $0$, whose differences
from $\id$ are periodic.

The fundamental theorem of calculus gives
\[
\begin{aligned}
 &\widehat v_{\mathrm n}^m\circ\chi^m\circ(\chi^{m+1})^{-1}
       -\widehat v_{\mathrm n}^m\\
 &\quad=\bigl(\chi^m\circ(\chi^{m+1})^{-1}-\id\bigr)
 \int_0^1(\partial_\xi\widehat v_{\mathrm n}^m)
 \circ\bigl[\id+\varepsilon
       \bigl(\chi^m\circ(\chi^{m+1})^{-1}-\id\bigr)\bigr]\,d\varepsilon.
\end{aligned}
\]
For $0\le\varepsilon\le1$, the maps in this integral have the same
uniform derivative bounds. Multiplication in $H^r$ and the composition
bounds therefore give, for $r\in\{1,\sigma\}$,
\[
 \|\widehat v_{\mathrm n}^m\circ\chi^m\circ(\chi^{m+1})^{-1}-\widehat v_{\mathrm n}^m\|_{H^r}
 \le C_R\tau\|\widehat v_{\mathrm n}^m\|_{H^{r+1}}.
\]
On the other hand, \eqref{time:exact-curvature} gives the exact identity
\[
 \widehat\kappa^{m+1}\circ\chi^{m+1}\circ(\chi^m)^{-1}
       +\widehat v_{\mathrm n}^m=-\widehat v_{\mathrm n}^m
 \left(\frac{1+\tau\widehat{\mathbf t}^m\cdot\partial_s\widehat v^m}
 {|\widehat{\mathbf t}^m+\tau\partial_s\widehat v^m|^3}-1\right).
\]
By \eqref{time:metric-estimates}, its right-hand side has
$H^\sigma$ norm at most $C_R\tau$. Compose this identity with
$\chi^m\circ(\chi^{m+1})^{-1}$ and use the preceding velocity
comparison to obtain
\[
 \|\widehat\kappa^{m+1}+\widehat v_{\mathrm n}^m\|_{H^r}
 \le C_R\tau\bigl(1+\|\widehat v_{\mathrm n}^m\|_{H^{r+1}}\bigr),
 \qquad r\in\{1,\sigma\}.
\]
The $H^2$ velocity bound and \eqref{time:additional-derivative} now give
\[
 \|\widehat\kappa^{m+1}+\widehat v_{\mathrm n}^m\|_{H^1}
 \le C_R\tau^{\sigma/2},\qquad
 \|\widehat\kappa^{m+1}+\widehat v_{\mathrm n}^m\|_{H^\sigma}
 \le C_R\tau^{1/2}.
\]
Interpolating between these two estimates gives the exponent
$(\sigma+1-r)/2$ in \eqref{time:curvature-increment} for
$1\le r\le\sigma$.
\end{proof}

\subsection{Estimates for the defect terms}\label{time:intermediate-section}

We first write the exact curvature and length equations, then define
and estimate the defects of their
time-semidiscrete counterparts for use in
Section~\ref{time:stability-section}.

\paragraph{Exact evolution equations.}
For the exact parametrization fixed by \eqref{eq:exact-marked-point},
the curvature evolution identity \cite[Lemma~39(i)]{BGN2020},
written in our curvature sign convention, gives
\begin{equation}\label{time:reference-equations}
 \partial_t\widehat\kappa
 =|\Gamma(t)|^{-2}\partial_\xi^2\widehat\kappa
     +\mathcal N(\widehat\kappa(t),|\Gamma(t)|),
\end{equation}
where
\[
 \mathcal N(\widehat\kappa(t),|\Gamma(t)|)
 =\underbrace{\widehat\kappa^3}_{\text{curvature reaction}}
 +\underbrace{\frac{\partial_t\widehat X\cdot\widehat{\mathbf t}}
 {|\Gamma(t)|}\,\partial_\xi\widehat\kappa}_{\text{tangential transport}}.
\]
The length satisfies
\begin{equation}\label{time:g0}
 \frac{d}{dt}|\Gamma(t)|
   =-|\Gamma(t)|\int_0^1\widehat\kappa(\xi,t)^2\,d\xi.
\end{equation}
To express the nonlinear term through curvature and length, use
\eqref{eq:exact-tangential-velocity} and
$\Dop_{f,\ell}f=3\ell^{-1}f\partial_\xi f$ to obtain
\begin{equation}\label{time:alphaM}
 \widehat\alpha(\xi,t)
 =-\bigl(\Lop_{\widehat\kappa(t),|\Gamma(t)|}^{-1}
         \Dop_{\widehat\kappa(t),|\Gamma(t)|}
         \widehat\kappa(t)\bigr)(\xi).
\end{equation}
The tangential velocity of the normalized-arclength parametrization is
\begin{equation}\label{time:alphaH}
 \partial_t\widehat X(\xi,t)\cdot\widehat{\mathbf t}(\xi,t)
 =\widehat\alpha(0,t)+|\Gamma(t)|\int_0^\xi
       \biggl(\widehat\kappa(\zeta,t)^2-
              \int_0^1\widehat\kappa(\eta,t)^2\,d\eta\biggr)d\zeta.
\end{equation}
To compare the time-semidiscrete pair
$(\widehat v_{\mathrm n}^m,\ell^m)$ with the exact pair
$(-\widehat\kappa(t_m),|\Gamma(t_m)|)$, we use the following nonlinear
map for nonzero $f\in H^\sigma(\Torus;\R)$ and $\ell>0$,
with estimates on each fixed set $\mathcal U(R)$:
\begin{equation}\label{time:N}
\begin{aligned}
 \mathcal N(f,\ell)(\xi)
 ={}&\underbrace{f(\xi)^3}_{\text{curvature reaction}}\\
 &+\underbrace{\frac{
  -(\Lop_{f,\ell}^{-1}\Dop_{f,\ell}f)(0)
  +\ell\displaystyle\int_0^\xi\biggl(f(\zeta)^2-
       \int_0^1f(\eta)^2\,d\eta\biggr)d\zeta}{\ell}
       \,\partial_\xi f(\xi)}_{\text{tangential transport}}.
\end{aligned}
\end{equation}
For $f=\widehat\kappa(t)$ and $\ell=|\Gamma(t)|$, the numerator in
the transport term equals $\partial_t\widehat X\cdot\widehat{\mathbf t}$
by \eqref{time:alphaM}--\eqref{time:alphaH}.
The operators $\Lop_{f,\ell},\Kop_{f,\ell},\Sop_{f,\ell}$ are even in $f$, whereas $\Dop_{f,\ell}$ is
odd. Consequently $\mathcal N(-f,\ell)=-\mathcal N(f,\ell)$.
Thus the same nonlinear map applies to the equation for
$-\widehat\kappa$.

\paragraph{Time-semidiscrete counterpart.}
For the final recurrence, define $\widehat v_{\mathrm n}^M$ by the first
identity in \eqref{time:normal-resolvent} with $m=M$.
For each available step $0\le m<M$, define the normal-velocity defect
$d_{\mathrm n}^m\in L^2(\Torus)$ and the length defect
$d_\ell^m\in\R$ by
\begin{equation}\label{time:scalar-diffusion}
\begin{aligned}
 \frac{\widehat v_{\mathrm n}^{m+1}-\widehat v_{\mathrm n}^m}{\tau}
       -(\ell^m)^{-2}\partial_\xi^2\widehat v_{\mathrm n}^{m+1}
 &=\mathcal N(\widehat v_{\mathrm n}^m,\ell^m)+d_{\mathrm n}^m,\\
 \frac{\ell^{m+1}-\ell^m}{\tau}
 &=-\ell^m\int_0^1(\widehat v_{\mathrm n}^m)^2\,d\xi+d_\ell^m.
\end{aligned}
\end{equation}

\begin{lemma}\label{time:intermediate-equation}
Suppose $(\widehat\kappa^j,\ell^j)\in\mathcal U(R)$ for
$0\le j\le m_*$, where $1\le m_*\le M$.
For sufficiently small $\tau$, the defects in
\eqref{time:scalar-diffusion} satisfy, for $0\le m<m_*$,
\begin{equation}\label{time:defect-bound}
 \|d_{\mathrm n}^m\|_{L^2}+|d_\ell^m|
 \le C_R\tau^{(\sigma-1)/2}.
\end{equation}
\end{lemma}

\begin{proof}
We first expand the change of normalized arclength to express the
curvature update at a common coordinate $\xi$.
The pointwise derivative bounds from the proof of Lemma~\ref{time:local-estimates}
and \eqref{time:q-expansion-exact} give
\begin{equation}\label{time:eta}
 (\chi^{m+1}\circ(\chi^m)^{-1})(\xi)
 =\xi+\tau\int_0^\xi\biggl[
  (\widehat{\mathbf t}^m\cdot\partial_s\widehat v^m)(\zeta)-\int_0^1(\widehat{\mathbf t}^m\cdot\partial_s\widehat v^m)(\eta)\,d\eta
 \biggr]d\zeta+O_{L^\infty}(\tau^2).
\end{equation}
The inverse map has the same first-order expansion with the opposite
sign, and
\[
 \frac{1+\tau\widehat{\mathbf t}^m\cdot\partial_s\widehat v^m}
 {|\widehat{\mathbf t}^m+\tau\partial_s\widehat v^m|^3}
 =1-2\tau\widehat{\mathbf t}^m\cdot\partial_s\widehat v^m+O_{L^\infty}(\tau^2).
\]
For $0\le\vartheta\le1$, the degree-one lifts
$(1-\vartheta)\id+\vartheta\chi^m\circ(\chi^{m+1})^{-1}$
have derivative $1+O(\tau)$ uniformly. Change of variables therefore
bounds the $L^2$ norm of $\partial_\xi^2\widehat v_{\mathrm n}^m$ composed
with these maps by $C\|\widehat v_{\mathrm n}^m\|_{H^2}$.
The $L^2$ Taylor remainder for the composition is bounded by
\[
 C\|\chi^m\circ(\chi^{m+1})^{-1}-\id\|_{L^\infty}^2\|\widehat v_{\mathrm n}^m\|_{H^2}.
\]
Replacing the inverse-coordinate increment by the negative first-order
integral in \eqref{time:eta} adds at most
$C_R\tau^2\|\widehat v_{\mathrm n}^m\|_{H^1}$.
Finally, $\widehat{\mathbf t}^m\cdot\partial_s\widehat v^m\in H^\sigma\hookrightarrow C^1$
implies that its composition with the inverse map differs from itself
by $O(\tau)$ in $L^\infty$.
Substitution into \eqref{time:exact-curvature} gives a geometric defect
$d_{\mathrm{geom}}^m\in L^2(\Torus)$ such that
\begin{equation}\label{time:curvature-expansion}
\begin{gathered}
 \begin{aligned}
 \widehat\kappa^{m+1}
 &=-\widehat v_{\mathrm n}^m+2\tau(\widehat{\mathbf t}^m\cdot\partial_s\widehat v^m)\widehat v_{\mathrm n}^m\\
 &\quad+\tau\biggl[\int_0^\xi\biggl(
    (\widehat{\mathbf t}^m\cdot\partial_s\widehat v^m)(\zeta)
    -\int_0^1(\widehat{\mathbf t}^m\cdot\partial_s\widehat v^m)(\eta)\,d\eta
    \biggr)d\zeta\biggr]\partial_\xi\widehat v_{\mathrm n}^m+\tau d_{\mathrm{geom}}^m,
 \end{aligned}\\
 \|d_{\mathrm{geom}}^m\|_{L^2}\le C_R\tau(1+\|\widehat v_{\mathrm n}^m\|_{H^2}).
\end{gathered}
\end{equation}

We next express the nonlinear terms through $\widehat v_{\mathrm n}^m$.
Equation~\eqref{time:resolvent-estimates} and its proof give
$\|\widehat v_{\mathrm n}^m\|_{H^\sigma}\le C_R$ and
$\|\widehat v_{\mathrm n}^m\|_{L^2}
\ge c_0-C_R\tau^{\sigma/2}\ge c_0/2$ for small $\tau$.
The operator estimates therefore apply with an enlarged radius and
lower bound $c_0/2$, with constants still depending only on the fixed data.
The coercivity and elliptic estimates in Lemma~\ref{time:operator-bounds}
show that the map
\[
 (f,w,\ell)\longmapsto
 \Lop_{f,\ell}^{-1}\Dop_{f,\ell}w
\]
is locally Lipschitz in the Sobolev spaces used in the proof of
Lemma~\ref{time:local-estimates}. Replacing the curvature
$\widehat\kappa^m$ by $-\widehat v_{\mathrm n}^m$ changes the tangential velocity
in $H^1$ by at most $C_R\|\widehat\kappa^m+\widehat v_{\mathrm n}^m\|_{H^1}$.
The tangential component of the velocity derivative changes by the
same bound in $L^2$, as does its mean-free primitive in $L^\infty$.
For the terms after this replacement, integration gives
\begin{equation}\label{time:G0}
\begin{aligned}
 &\int_0^\xi\biggl[
 -\ell^{-1}\partial_\zeta(\Lop_{f,\ell}^{-1}\Dop_{f,\ell}f)(\zeta)
 -f(\zeta)^2+\int_0^1f(\eta)^2\,d\eta\biggr]d\zeta\\
 &\quad=-\ell^{-1}\bigl[
 (\Lop_{f,\ell}^{-1}\Dop_{f,\ell}f)(\xi)
 -(\Lop_{f,\ell}^{-1}\Dop_{f,\ell}f)(0)\bigr]-\int_0^\xi\biggl(f(\zeta)^2-
                          \int_0^1f(\eta)^2\,d\eta\biggr)d\zeta.
\end{aligned}
\end{equation}
Substituting \eqref{time:G0} and using
$\Dop_{f,\ell}f=3\ell^{-1}f\partial_\xi f$ gives the cancellation
\begin{equation}\label{time:nonlocal-cancellation}
\begin{aligned}
 &2f^3+2\ell^{-1}f\partial_\xi(\Lop_{f,\ell}^{-1}\Dop_{f,\ell}f)
       -f^3-\Kop_{f,\ell}f\\
 &\quad+\biggl[(\Lop_{f,\ell}^{-1}\Dop_{f,\ell}f)(\xi)
       -(\Lop_{f,\ell}^{-1}\Dop_{f,\ell}f)(0)\\
 &\qquad\quad+\ell\int_0^\xi
       \biggl(f(\zeta)^2-\int_0^1f(\eta)^2\,d\eta\biggr)d\zeta
       \biggr]\ell^{-1}\partial_\xi f
   =\mathcal N(f,\ell).
\end{aligned}
\end{equation}
Apply these identities with $f=\widehat v_{\mathrm n}^m$ and $\ell=\ell^m$
to \eqref{time:curvature-expansion}. Since
$\widehat v_{\mathrm n}^m$ and $\partial_\xi\widehat v_{\mathrm n}^m$ are
uniformly bounded pointwise, \eqref{time:resolvent-estimates} gives
a curvature defect $d_\kappa^m\in L^2(\Torus)$ such that
\begin{equation}\label{time:state-remainder}
 \widehat\kappa^{m+1}=-\widehat v_{\mathrm n}^m-\tau(\widehat v_{\mathrm n}^m)^3
         -\tau\Kop_{\widehat v_{\mathrm n}^m}^m \widehat v_{\mathrm n}^m
         -\tau\mathcal N(\widehat v_{\mathrm n}^m,\ell^m)+\tau d_\kappa^m,\quad \|d_\kappa^m\|_{L^2}\le C_R\tau^{(\sigma-1)/2}.
\end{equation}
The length defect is estimated by dividing the length increment in
\eqref{time:exact-length} by $\tau$ and subtracting the leading term
in \eqref{time:scalar-diffusion}.
The bounds for $d_\kappa^m$ and $d_\ell^m$
use \eqref{time:resolvent-estimates} and the normal-velocity bounds
in the proof of Lemma~\ref{time:local-estimates}:
\[
 \|\widehat\kappa^m+\widehat v_{\mathrm n}^m\|_{H^1}
       \le C_R\tau^{(\sigma-1)/2},\qquad
 \tau\|\widehat v_{\mathrm n}^m\|_{H^2}\le C_R\tau^{\sigma/2}.
\]

Combining the normal resolvent \eqref{time:normal-resolvent} at level
$m+1$, including $m+1=M$, and
\eqref{time:state-remainder} gives an intermediate normal-velocity defect
$\widetilde d_{\mathrm n}^m\in L^2(\Torus)$ such that
\begin{equation}\label{time:schur-recurrence}
\begin{aligned}
 &\frac{\widehat v_{\mathrm n}^{m+1}-\widehat v_{\mathrm n}^m}{\tau}
       +\Sop_{\widehat v_{\mathrm n}^m}^m \widehat v_{\mathrm n}^{m+1}=(\widehat v_{\mathrm n}^m)^3+\Kop_{\widehat v_{\mathrm n}^m}^m \widehat v_{\mathrm n}^m
       +\mathcal N(\widehat v_{\mathrm n}^m,\ell^m)+\widetilde d_{\mathrm n}^m,\\
 &\|\widetilde d_{\mathrm n}^m\|_{L^2}
       \le C_R\tau^{(\sigma-1)/2}+C_R\tau\|\widehat v_{\mathrm n}^{m+1}\|_{H^2}.
\end{aligned}
\end{equation}
Indeed, \eqref{time:K-lipschitz}, the even dependence of $\Sop_{f,\ell}$
on $f$, and
\[
 \|\widehat\kappa^{m+1}+\widehat v_{\mathrm n}^m\|_{H^1}
       \le C_R\tau^{\sigma/2}
\]
control the curvature coefficient change. The length changes by
$O(\tau)$. Its contribution to the coefficient of
$\partial_\xi^2$ gives the $H^2$ term in
\eqref{time:schur-recurrence}, bounded by $C_R\tau^{\sigma/2}$
using the normal-velocity bound from the proof of
Lemma~\ref{time:local-estimates} at the next state. In the induction
proof of Theorem~\ref{time:theorem}, the new curvature state is first
shown to lie in $\mathcal U(R_{\mathrm{ext}})$, which justifies
these estimates before the error bound at that step is established.

For $(f,\ell)\in\mathcal U(R)$,
\[
 \Sop_{f,\ell}=-\ell^{-2}\partial_\xi^2+f^2+\Kop_{f,\ell},
\]
where $f^2+\Kop_{f,\ell}$ is uniformly bounded on $L^2$.
Furthermore,
\begin{equation}\label{time:z-increment}
 \|\widehat v_{\mathrm n}^{m+1}-\widehat v_{\mathrm n}^m\|_{L^2}
 \le\|\widehat\kappa^{m+1}+\widehat v_{\mathrm n}^{m+1}\|_{L^2}
+\|\widehat\kappa^{m+1}+\widehat v_{\mathrm n}^m\|_{L^2}\le C_R\tau^{\sigma/2}.
\end{equation}
The first term is controlled by \eqref{time:resolvent-estimates}.
The second is $O(\tau)$ by \eqref{time:state-remainder}.
Splitting $\Sop_{\widehat v_{\mathrm n}^m}^m$ in
\eqref{time:schur-recurrence} and comparing with
\eqref{time:scalar-diffusion} identifies the normal-velocity defect as
\[
 d_{\mathrm n}^m=\widetilde d_{\mathrm n}^m
 -\bigl((\widehat v_{\mathrm n}^m)^2+\Kop_{\widehat v_{\mathrm n}^m}^m\bigr)
       (\widehat v_{\mathrm n}^{m+1}-\widehat v_{\mathrm n}^m).
\]
The boundedness of the operator in parentheses and
\eqref{time:z-increment} give \eqref{time:defect-bound}.
\end{proof}

\subsection{Stability of the error equations}\label{time:stability-section}

We compare the time-semidiscrete pair
$(\widehat v_{\mathrm n}^m,\ell^m)$ with the exact pair
$(-\widehat\kappa(t_m),|\Gamma(t_m)|)$.
Lemma~\ref{time:prefix-error} bounds their $H^\sigma\times\R$ error
in terms of the initial error and the time-discretization defects.
Its proof uses the following discrete smoothing bound for the
normal-velocity equation.
\begin{lemma}\label{time:fourier-smoothing}
For integers $0\le j<m\le M$, suppose the lengths $\ell^r$,
$j\le r<m$, are defined and satisfy $0<\ell^r\le\ell_{\max}$.
Then, for $0\le\beta<2$,
\begin{equation}\label{time:smoothing}
 \left\|\prod_{r=j}^{m-1}
    (I-\tau(\ell^r)^{-2}\partial_\xi^2)^{-1}\right\|_
       {\mathcal L(H^{\sigma-\beta},H^\sigma)}
 \le C\left(1+((m-j)\tau)^{-\beta/2}\right).
\end{equation}
\end{lemma}

\begin{proof}
Let $f(\xi)=\sum_{n\in\Z}f_n e^{2\pi\mathrm i n\xi}$ be a
trigonometric polynomial, with Fourier coefficients $f_n$.
Since $\partial_\xi^2e^{2\pi\mathrm i n\xi}
=-(2\pi n)^2e^{2\pi\mathrm i n\xi}$, each resolvent acts by
\[
 (I-\tau(\ell^r)^{-2}\partial_\xi^2)^{-1}e^{2\pi\mathrm i n\xi}
 =\frac{e^{2\pi\mathrm i n\xi}}
 {1+\tau(\ell^r)^{-2}(2\pi n)^2}.
\]
The resolvents are diagonal in the same Fourier basis and therefore
commute. The length bound and Bernoulli's inequality give
\[
\begin{aligned}
 \prod_{r=j}^{m-1}(1+\tau(\ell^r)^{-2}(2\pi n)^2)
 &\ge(1+\tau\ell_{\max}^{-2}(2\pi n)^2)^{m-j}\ge1+(m-j)\tau\ell_{\max}^{-2}(2\pi n)^2.
\end{aligned}
\]
For $\beta=0$, the multiplier is bounded by one.
For $0<\beta<2$, the inequalities
$(1+z)^{\beta/2}\le1+z^{\beta/2}$ and
$z^{\beta/2}/(1+z)\le1$, $z\ge0$, yield
\[
\begin{aligned}
 \frac{(1+(2\pi n)^2)^{\beta/2}}
 {\displaystyle\prod_{r=j}^{m-1}
       (1+\tau(\ell^r)^{-2}(2\pi n)^2)}
 &\le\frac{1+|2\pi n|^\beta}
 {1+(m-j)\tau\ell_{\max}^{-2}(2\pi n)^2}\\
 &\le1+\ell_{\max}^{\beta}((m-j)\tau)^{-\beta/2}\\
 &\le C\bigl(1+((m-j)\tau)^{-\beta/2}\bigr).
\end{aligned}
\]
Parseval's identity with the Sobolev weights now gives
\[
\begin{aligned}
 \left\|\prod_{r=j}^{m-1}
       (I-\tau(\ell^r)^{-2}\partial_\xi^2)^{-1}f\right\|_{H^\sigma}^2
 &=\sum_{n\in\Z}
 \frac{(1+(2\pi n)^2)^\sigma|f_n|^2}
 {\left[\displaystyle\prod_{r=j}^{m-1}
       (1+\tau(\ell^r)^{-2}(2\pi n)^2)\right]^2}\\
 &\le C\bigl(1+((m-j)\tau)^{-\beta/2}\bigr)^2
       \sum_{n\in\Z}(1+(2\pi n)^2)^{\sigma-\beta}|f_n|^2\\
 &=C\bigl(1+((m-j)\tau)^{-\beta/2}\bigr)^2
       \|f\|_{H^{\sigma-\beta}}^2.
\end{aligned}
\]
Taking square roots and using the density of trigonometric polynomials
in $H^{\sigma-\beta}$ proves \eqref{time:smoothing}.
The constant depends only on $\beta$ and $\ell_{\max}$, independently
of $j,m,\tau$ and the admissible lengths.
The argument includes $m=j+1$. The mode $n=0$ is unchanged by every
resolvent, which accounts for the constant term in the estimate.
\end{proof}

We use the next estimate in the proof of Theorem~\ref{time:theorem}
and again with $O(\tau)$ defect bounds in
Section~\ref{up:sec:first-order} to obtain first-order convergence.
\begin{lemma}\label{time:prefix-error}
Suppose $(\widehat v_{\mathrm n}^m,\ell^m)\in\mathcal U(R)$
on an admissible prefix $0\le m\le m_*$,
$1\le m_*\le M$, and define
\begin{equation}\label{time:joint-error}
 E^m=\|\widehat v_{\mathrm n}^m+\widehat\kappa(t_m)\|_{H^\sigma}
           +\bigl|\ell^m-|\Gamma(t_m)|\bigr|.
\end{equation}
For the defects $d_{\mathrm n}^j,d_\ell^j$, $0\le j<m_*$, in
\eqref{time:scalar-diffusion},
\begin{equation}\label{time:prefix-bound}
 \max_{0\le m\le m_*}E^m
 \le C_{R,T}\left(E^0+\tau
       +\max_{0\le j<m_*}(\|d_{\mathrm n}^j\|_{L^2}+|d_\ell^j|)\right).
\end{equation}
\end{lemma}

\begin{proof}
For this proof, write the normal-velocity and length errors as
\[
 e_{\mathrm n}^m=\widehat v_{\mathrm n}^m+\widehat\kappa(t_m),
 \qquad e_\ell^m=\ell^m-|\Gamma(t_m)|,
 \qquad 0\le m\le m_*.
\]
Using the oddness of $\mathcal N$, subtract the semi-implicit equations
for $-\widehat\kappa$ and $|\Gamma|$ obtained from
\eqref{time:reference-equations} and \eqref{time:g0} from
\eqref{time:scalar-diffusion}. For $0\le m<m_*$, this gives
\[
\begin{aligned}
 \frac{e_{\mathrm n}^{m+1}-e_{\mathrm n}^m}{\tau}
 -(\ell^m)^{-2}\partial_\xi^2e_{\mathrm n}^{m+1}
 ={}&\mathcal N(\widehat v_{\mathrm n}^m,\ell^m)
       -\mathcal N(-\widehat\kappa(t_m),|\Gamma(t_m)|)\\
 &+\bigl(|\Gamma(t_m)|^{-2}-(\ell^m)^{-2}\bigr)
       \partial_\xi^2\widehat\kappa(t_{m+1})\\
 &+d_{\mathrm n}^m+O_{L^2}(\tau),\\
 \frac{e_\ell^{m+1}-e_\ell^m}{\tau}
 ={}&-\ell^m\int_0^1(\widehat v_{\mathrm n}^m)^2\,d\xi
       +|\Gamma(t_m)|\int_0^1\widehat\kappa(\xi,t_m)^2\,d\xi\\
 &+d_\ell^m+O(\tau).
\end{aligned}
\]
The $O(\tau)$ terms are the consistency defects of the smooth exact
equations. Constants in this proof may also depend on that fixed solution.

The inverse-operator identity and elliptic estimates in the proof of
Lemma~\ref{time:operator-bounds} give uniform Lipschitz bounds for
$(f,\ell)\mapsto\Lop_{f,\ell}^{-1}\Dop_{f,\ell}f$ from
$\mathcal U(R)$ into $H^{\sigma+1}$.
Point evaluation is continuous there, and $H^{\sigma-1}$ is an algebra.
The formula \eqref{time:N}, the length bounds, and exact regularity yield
\[
\begin{aligned}
 &\|\mathcal N(\widehat v_{\mathrm n}^m,\ell^m)
       -\mathcal N(-\widehat\kappa(t_m),|\Gamma(t_m)|)\|_{H^{\sigma-1}}
       \le C_RE^m,\\
 &\bigl||\Gamma(t_m)|^{-2}-(\ell^m)^{-2}\bigr|
       \|\partial_\xi^2\widehat\kappa(t_{m+1})\|_{H^{\sigma-1}}
       \le C_R|e_\ell^m|,\\
 &\left|\ell^m\int_0^1(\widehat v_{\mathrm n}^m)^2\,d\xi
       -|\Gamma(t_m)|\int_0^1\widehat\kappa(\xi,t_m)^2\,d\xi\right|
       \le C_RE^m.
\end{aligned}
\]
Discrete variation of constants and Lemma~\ref{time:fourier-smoothing}
with $\beta=0,1,\sigma$ give, respectively,
\[
\begin{aligned}
 \|e_{\mathrm n}^m\|_{H^\sigma}
 \le{}&C\|e_{\mathrm n}^0\|_{H^\sigma}
   +C_R\sum_{j=0}^{m-1}\tau
       \bigl(1+(t_m-t_j)^{-1/2}\bigr)E^j\\
 &+C_R\sum_{j=0}^{m-1}\tau
       \bigl(1+(t_m-t_j)^{-\sigma/2}\bigr)
       \bigl(\|d_{\mathrm n}^j\|_{L^2}+\tau\bigr).
\end{aligned}
\]
Summing the length-error equation gives
\[
 |e_\ell^m|\le |e_\ell^0|
       +C_R\sum_{j=0}^{m-1}\tau E^j
       +\sum_{j=0}^{m-1}\tau\bigl(|d_\ell^j|+C\tau\bigr).
\]
Since $\sigma/2<1$,
\[
 \sum_{j=0}^{m-1}\tau\bigl(1+(t_m-t_j)^{-\sigma/2}\bigr)
 \le T+\frac{T^{1-\sigma/2}}{1-\sigma/2}.
\]
Combining these estimates yields
\begin{equation}\label{time:volterra}
\begin{split}
 E^m\le{}&C_{R,T}\left(E^0+\tau
        +\max_{0\le j<m_*}(\|d_{\mathrm n}^j\|_{L^2}+|d_\ell^j|)\right)\\
 &+C_R\sum_{j=0}^{m-1}\tau
       \left(1+(t_m-t_j)^{-1/2}\right)E^j.
\end{split}
\end{equation}
Enlarge $C_{R,T}$ so that \eqref{time:volterra} also holds for $m=0$.
Choose $d>0$, independent of $\tau$, such that
\[
 C_R(d+2\sqrt d)\le\tfrac12.
\]
The function $t\mapsto1+t^{-1/2}$ is decreasing and integrable near
zero, so for every $\tau>0$,
\[
 \sum_{\substack{j\ge1\\j\tau\le d}}\tau\bigl(1+(j\tau)^{-1/2}\bigr)
 \le\int_0^d(1+t^{-1/2})\,dt=d+2\sqrt d.
\]
For each integer $q\ge0$, split the sum in \eqref{time:volterra}
at $t_j=qd$ for $t_m\in[qd,(q+1)d)$.
The terms with $t_j\ge qd$ are absorbed using the bound above.
Those with $t_j<qd$ have total weight at most $T+2\sqrt T$.
Taking maxima therefore gives
\[
\begin{aligned}
 \max_{\substack{0\le m\le m_*\\t_m<(q+1)d}}E^m
 \le{}&2C_{R,T}\left(E^0+\tau+
       \max_{0\le j<m_*}(\|d_{\mathrm n}^j\|_{L^2}+|d_\ell^j|)\right)\\
 &+\bigl(1+2C_R(T+2\sqrt T)\bigr)
       \max_{\substack{0\le m\le m_*\\t_m<qd}}E^m,
\end{aligned}
\]
where a maximum over an empty set is zero.
Induction for $q=0,\ldots,\lfloor T/d\rfloor$ proves
\eqref{time:prefix-bound}.
\end{proof}

\subsection{Proof of Theorem~\ref{time:theorem}}\label{time:continuation-section}

We use the defect bounds in Lemma~\ref{time:intermediate-equation}
and the stability estimate above to prove existence and the
curvature and length error bounds by induction on the time index.

\begin{proof}[Proof of Theorem~\ref{time:theorem}]
The set $\{(\widehat\kappa(t),|\Gamma(t)|):0\le t\le T\}$ is a compact
subset of $\operatorname{int}\mathcal U(R)$. Choose $\delta>0$ such that
\[
 2\delta<\min_{0\le t\le T}\bigl\{
 R-\|\widehat\kappa(t)\|_{H^\sigma},\,
 \|\widehat\kappa(t)\|_{L^2}-c_0,|\Gamma(t)|-\ell_{\min},\,
 \ell_{\max}-|\Gamma(t)|\bigr\}.
\]
Thus $\|f-\widehat\kappa(t)\|_{H^\sigma}
+\bigl|\ell-|\Gamma(t)|\bigr|\le\delta$ implies
$(f,\ell)\in\operatorname{int}\mathcal U(R)$.
We prove by induction on $m$ that the updates
$x^j\in C^\infty(\Torus;\R^2)$, $0\le j\le m$, exist uniquely,
are regular embeddings, and satisfy
\[
 \max_{0\le j\le m}\bigl(
   \|\widehat\kappa^j-\widehat\kappa(t_j)\|_{H^\sigma}
   +\bigl|\ell^j-|\Gamma(t_j)|\bigr|\bigr)<\delta.
\]
This holds for $m=0$ because $x^0$ is a $C^\infty$ regular embedding
of the exact initial curve.
At each available curve $x^j$, the first identity in
\eqref{time:normal-resolvent} determines $\widehat v_{\mathrm n}^j$
before the next position update is formed. Lemma~\ref{time:exact-step}
identifies it with the normal component of that update, and the same
formula defines the terminal scalar at $j=M$.
Smoothness of the initial curve gives
\[
 \widehat v_{\mathrm n}^0+\widehat\kappa^0
 =\tau(I+\tau\Sop_{\widehat\kappa^0}^0)^{-1}
       \Sop_{\widehat\kappa^0}^0\widehat\kappa^0
 =O_{H^\sigma}(\tau).
\]
Thus $E^0=O(\tau)$.

Assume the assertion holds through some $m<M$.
Then $(\widehat\kappa^j,\ell^j)\in\mathcal U(R)$ for $0\le j\le m$.
Since $x^m$ is $C^\infty$ and regular, Lemma~\ref{time:exact-step}
gives the unique solution $x^{m+1}\in C^\infty(\Torus;\R^2)$
of \eqref{time:bgn} by the spatial regularity argument in its proof.
The curvature and length bounds do not control position or orientation.
For the embedding check, choose a rotation $Q^m\in SO(2)$ and a
translation $b^m\in\R^2$ satisfying
\[
 Q^m\widehat{\mathbf t}^m(0)=\widehat{\mathbf t}(0,t_m),
 \qquad b^m=\widehat X(0,t_m)-Q^m\widehat x^m(0).
\]
Thus $Q^m\widehat x^m+b^m$ has the same initial point and tangent
as $\widehat X(t_m)$. Apply the same rigid motion
$y\mapsto Q^my+b^m$ to the current curve and its update:
\[
 \widehat x^m\longmapsto Q^m\widehat x^m+b^m,
 \qquad
 \widehat x^m+\tau\widehat v^m
 \longmapsto Q^m\widehat x^m+b^m+\tau Q^m\widehat v^m.
\]
This transformation preserves curvature and length and is used only
in the embedding check.
Since $Q^mJ=JQ^m$, the Frenet identities \eqref{eq:orientation} give
\[
\begin{aligned}
 \partial_\xi(Q^m\widehat{\mathbf t}^m)
 &=-\ell^m\widehat\kappa^m J(Q^m\widehat{\mathbf t}^m),\\
 \partial_\xi\widehat{\mathbf t}(t_m)
 &=-|\Gamma(t_m)|\widehat\kappa(t_m)J\widehat{\mathbf t}(t_m).
\end{aligned}
\]
Subtracting the equations gives
\[
 \partial_\xi\bigl(Q^m\widehat{\mathbf t}^m
                     -\widehat{\mathbf t}(t_m)\bigr)
 =-\ell^m\widehat\kappa^m
      J\bigl(Q^m\widehat{\mathbf t}^m-\widehat{\mathbf t}(t_m)\bigr)-\bigl(\ell^m\widehat\kappa^m
           -|\Gamma(t_m)|\widehat\kappa(t_m)\bigr)
      J\widehat{\mathbf t}(t_m).
\]
The tangent difference vanishes at $\xi=0$, and
$|J\widehat{\mathbf t}|=1$. Integrating and taking absolute values
therefore gives, for $0\le\xi\le1$,
\[
\begin{aligned}
 &\bigl|Q^m\widehat{\mathbf t}^m(\xi)
           -\widehat{\mathbf t}(\xi,t_m)\bigr|\\
 &\quad\le\int_0^\xi\ell^m|\widehat\kappa^m(\zeta)|
       \bigl|Q^m\widehat{\mathbf t}^m(\zeta)
                  -\widehat{\mathbf t}(\zeta,t_m)\bigr|\,d\zeta+\int_0^\xi
       \bigl|\ell^m\widehat\kappa^m(\zeta)
              -|\Gamma(t_m)|\widehat\kappa(\zeta,t_m)\bigr|\,d\zeta.
\end{aligned}
\]
Since $\ell^m\|\widehat\kappa^m\|_{L^1}\le\ell_{\max}R$,
Gronwall's inequality in the spatial variable $\xi\in[0,1]$ yields
\[
\begin{aligned}
 \|Q^m\widehat{\mathbf t}^m-\widehat{\mathbf t}(t_m)\|_{L^\infty}
 &\le\exp\!\bigl(\ell^m\|\widehat\kappa^m\|_{L^1}\bigr)
       \|\ell^m\widehat\kappa^m
              -|\Gamma(t_m)|\widehat\kappa(t_m)\|_{L^1}\\
 &\le e^{\ell_{\max}R}\Bigl(
       \ell_{\max}\|\widehat\kappa^m-\widehat\kappa(t_m)\|_{L^1}
       +\bigl|\ell^m-|\Gamma(t_m)|\bigr|
          \|\widehat\kappa(t_m)\|_{L^1}\Bigr)\\
 &\le C\bigl(\|\widehat\kappa^m-\widehat\kappa(t_m)\|_{H^\sigma}
             +\bigl|\ell^m-|\Gamma(t_m)|\bigr|\bigr).
\end{aligned}
\]
The last step uses the fixed exact curvature bound and
$\|f\|_{L^1}\le\|f\|_{H^\sigma}$ on $\Torus$.
The initial positions also agree after alignment. Since
$\partial_\xi\widehat x^m=\ell^m\widehat{\mathbf t}^m$ and
$\partial_\xi\widehat X(t_m)=|\Gamma(t_m)|\widehat{\mathbf t}(t_m)$,
\[
\begin{aligned}
 &Q^m\widehat x^m(\xi)+b^m-\widehat X(\xi,t_m)\\
 &\quad=\int_0^\xi\Bigl[
       \ell^m\bigl(Q^m\widehat{\mathbf t}^m(\zeta)
                     -\widehat{\mathbf t}(\zeta,t_m)\bigr)
       +(\ell^m-|\Gamma(t_m)|)\widehat{\mathbf t}(\zeta,t_m)
       \Bigr]\,d\zeta.
\end{aligned}
\]
Bounding the integral and its derivative therefore gives
\[
 \|Q^m\widehat x^m+b^m-\widehat X(t_m)\|_{C^1}
 \le C\bigl(\|\widehat\kappa^m-\widehat\kappa(t_m)\|_{H^\sigma}
          +\bigl|\ell^m-|\Gamma(t_m)|\bigr|\bigr)
 \le C\delta.
\]

For the velocity, the proof of Lemma~\ref{time:local-estimates}
establishes
\[
 \|\widehat v_{\mathrm n}^m\|_{H^\sigma}
 +\|\widehat v_{\mathrm t}^m\|_{H^{\sigma+1}}\le C_R.
\]
The decomposition
$\widehat v^m=\widehat v_{\mathrm n}^m\widehat{\mathbf n}^m
             +\widehat v_{\mathrm t}^m\widehat{\mathbf t}^m$
and the Frenet identities imply
\[
 \partial_\xi\widehat v^m
 =\bigl(\partial_\xi\widehat v_{\mathrm n}^m
             -\ell^m\widehat\kappa^m\widehat v_{\mathrm t}^m\bigr)
              \widehat{\mathbf n}^m
  +\bigl(\partial_\xi\widehat v_{\mathrm t}^m
             +\ell^m\widehat\kappa^m\widehat v_{\mathrm n}^m\bigr)
              \widehat{\mathbf t}^m.
\]
Since $\sigma>3/2$, the embedding $H^\sigma(\Torus)\hookrightarrow C^1(\Torus)$
and the uniform $L^\infty$ bound on $\ell^m\widehat\kappa^m$ give
\[
 \|Q^m\widehat v^m\|_{C^1}
 \le C_R\bigl(\|\widehat v_{\mathrm n}^m\|_{H^\sigma}
               +\|\widehat v_{\mathrm t}^m\|_{H^{\sigma+1}}\bigr)
 \le C_R.
\]
The triangle inequality now yields
\[
\begin{aligned}
 \|Q^m(\widehat x^m+\tau\widehat v^m)+b^m-\widehat X(t_m)\|_{C^1}
 &\le\|Q^m\widehat x^m+b^m-\widehat X(t_m)\|_{C^1}
        +\tau\|Q^m\widehat v^m\|_{C^1}\\
 &\le C\delta+C_R\tau.
\end{aligned}
\]
Write $d_{\Torus}(\xi,\eta)=\min_{k\in\Z}|\xi-\eta+k|$.
Uniform continuity of $\widehat{\mathbf t}$ on $\Torus\times[0,T]$
gives $a\in(0,1/2)$, independent of $m$ and $\tau$, such that
\[
 |\widehat{\mathbf t}(\xi,t)-\widehat{\mathbf t}(\xi_0,t)|\le\tfrac14
 \quad\text{if }d_{\Torus}(\xi,\xi_0)<a,\quad 0\le t\le T.
\]
Since both tangents have unit length, their scalar product is at least
$3/4$. For this fixed $a$, the set
$\{(\xi,\eta,t)\in\Torus^2\times[0,T]:d_{\Torus}(\xi,\eta)\ge a\}$
is compact. Injectivity of $\widehat X(\cdot,t)$ at every time gives
\[
 \min_{\substack{0\le t\le T\\d_{\Torus}(\xi,\eta)\ge a}}
       |\widehat X(\xi,t)-\widehat X(\eta,t)|>0.
\]
Both $a$ and this minimum depend only on the exact solution. We may
therefore decrease $\delta$ and then restrict $\tau$, independently of
$m$, so that
\[
 C\delta+C_R\tau\le\frac14\min\left\{
 \ell_{\min},\,
 \min_{\substack{0\le t\le T\\d_{\Torus}(\xi,\eta)\ge a}}
       |\widehat X(\xi,t)-\widehat X(\eta,t)|\right\}.
\]
For $d_{\Torus}(\xi,\xi_0)<a$, the $C^1$ estimate above and
$\partial_\xi\widehat X=|\Gamma(t)|\widehat{\mathbf t}$ give
\[
\begin{aligned}
 &Q^m\partial_\xi(\widehat x^m+\tau\widehat v^m)(\xi)
       \cdot\widehat{\mathbf t}(\xi_0,t_m)\\
 &\quad\ge |\Gamma(t_m)|\widehat{\mathbf t}(\xi,t_m)
       \cdot\widehat{\mathbf t}(\xi_0,t_m)-(C\delta+C_R\tau)\ge\tfrac34\ell_{\min}-(C\delta+C_R\tau)
       \ge\tfrac12\ell_{\min}.
\end{aligned}
\]
For any distinct pair with $d_{\Torus}(\xi,\eta)<a$, interchange its
points if necessary and choose lifts with
$0<\eta-\xi=d_{\Torus}(\xi,\eta)<a$. Integration along this interval
with $\xi_0=\xi$ yields
\[
\begin{aligned}
 &Q^m\bigl[(\widehat x^m+\tau\widehat v^m)(\eta)
           -(\widehat x^m+\tau\widehat v^m)(\xi)\bigr]
        \cdot\widehat{\mathbf t}(\xi,t_m)\\
 &\quad=\int_\xi^\eta
       Q^m\partial_\zeta(\widehat x^m+\tau\widehat v^m)(\zeta)
         \cdot\widehat{\mathbf t}(\xi,t_m)\,d\zeta\\
 &\quad\ge\tfrac12\ell_{\min}(\eta-\xi)>0.
\end{aligned}
\]
For the remaining pairs, with $d_{\Torus}(\xi,\eta)\ge a$,
invariance of distances under rigid motions and the $C^0$ estimate give
\[
\begin{aligned}
 &|(\widehat x^m+\tau\widehat v^m)(\xi)
       -(\widehat x^m+\tau\widehat v^m)(\eta)|\\
 &\quad\ge|\widehat X(\xi,t_m)-\widehat X(\eta,t_m)|
       -2(C\delta+C_R\tau)\\
 &\quad\ge\frac12
       \min_{\substack{0\le t\le T\\d_{\Torus}(\xi,\eta)\ge a}}
       |\widehat X(\xi,t)-\widehat X(\eta,t)|>0.
\end{aligned}
\]
Thus the update is injective on $\Torus$. Since $\chi^m$ is invertible,
$x^{m+1}$ is injective.
Its regularity follows from
\[
 |\widehat{\mathbf t}^m+\tau\partial_s\widehat v^m|
 \ge1+\tau(\partial_s\widehat v_{\mathrm t}^m+\widehat\kappa^m\widehat v_{\mathrm n}^m)
 \ge1-C_R\tau\ge\tfrac12.
\]

Equations \eqref{time:resolvent-estimates},
\eqref{time:curvature-increment} and \eqref{time:metric-estimates}
now give
\[
\begin{aligned}
 \|\widehat\kappa^{m+1}\|_{H^\sigma}
 &\le\|\widehat v_{\mathrm n}^m\|_{H^\sigma}+C_R\tau^{1/2}
 \le C_R,\\
 \|\widehat\kappa^{m+1}-\widehat\kappa^m\|_{L^2}
 &\le\|\widehat\kappa^{m+1}+\widehat v_{\mathrm n}^m\|_{L^2}
       +\|\widehat\kappa^m+\widehat v_{\mathrm n}^m\|_{L^2}
 \le C_R\tau^{\sigma/2},\\
 |\ell^{m+1}-\ell^m|&\le C_R\tau.
\end{aligned}
\]
By the choice of $\delta$,
\[
\begin{aligned}
 \ell_{\min}+\delta-C_R\tau
 &\le\ell^{m+1}\le\ell_{\max}-\delta+C_R\tau,\\
 \|\widehat\kappa^{m+1}\|_{L^2}
 &\ge c_0+\delta-C_R\tau^{\sigma/2}.
\end{aligned}
\]
Fix $R_{\mathrm{ext}}>\max\{R,C_R\}$ independently of $m$ and $\tau$.
For sufficiently small $\tau$, the new state satisfies
\[
 (\widehat\kappa^{m+1},\ell^{m+1})\in\mathcal U(R_{\mathrm{ext}}).
\]
The induction hypothesis therefore places all curvature states through
$m+1$ in this fixed enlarged set. The normal resolvent at $j=m+1$
is now available, including when $m+1=M$.
The normal-velocity bounds from Lemma~\ref{time:local-estimates} give,
for $0\le j\le m+1$ and small $\tau$,
\[
 \|\widehat v_{\mathrm n}^j\|_{H^\sigma}\le C_{R_{\mathrm{ext}}},\qquad
 \|\widehat v_{\mathrm n}^j\|_{L^2}
 \ge c_0+\delta-C_{R_{\mathrm{ext}}}\tau^{\sigma/2}\ge c_0.
\]
Thus $(\widehat v_{\mathrm n}^j,\ell^j)\in\mathcal U(C_{R_{\mathrm{ext}}})$.
Increase $C_{R_{\mathrm{ext}}}$ if needed so that
$C_{R_{\mathrm{ext}}}\ge R$, which also places the exact pair in this set.
Lemma~\ref{time:intermediate-equation}, applied with radius
$R_{\mathrm{ext}}$, and Lemma~\ref{time:prefix-error}, with $m_*=m+1$, yield
\[
\begin{aligned}
 \max_{0\le j\le m}
   \bigl(\|d_{\mathrm n}^j\|_{L^2}+|d_\ell^j|\bigr)
 &\le C_R\tau^{(\sigma-1)/2},\\
 \max_{0\le j\le m+1}E^j
 &\le C_{R,T}\tau^{(\sigma-1)/2}.
\end{aligned}
\]
The radii are fixed in terms of $R$ and the prescribed data, so these
constants are independent of $m$ and $\tau$. Together with
\eqref{time:curvature-increment}, this yields
\begin{equation}\label{time:exit-closure}
\begin{aligned}
	\|\widehat\kappa^{m+1}-\widehat\kappa(t_{m+1})\|_{H^\sigma} &\le \|\widehat\kappa^{m+1}+\widehat v_{\mathrm n}^m\|_{H^\sigma}
      +\|\widehat v_{\mathrm n}^m+\widehat\kappa(t_m)\|_{H^\sigma}+\|\widehat\kappa(t_m)
                   -\widehat\kappa(t_{m+1})\|_{H^\sigma}\\
												   &\le  C_R\tau^{1/2}+C_{R,T}\tau^{(\sigma-1)/2}+C\tau
 \le C_{R,T}\tau^{(\sigma-1)/2}.
\end{aligned}
\end{equation}
The length estimate follows from Lemma~\ref{time:prefix-error}.
Choose $\tau_0>0$ to satisfy all the preceding smallness conditions
and $C_{R,T}\tau_0^{(\sigma-1)/2}<\delta$.
For $0<\tau\le\tau_0$,
\[
 \|\widehat\kappa^{m+1}-\widehat\kappa(t_{m+1})\|_{H^\sigma}
 +\bigl|\ell^{m+1}-|\Gamma(t_{m+1})|\bigr|
 \le C_{R,T}\tau^{(\sigma-1)/2}<\delta,
\]
which closes the induction. Thus all updates through $M$ exist
uniquely as $C^\infty$ regular embeddings. The error bound obtained at each
induction step proves \eqref{time:curvature-convergence}.
\end{proof}

\section{Uniform regularity and first-order time convergence}\label{sec:upgrade}

We use the estimates established in Section~\ref{sec:time} to prove
uniform $H^3$ bounds for the curvature and improve the
time-discretization defects to order $\tau$.
By Theorem~\ref{time:theorem},
$(\widehat\kappa^m,\ell^m)\in\mathcal U(R)$ for $0\le m\le M$.
The velocity and metric estimates established in the proof of
Lemma~\ref{time:local-estimates} therefore hold uniformly for
$0\le m<M$ and sufficiently small $\tau$.
Unless stated otherwise, constants in Section~\ref{up:sec:uniform-h3}
depend only on $R,\sigma,\ell_{\min},\ell_{\max}$, and $c_0$.

\begin{theorem}\label{up:thm:time}
Under Assumption~\ref{ass:exact-flow}, for every fixed $3/2<\sigma<2$
and all sufficiently small $\tau=T/M$,
\begin{equation}\label{up:eq:first-order}
 \max_{0\le m\le M}\biggl(
 \|\widehat\kappa^m-\widehat\kappa(t_m)\|_{H^\sigma}
 +\bigl|\ell^m-|\Gamma(t_m)|\bigr|+\|\widehat x^m-\widehat X(t_m)\|_{C^1}\biggr)
 \le C\tau.
\end{equation}
The constant is independent of $m,\tau$.
\end{theorem}
The proof is given in Section~\ref{up:sec:first-order}.

\subsection{Uniform curvature bounds in \texorpdfstring{$H^3$}{H3}}
\label{up:sec:uniform-h3}

The following lemma establishes uniform $H^3$ bounds for the curvature
and normal velocity. These bounds are used in
Section~\ref{up:sec:first-order} to estimate the defects by $C\tau$.

\begin{lemma}[Uniform $H^3$ bounds]\label{up:lem:resolvent}
The time-semidiscrete solution satisfies
\begin{equation}\label{up:eq:uniform-H3}
 \max_{0\le m\le M}
 \bigl(\|\widehat\kappa^m\|_{H^3}+\|\widehat v_{\mathrm n}^m\|_{H^3}\bigr)\le C.
\end{equation}
Here $C$ may also depend on $T$ and $\|\widehat\kappa^0\|_{H^3}$.
\end{lemma}
\begin{proof}
Fix $0\le m<M$. Theorem~\ref{time:theorem} and the proof of
Lemma~\ref{time:local-estimates} give
\[
 \|\widehat\kappa^m\|_{H^\sigma}
 +\|\widehat v_{\mathrm n}^m\|_{H^\sigma}
 +\|\widehat v_{\mathrm t}^m\|_{H^{\sigma+1}}\le C.
\]
Since $\sigma>3/2$, Sobolev embedding yields
\begin{equation}\label{up:eq:velocity-uniform-bounds}
 \|\widehat\kappa^m\|_{W^{1,\infty}}
 +\|\widehat v_{\mathrm n}^m\|_{W^{1,\infty}}
 +\|\widehat v_{\mathrm t}^m\|_{W^{2,\infty}}\le C.
\end{equation}
For smooth periodic functions $f,g$ and every fixed integer $j\ge1$, we use the
following product estimate, c.f., \cite[Chapter~13, Proposition~3.7]{Taylor2011}:
\[
 \|fg\|_{H^j}
 \le C_j\bigl(\|f\|_{L^\infty}\|g\|_{H^j}
            +\|g\|_{L^\infty}\|f\|_{H^j}\bigr).
\]

We first prove
\begin{equation}\label{up:eq:near-contraction}
 \|\widehat v_{\mathrm n}^m\|_{H^3}\le(1+C\tau)\|\widehat\kappa^m\|_{H^3}.
\end{equation}
The normal resolvent equation reads
\begin{equation}\label{up:eq:scalar}
\begin{aligned}
 \widehat v_{\mathrm n}^m-\tau(\ell^m)^{-2}\partial_\xi^2\widehat v_{\mathrm n}^m=-\widehat\kappa^m-\tau\bigl[(\widehat\kappa^m)^2\widehat v_{\mathrm n}^m+\Dop_{\widehat\kappa^m}^m\widehat v_{\mathrm t}^m\bigr].
\end{aligned}
\end{equation}
The second identity in \eqref{time:normal-resolvent},
$\Lop_{\widehat\kappa^m}^m\widehat v_{\mathrm t}^m
=\Dop_{\widehat\kappa^m}^m\widehat v_{\mathrm n}^m$,
reads, by \eqref{time:LD},
\[
 -(\ell^m)^{-2}\partial_\xi^2\widehat v_{\mathrm t}^m
 =\Dop_{\widehat\kappa^m}^m\widehat v_{\mathrm n}^m
       -(\widehat\kappa^m)^2\widehat v_{\mathrm t}^m.
\]
By \eqref{time:local-set}, \eqref{up:eq:velocity-uniform-bounds},
and the product estimate,
\[
\begin{aligned}
 \|\Dop_{\widehat\kappa^m}^m\widehat v_{\mathrm n}^m\|_{H^1}
 &\le C\bigl(\|(\partial_\xi\widehat\kappa^m)\widehat v_{\mathrm n}^m\|_{H^1}
             +\|\widehat\kappa^m\partial_\xi\widehat v_{\mathrm n}^m\|_{H^1}\bigr)\le C(\|\widehat\kappa^m\|_{H^2}+\|\widehat v_{\mathrm n}^m\|_{H^2}),\\
 \|(\widehat\kappa^m)^2\widehat v_{\mathrm t}^m\|_{H^1}&\le C.
\end{aligned}
\]
The elliptic regularity estimate
\cite[Section~6.3.1, Theorem~2]{Evans2010} controls the $H^3$ norm of
$\widehat v_{\mathrm t}^m$. 
\begin{equation}\label{up:eq:V-high}
\begin{aligned}
 \|\widehat v_{\mathrm t}^m\|_{H^3}
 &\le C\bigl(\|\widehat v_{\mathrm t}^m\|_{L^2}+\|\Dop_{\widehat\kappa^m}^m\widehat v_{\mathrm n}^m\|_{H^1}
                 +\|(\widehat\kappa^m)^2\widehat v_{\mathrm t}^m\|_{H^1}\bigr)\\
 &\le C(1+\|\widehat\kappa^m\|_{H^2}+\|\widehat v_{\mathrm n}^m\|_{H^2}).
\end{aligned}
\end{equation}
The same product estimate gives
\[
\begin{aligned}
 \|(\widehat\kappa^m)^2\widehat v_{\mathrm n}^m\|_{H^2}
 &\le C(\|\widehat\kappa^m\|_{H^2}+\|\widehat v_{\mathrm n}^m\|_{H^2}),\\
 \|\Dop_{\widehat\kappa^m}^m\widehat v_{\mathrm t}^m\|_{H^2}
 &\le C\bigl(\|(\partial_\xi\widehat\kappa^m)\widehat v_{\mathrm t}^m\|_{H^2}
             +\|\widehat\kappa^m\partial_\xi\widehat v_{\mathrm t}^m\|_{H^2}\bigr)\\
 &\le C(\|\widehat\kappa^m\|_{H^3}+\|\widehat v_{\mathrm t}^m\|_{H^3}).
\end{aligned}
\]
Combining these bounds with \eqref{up:eq:V-high} yields
\begin{equation}\label{up:eq:P-tame}
\begin{aligned}
 \|(\widehat\kappa^m)^2\widehat v_{\mathrm n}^m+\Dop_{\widehat\kappa^m}^m\widehat v_{\mathrm t}^m\|_{H^2}\le C(1+\|\widehat\kappa^m\|_{H^3}+\|\widehat v_{\mathrm n}^m\|_{H^3}).
\end{aligned}
\end{equation}

Apply $\partial_\xi^j$ to \eqref{up:eq:scalar}, take the
$L^2(\Torus)$ inner product with $\partial_\xi^j\widehat v_{\mathrm n}^m$,
and sum over $j=0,1,2,3$. Periodic integration by parts and the
identity $a(a+b)=\tfrac12[a^2-b^2+(a+b)^2]$ give
\[
\begin{aligned}
 &\frac12\bigl(\|\widehat v_{\mathrm n}^m\|_{H^3}^2
                   -\|\widehat\kappa^m\|_{H^3}^2\bigr)
 +\frac12\|\widehat\kappa^m+\widehat v_{\mathrm n}^m\|_{H^3}^2
 +\frac{\tau}{(\ell^m)^2}\|\partial_\xi\widehat v_{\mathrm n}^m\|_{H^3}^2\\
 &\qquad=-\tau\sum_{j=0}^3
 \bigl(\partial_\xi^j[(\widehat\kappa^m)^2\widehat v_{\mathrm n}^m
       +\Dop_{\widehat\kappa^m}^m\widehat v_{\mathrm t}^m],
        \partial_\xi^j\widehat v_{\mathrm n}^m\bigr)_{L^2}.
\end{aligned}
\]
For any smooth periodic $f$, integration by parts yields
\[
\begin{aligned}
 \sum_{j=0}^3(\partial_\xi^j f,\partial_\xi^j\widehat v_{\mathrm n}^m)_{L^2}
 &=\sum_{j=0}^2(\partial_\xi^j f,\partial_\xi^j\widehat v_{\mathrm n}^m)_{L^2}
   -(\partial_\xi^2 f,\partial_\xi^4\widehat v_{\mathrm n}^m)_{L^2},\\
 \tau\left|\sum_{j=0}^3(\partial_\xi^j f,\partial_\xi^j\widehat v_{\mathrm n}^m)_{L^2}\right|
 &\le C\tau\|f\|_{H^2}
       \bigl(\|\widehat v_{\mathrm n}^m\|_{H^3}
             +\|\partial_\xi\widehat v_{\mathrm n}^m\|_{H^3}\bigr)\\
 &\le\frac{\tau}{2(\ell^m)^2}\|\partial_\xi\widehat v_{\mathrm n}^m\|_{H^3}^2
       +C\tau\bigl(\|f\|_{H^2}^2+\|\widehat v_{\mathrm n}^m\|_{H^3}^2\bigr).
\end{aligned}
\]
Taking $f=(\widehat\kappa^m)^2\widehat v_{\mathrm n}^m
              +\Dop_{\widehat\kappa^m}^m\widehat v_{\mathrm t}^m$
and using \eqref{up:eq:P-tame} gives
\begin{equation}\label{up:eq:high-energy}
 \frac12\bigl(\|\widehat v_{\mathrm n}^m\|_{H^3}^2-\|\widehat\kappa^m\|_{H^3}^2\bigr)
       +\frac12\|\widehat\kappa^m+\widehat v_{\mathrm n}^m\|_{H^3}^2 + \frac{\tau}{2(\ell^m)^2}\|\partial_\xi\widehat v_{\mathrm n}^m\|_{H^3}^2
 \le C\tau(1+\|\widehat\kappa^m\|_{H^3}^2+\|\widehat v_{\mathrm n}^m\|_{H^3}^2).
\end{equation}
Since $(\widehat\kappa^m,\ell^m)\in\mathcal U(R)$,
\eqref{time:local-set} and \eqref{eq:integer-sobolev-norm} imply
\[
 \|\widehat\kappa^m\|_{H^3}
 \ge\|\widehat\kappa^m\|_{L^2}\ge c_0,
 \qquad 1\le c_0^{-2}\|\widehat\kappa^m\|_{H^3}^2.
\]
Thus \eqref{up:eq:high-energy} yields, for small $\tau$,
\[
 (1-C\tau)\|\widehat v_{\mathrm n}^m\|_{H^3}^2
 \le(1+C\tau)\|\widehat\kappa^m\|_{H^3}^2+C\tau\le\bigl[1+C(1+c_0^{-2})\tau\bigr]\|\widehat\kappa^m\|_{H^3}^2.
\]
Enlarge $C$ to absorb $1+c_0^{-2}$ and take $C\tau\le1/2$. Since
\[
 \left(\frac{1+C\tau}{1-C\tau}\right)^{1/2}
 =\left(1+\frac{2C\tau}{1-C\tau}\right)^{1/2}
 \le(1+4C\tau)^{1/2}\le1+2C\tau,
\]
we obtain
\[
\begin{aligned}
 \|\widehat v_{\mathrm n}^m\|_{H^3}
 &\le\left(\frac{1+C\tau}{1-C\tau}\right)^{1/2}
          \|\widehat\kappa^m\|_{H^3}
 \le(1+2C\tau)\|\widehat\kappa^m\|_{H^3}.
\end{aligned}
\]
Increasing $C$ once more proves \eqref{up:eq:near-contraction}.

We next prove
\begin{equation}\label{up:eq:geo-bound}
 \|\widehat\kappa^{m+1}\|_{H^3}
\le(1+C\tau)\|\widehat v_{\mathrm n}^m\|_{H^3}+C\tau(1+\|\widehat\kappa^m\|_{H^3}).
\end{equation}
The identities \eqref{time:strong-identities} give
\begin{equation}\label{up:eq:metric-cancel}
\begin{aligned}
 \partial_\xi(\widehat{\mathbf t}^m\cdot\partial_s\widehat v^m)
   &=-\ell^m\widehat\kappa^m(\widehat{\mathbf n}^m\cdot\partial_s\widehat v^m),\\
 \partial_\xi|\widehat{\mathbf t}^m+\tau\partial_s\widehat v^m|
   &=\tau\ell^m
     \frac{(\widehat{\mathbf n}^m\cdot\partial_s\widehat v^m)\widehat v_{\mathrm n}^m}
          {|\widehat{\mathbf t}^m+\tau\partial_s\widehat v^m|}.
\end{aligned}
\end{equation}
By \eqref{time:ABCq},
\[
\begin{aligned}
 \widehat{\mathbf t}^m\cdot\partial_s\widehat v^m
 &=(\ell^m)^{-1}\partial_\xi\widehat v_{\mathrm t}^m
       +\widehat\kappa^m\widehat v_{\mathrm n}^m,\\
 \widehat{\mathbf n}^m\cdot\partial_s\widehat v^m
 &=(\ell^m)^{-1}\partial_\xi\widehat v_{\mathrm n}^m
       -\widehat\kappa^m\widehat v_{\mathrm t}^m.
\end{aligned}
\]
By \eqref{time:local-set}, \eqref{up:eq:velocity-uniform-bounds},
and the product estimate,
\[
\begin{aligned}
 \|\widehat{\mathbf t}^m\cdot\partial_s\widehat v^m\|_{L^\infty}
 +\|\widehat{\mathbf n}^m\cdot\partial_s\widehat v^m\|_{L^\infty}
   &\le C,\\
 \|\widehat{\mathbf n}^m\cdot\partial_s\widehat v^m\|_{H^1}
   &\le C(1+\|\widehat v_{\mathrm n}^m\|_{H^2}),\\
 \|\widehat{\mathbf n}^m\cdot\partial_s\widehat v^m\|_{H^2}
   &\le C(1+\|\widehat v_{\mathrm n}^m\|_{H^3}+\|\widehat\kappa^m\|_{H^2}).
\end{aligned}
\]
By \eqref{time:ABCq}, periodicity, and \eqref{up:eq:velocity-uniform-bounds},
\[
\begin{aligned}
 \int_0^1\widehat{\mathbf t}^m\cdot\partial_s\widehat v^m\,d\xi
 &=(\ell^m)^{-1}\bigl[\widehat v_{\mathrm t}^m(1)
                         -\widehat v_{\mathrm t}^m(0)\bigr]
       +\int_0^1\widehat\kappa^m\widehat v_{\mathrm n}^m\,d\xi=\int_0^1\widehat\kappa^m\widehat v_{\mathrm n}^m\,d\xi,\\
 \left|\int_0^1\widehat{\mathbf t}^m\cdot\partial_s\widehat v^m\,d\xi\right|
 &\le\|\widehat\kappa^m\|_{L^2}
       \|\widehat v_{\mathrm n}^m\|_{L^2}\le C.
\end{aligned}
\]
Combining this mean bound with the first identity in
\eqref{up:eq:metric-cancel} gives
\[
\begin{aligned}
 \|\widehat{\mathbf t}^m\cdot\partial_s\widehat v^m\|_{W^{1,\infty}}&\le C,\\
 \|\widehat{\mathbf t}^m\cdot\partial_s\widehat v^m\|_{H^3}
 &\le C\bigl(1+\|\widehat\kappa^m
       (\widehat{\mathbf n}^m\cdot\partial_s\widehat v^m)\|_{H^2}\bigr)\le C(1+\|\widehat v_{\mathrm n}^m\|_{H^3}+\|\widehat\kappa^m\|_{H^2}).
\end{aligned}
\]
By \eqref{time:ABCq}, \eqref{time:local-set}, and
\eqref{up:eq:velocity-uniform-bounds},
\[
 \|\partial_s\widehat v^m\|_{L^\infty}
 \le\|\widehat{\mathbf t}^m\cdot\partial_s\widehat v^m\|_{L^\infty}
    +\|\widehat{\mathbf n}^m\cdot\partial_s\widehat v^m\|_{L^\infty}
 \le C.
\]
Since $|\widehat{\mathbf t}^m|=1$, the reverse triangle inequality gives
\[
 \bigl||\widehat{\mathbf t}^m+\tau\partial_s\widehat v^m|-1\bigr|
 \le\tau|\partial_s\widehat v^m|\le C\tau.
\]
Thus
\[
 1-C\tau\le|\widehat{\mathbf t}^m+\tau\partial_s\widehat v^m|
             \le1+C\tau.
\]
For $C\tau\le1/2$, the second identity in \eqref{up:eq:metric-cancel}
and \eqref{up:eq:velocity-uniform-bounds} yield
\[
 \|\partial_\xi|\widehat{\mathbf t}^m+\tau\partial_s\widehat v^m|\|_{L^\infty}
 \le2\tau\ell_{\max}
       \|\widehat{\mathbf n}^m\cdot\partial_s\widehat v^m\|_{L^\infty}
       \|\widehat v_{\mathrm n}^m\|_{L^\infty}\le C\tau.
\]
Consequently,
\[
 \bigl\||\widehat{\mathbf t}^m+\tau\partial_s\widehat v^m|-1\bigr\|_{W^{1,\infty}}
 \le C\tau.
\]
For $f\in H^{j+1}(\Torus)$, $j=1,2$, Poincar\'e's inequality gives
\[
 \|f\|_{H^{j+1}}^2
 \le\left|\int_0^1 f\,d\xi\right|^2
       +C\|\partial_\xi f\|_{H^j}^2.
\]
Taking
$f=|\widehat{\mathbf t}^m+\tau\partial_s\widehat v^m|-1$,
its mean is $O(\tau)$ by \eqref{time:metric-estimates} and the embedding
$H^\sigma(\Torus)\hookrightarrow L^\infty(\Torus)$. Moreover,
\eqref{time:ABCq}, \eqref{time:local-set}, and
\eqref{up:eq:velocity-uniform-bounds} give uniform $L^\infty$ bounds
for $\widehat{\mathbf n}^m\cdot\partial_s\widehat v^m$ and
$\widehat v_{\mathrm n}^m$. For sufficiently small $\tau$,
\eqref{time:metric-estimates} also gives
$|\widehat{\mathbf t}^m+\tau\partial_s\widehat v^m|\ge\frac12$,
and hence a uniform $L^\infty$ bound for its reciprocal.
The second identity in \eqref{up:eq:metric-cancel} and two applications
of the product estimate therefore yield
\[
\begin{aligned}
 &\bigl\||\widehat{\mathbf t}^m+\tau\partial_s\widehat v^m|-1\bigr\|_{H^{j+1}}\\
 &\quad\le C\left|\int_0^1
   (|\widehat{\mathbf t}^m+\tau\partial_s\widehat v^m|-1)\,d\xi\right|
       +C\|\partial_\xi|\widehat{\mathbf t}^m+\tau\partial_s\widehat v^m|\|_{H^j}\\
 &\quad\le C\tau+C\tau\ell^m
 \left\|\frac{(\widehat{\mathbf n}^m\cdot\partial_s\widehat v^m)
                    \widehat v_{\mathrm n}^m}
                   {|\widehat{\mathbf t}^m+\tau\partial_s\widehat v^m|}\right\|_{H^j}\\
 &\quad\le C\tau\bigl(1
       +\|\widehat{\mathbf n}^m\cdot\partial_s\widehat v^m\|_{H^j}
       +\|\widehat v_{\mathrm n}^m\|_{H^j}
       +\bigl\||\widehat{\mathbf t}^m+\tau\partial_s\widehat v^m|^{-1}\bigr\|_{H^j}\bigr).
\end{aligned}
\]
To control the last term, the chain rule gives
\[
 \bigl\||\widehat{\mathbf t}^m+\tau\partial_s\widehat v^m|^{-1}-1\bigr\|_{H^j}
 \le C\bigl\||\widehat{\mathbf t}^m+\tau\partial_s\widehat v^m|-1\bigr\|_{H^j},
 \qquad j=1,2.
\]
Applying these inequalities first with $j=1$, then with $j=2$, gives
\begin{equation}\label{up:eq:q-high}
\begin{aligned}
 \bigl\||\widehat{\mathbf t}^m+\tau\partial_s\widehat v^m|-1\bigr\|_{H^2}
 &\le C\tau(1+\|\widehat v_{\mathrm n}^m\|_{H^2}),\\
 \bigl\||\widehat{\mathbf t}^m+\tau\partial_s\widehat v^m|-1\bigr\|_{H^3}
 &\le C\tau(1+\|\widehat v_{\mathrm n}^m\|_{H^3}+\|\widehat\kappa^m\|_{H^2}).
\end{aligned}
\end{equation}
For the factor in \eqref{time:exact-curvature}, the chain rule gives
\[
 \bigl\||\widehat{\mathbf t}^m+\tau\partial_s\widehat v^m|^{-3}-1\bigr\|_{H^3}
 \le C\bigl\||\widehat{\mathbf t}^m+\tau\partial_s\widehat v^m|-1\bigr\|_{H^3}.
\]
The decomposition
\[
\begin{aligned}
 &\frac{1+\tau\widehat{\mathbf t}^m\cdot\partial_s\widehat v^m}
       {|\widehat{\mathbf t}^m+\tau\partial_s\widehat v^m|^3}-1=|\widehat{\mathbf t}^m+\tau\partial_s\widehat v^m|^{-3}-1
       +\tau(\widehat{\mathbf t}^m\cdot\partial_s\widehat v^m)
          |\widehat{\mathbf t}^m+\tau\partial_s\widehat v^m|^{-3}
\end{aligned}
\]
together with \eqref{up:eq:q-high} and the product estimate implies
\begin{equation}\label{up:eq:D-high}
\begin{aligned}
 \left\|\frac{1+\tau\widehat{\mathbf t}^m\cdot\partial_s\widehat v^m}
 {|\widehat{\mathbf t}^m+\tau\partial_s\widehat v^m|^3}-1\right\|_{H^3} & \le C\tau(1+\|\widehat v_{\mathrm n}^m\|_{H^3}+\|\widehat\kappa^m\|_{H^2}),\\
 \left\|\frac{1+\tau\widehat{\mathbf t}^m\cdot\partial_s\widehat v^m}
	{|\widehat{\mathbf t}^m+\tau\partial_s\widehat v^m|^3}-1\right\|_{W^{1,\infty}} & \le C\tau.
\end{aligned}
\end{equation}
By \eqref{time:exact-curvature},
\[
 \widehat\kappa^{m+1}\circ\chi^{m+1}\circ(\chi^m)^{-1}
       +\widehat v_{\mathrm n}^m=-\widehat v_{\mathrm n}^m
       \left(\frac{1+\tau\widehat{\mathbf t}^m\cdot\partial_s\widehat v^m}
       {|\widehat{\mathbf t}^m+\tau\partial_s\widehat v^m|^3}-1\right).
\]
The product estimate and \eqref{up:eq:D-high} therefore give
\begin{equation}\label{up:eq:w-high}
\begin{aligned}
 \|\widehat\kappa^{m+1}\circ\chi^{m+1}\circ(\chi^m)^{-1}\|_{H^3}
 &\le(1+C\tau)\|\widehat v_{\mathrm n}^m\|_{H^3}
       +C\tau(1+\|\widehat\kappa^m\|_{H^2}),\\
 \|\widehat\kappa^{m+1}\circ\chi^{m+1}\circ(\chi^m)^{-1}\|_{W^{1,\infty}}
 &\le C.
\end{aligned}
\end{equation}
The derivative of the coordinate change is
\[
 \partial_\xi(\chi^{m+1}\circ(\chi^m)^{-1})
 =\frac{|\widehat{\mathbf t}^m+\tau\partial_s\widehat v^m|}
        {\displaystyle\int_0^1|\widehat{\mathbf t}^m+\tau\partial_s\widehat v^m|\,d\xi}.
\]
Thus \eqref{up:eq:metric-cancel} and \eqref{up:eq:q-high} yield
\[
\begin{aligned}
 &\|\partial_\xi(\chi^{m+1}\circ(\chi^m)^{-1})-1\|_{L^\infty}
 +\|\partial_\xi^2(\chi^{m+1}\circ(\chi^m)^{-1})\|_{L^\infty}
 \le C\tau,\\
 &\|\partial_\xi^3(\chi^{m+1}\circ(\chi^m)^{-1})\|_{L^2}
 \le C\tau(1+\|\widehat v_{\mathrm n}^m\|_{H^2}).
\end{aligned}
\]
Set $g=\chi^{m+1}\circ(\chi^m)^{-1}$. Since
$\partial_\xi g\ge1-C\tau\ge\frac12$, the identities obtained by
differentiating $g\circ g^{-1}=\id$ up to order three, together with a
change of variables, transfer the preceding derivative bounds to
$g^{-1}$ and give
\begin{equation}\label{up:eq:inverse}
\begin{aligned}
 &\|\partial_\xi(\chi^m\circ(\chi^{m+1})^{-1})-1\|_\infty
 +\|\partial_\xi^2(\chi^m\circ(\chi^{m+1})^{-1})\|_\infty\le C\tau,\\
 &\|\partial_\xi^3(\chi^m\circ(\chi^{m+1})^{-1})\|_{L^2}
       \le C\tau(1+\|\widehat v_{\mathrm n}^m\|_{H^2}).
\end{aligned}
\end{equation}
Now set $g=\chi^m\circ(\chi^{m+1})^{-1}$. For a smooth scalar
function $f$, the chain rule up to order three and a change of variables
give, for $0\le j\le3$,
\[
\begin{aligned}
 \|((\partial_\xi^jf)\circ g)(\partial_\xi g)^j\|_{L^2}^2
 &=\int_0^1|\partial_\xi^jf(\zeta)|^2
      [\partial_\xi g(g^{-1}(\zeta))]^{2j-1}\,d\zeta\le(1+C\tau)\|\partial_\xi^jf\|_{L^2}^2.
\end{aligned}
\]
Summing the squared derivative norms and using \eqref{up:eq:inverse} yields
\begin{equation}\label{up:eq:composition}
\begin{aligned}
 \|f\circ\chi^m\circ(\chi^{m+1})^{-1}\|_{H^3}
 &\le(1+C\tau)\|f\|_{H^3}+C\tau\|f\|_{H^2}\\
 &\qquad+C\tau(1+\|\widehat v_{\mathrm n}^m\|_{H^2})\|\partial_\xi f\|_\infty.
\end{aligned}
\end{equation}
Take $f=\widehat\kappa^{m+1}\circ\chi^{m+1}\circ(\chi^m)^{-1}$.
By \eqref{up:eq:w-high}, $\|\partial_\xi f\|_{L^\infty}\le C$.
Using $\|f\|_{H^2}\le\|f\|_{H^3}$ in
\eqref{up:eq:composition} and enlarging $C$ gives
\[
\begin{aligned}
 \|\widehat\kappa^{m+1}\|_{H^3}
 &\le(1+C\tau)\|f\|_{H^3}
       +C\tau(1+\|\widehat v_{\mathrm n}^m\|_{H^2})\\
 &\le(1+C\tau)^2\|\widehat v_{\mathrm n}^m\|_{H^3}
       +C\tau(1+\|\widehat v_{\mathrm n}^m\|_{H^2})+C\tau(1+C\tau)(1+\|\widehat\kappa^m\|_{H^2}).
\end{aligned}
\]
Since $\|\widehat v_{\mathrm n}^m\|_{H^2}
\le\|\widehat v_{\mathrm n}^m\|_{H^3}$ and $\tau\le1$, increasing $C$
once more yields
\[
\begin{aligned}
 \|\widehat\kappa^{m+1}\|_{H^3}
 &\le(1+C\tau)\|\widehat v_{\mathrm n}^m\|_{H^3}
       +C\tau(1+\|\widehat\kappa^m\|_{H^2})\\
 &\le(1+C\tau)\|\widehat v_{\mathrm n}^m\|_{H^3}
       +C\tau(1+\|\widehat\kappa^m\|_{H^3}),
\end{aligned}
\]
which is \eqref{up:eq:geo-bound}.

Combining \eqref{up:eq:near-contraction} and
\eqref{up:eq:geo-bound} gives
\[
 \|\widehat\kappa^{m+1}\|_{H^3}
 \le(1+C\tau)\|\widehat\kappa^m\|_{H^3}+C\tau.
\]
Iteration gives, for $0\le m\le M$,
\[
 \|\widehat\kappa^m\|_{H^3}
 \le e^{CT}\bigl(\|\widehat\kappa^0\|_{H^3}+CT\bigr).
\]
The estimate \eqref{up:eq:near-contraction} also applies to the
auxiliary normal velocity at $m=M$. Hence
\[
 \max_{0\le m\le M}\|\widehat v_{\mathrm n}^m\|_{H^3}
 \le(1+C\tau)\max_{0\le m\le M}\|\widehat\kappa^m\|_{H^3}\le C.
\]
\end{proof}

\subsection{First-order defects and convergence}\label{up:sec:first-order}

The $H^3$ bound improves defects to $O(\tau)$, and
Lemma~\ref{time:prefix-error} bounds the curvature and length errors.
Recovering the marked point, tangent and parametrization proves
Theorem~\ref{up:thm:time}.

\begin{proof}[Proof of Theorem~\ref{up:thm:time}]
The uniform $H^3$ bound and the resolvent equation imply
\begin{equation}\label{up:eq:improved-increments}
\begin{aligned}
 \|\widehat\kappa^m+\widehat v_{\mathrm n}^m\|_{H^1}
   &=\tau\|\Sop_{\widehat\kappa^m}^m\widehat v_{\mathrm n}^m\|_{H^1}\le C\tau,\\
 \|\widehat v_{\mathrm n}^m\|_{H^2}+\|\widehat v_{\mathrm n}^m\|_{H^{\sigma+1}}&\le C.
\end{aligned}
\end{equation}
The composition estimate in the proof of
Lemma~\ref{time:local-estimates} now gives
\begin{equation}\label{up:eq:geometric-increment}
\begin{aligned}
 &\|\widehat\kappa^{m+1}+\widehat v_{\mathrm n}^m\|_{H^\sigma}
       +\|\widehat\kappa^{m+1}+\widehat v_{\mathrm n}^m\|_{H^1}\le C\tau.
\end{aligned}
\end{equation}
The estimate for the geometric Taylor remainder in
\eqref{time:curvature-expansion}, together with
\eqref{up:eq:improved-increments}, gives
\[
 \|d_{\mathrm{geom}}^m\|_{L^2}
 \le C\tau(1+\|\widehat v_{\mathrm n}^m\|_{H^2})\le C\tau.
\]
The local Lipschitz continuity after
\eqref{time:curvature-expansion} yields
\[
 \left\|\widehat v_{\mathrm t}^m
 +\Lop_{\widehat v_{\mathrm n}^m,\ell^m}^{-1}
       \Dop_{\widehat v_{\mathrm n}^m,\ell^m}
       \widehat v_{\mathrm n}^m\right\|_{H^1}
 \le C\|\widehat\kappa^m+\widehat v_{\mathrm n}^m\|_{H^1}
 \le C\tau.
\]
Equation \eqref{time:ABCq} and the preceding tangential-velocity
estimate give
\[
\begin{aligned}
 &\left\|\widehat{\mathbf t}^m\cdot\partial_s\widehat v^m
 +(\ell^m)^{-1}\partial_\xi\left(
 \Lop_{\widehat v_{\mathrm n}^m,\ell^m}^{-1}
 \Dop_{\widehat v_{\mathrm n}^m,\ell^m}
 \widehat v_{\mathrm n}^m\right)
 +(\widehat v_{\mathrm n}^m)^2\right\|_{L^2}\\
 &\qquad\le C\|\widehat\kappa^m+\widehat v_{\mathrm n}^m\|_{H^1}
 \le C\tau.
\end{aligned}
\]
Integration and the Cauchy--Schwarz inequality give
\[
\begin{aligned}
 &\left\|\int_0^\xi\Bigg[
 (\widehat{\mathbf t}^m\cdot\partial_s\widehat v^m)(\zeta)
 +(\ell^m)^{-1}\partial_\zeta\left(
 \Lop_{\widehat v_{\mathrm n}^m,\ell^m}^{-1}
 \Dop_{\widehat v_{\mathrm n}^m,\ell^m}
 \widehat v_{\mathrm n}^m\right)(\zeta)
 +(\widehat v_{\mathrm n}^m(\zeta))^2\right.\\
 &\qquad\left.-\int_0^1\left\{
 (\widehat{\mathbf t}^m\cdot\partial_s\widehat v^m)(\eta)
 +(\ell^m)^{-1}\partial_\eta\left(
 \Lop_{\widehat v_{\mathrm n}^m,\ell^m}^{-1}
 \Dop_{\widehat v_{\mathrm n}^m,\ell^m}
 \widehat v_{\mathrm n}^m\right)(\eta)
 +(\widehat v_{\mathrm n}^m(\eta))^2
 \right\}d\eta\Bigg]d\zeta\right\|_{L^\infty}\\
 &\qquad\le C\|\widehat\kappa^m+\widehat v_{\mathrm n}^m\|_{H^1}
 \le C\tau.
\end{aligned}
\]
Substituting \eqref{time:G0} into the preceding integral and using
\eqref{time:nonlocal-cancellation} give
\[
\begin{aligned}
 &2\left[-(\ell^m)^{-1}\partial_\xi\left(
 \Lop_{\widehat v_{\mathrm n}^m,\ell^m}^{-1}
 \Dop_{\widehat v_{\mathrm n}^m,\ell^m}
 \widehat v_{\mathrm n}^m\right)
 -(\widehat v_{\mathrm n}^m)^2\right]\widehat v_{\mathrm n}^m\\
 &\quad+\left\{-(\ell^m)^{-1}\left[
 \left(\Lop_{\widehat v_{\mathrm n}^m,\ell^m}^{-1}
 \Dop_{\widehat v_{\mathrm n}^m,\ell^m}
 \widehat v_{\mathrm n}^m\right)(\xi)
 -\left(\Lop_{\widehat v_{\mathrm n}^m,\ell^m}^{-1}
 \Dop_{\widehat v_{\mathrm n}^m,\ell^m}
 \widehat v_{\mathrm n}^m\right)(0)\right]\right.\\
 &\qquad\left.-\int_0^\xi\left(
 (\widehat v_{\mathrm n}^m(\zeta))^2
 -\int_0^1(\widehat v_{\mathrm n}^m(\eta))^2\,d\eta
 \right)d\zeta\right\}\partial_\xi\widehat v_{\mathrm n}^m\\
 &\quad=-(\widehat v_{\mathrm n}^m)^3
 -\Kop_{\widehat v_{\mathrm n}^m}^m\widehat v_{\mathrm n}^m
 -\mathcal N(\widehat v_{\mathrm n}^m,\ell^m).
\end{aligned}
\]
Combining the preceding $L^2$ and $L^\infty$ estimates with the
uniform $W^{1,\infty}$ bound for $\widehat v_{\mathrm n}^m$ gives
\[
\begin{aligned}
 &\Bigg\|2(\widehat{\mathbf t}^m\cdot\partial_s\widehat v^m)
       \widehat v_{\mathrm n}^m+\left[\int_0^\xi\left(
 (\widehat{\mathbf t}^m\cdot\partial_s\widehat v^m)(\zeta)
 -\int_0^1(\widehat{\mathbf t}^m\cdot\partial_s\widehat v^m)(\eta)
       \,d\eta\right)d\zeta\right]\partial_\xi\widehat v_{\mathrm n}^m\\
 &\quad+(\widehat v_{\mathrm n}^m)^3
 +\Kop_{\widehat v_{\mathrm n}^m}^m\widehat v_{\mathrm n}^m
 +\mathcal N(\widehat v_{\mathrm n}^m,\ell^m)\Bigg\|_{L^2}\\
 &\qquad\le C\|\widehat\kappa^m+\widehat v_{\mathrm n}^m\|_{H^1}
 \le C\tau.
\end{aligned}
\]
Hence the curvature defect in \eqref{time:state-remainder} satisfies
\begin{equation}\label{up:eq:curvature-defect}
 \|d_\kappa^m\|_{L^2}
 \le C\tau(1+\|\widehat v_{\mathrm n}^m\|_{H^2})
      +C\|\widehat\kappa^m+\widehat v_{\mathrm n}^m\|_{H^1}
 \le C\tau.
\end{equation}
For the length defect, writing \eqref{time:exact-length} in the
normalized coordinate and comparing it with the second identity in
\eqref{time:scalar-diffusion} give
\[
\begin{aligned}
 d_\ell^m
 &=\ell^m\int_0^1
    (\widehat\kappa^m+\widehat v_{\mathrm n}^m)
       \widehat v_{\mathrm n}^m\,d\xi+\tau\ell^m\int_0^1
 \frac{(\widehat{\mathbf n}^m\cdot\partial_s\widehat v^m)^2}
 {|\widehat{\mathbf t}^m+\tau\partial_s\widehat v^m|
       +(1+\tau\widehat{\mathbf t}^m\cdot\partial_s\widehat v^m )}\,d\xi.
\end{aligned}
\]
For small $\tau$, the first term in the denominator is at least
$1/2$ by \eqref{time:metric-estimates}, while the second is at least
$1/2$ by \eqref{time:ABCq} and
\eqref{up:eq:velocity-uniform-bounds}. These estimates and
\eqref{up:eq:improved-increments} imply
\begin{equation}\label{up:eq:length-defect}
 |d_\ell^m|
 \le C\|\widehat\kappa^m+\widehat v_{\mathrm n}^m\|_{L^2}
          \|\widehat v_{\mathrm n}^m\|_{L^2}
      +C\tau\|\widehat{\mathbf n}^m\cdot
                    \partial_s\widehat v^m\|_{L^2}^2
 \le C\tau.
\end{equation}
Combining the normal resolvent at level $m+1$ with
\eqref{time:state-remainder}, and using the evenness of $\Sop$ in its
curvature coefficient, gives
\[
\begin{aligned}
 \frac{\widehat v_{\mathrm n}^{m+1}-\widehat v_{\mathrm n}^m}{\tau}
 +\Sop_{\widehat v_{\mathrm n}^m,\ell^m}
       \widehat v_{\mathrm n}^{m+1}
 ={}&(\widehat v_{\mathrm n}^m)^3
 +\Kop_{\widehat v_{\mathrm n}^m}^m\widehat v_{\mathrm n}^m+\mathcal N(\widehat v_{\mathrm n}^m,\ell^m)
 +\widetilde d_{\mathrm n}^m,
\end{aligned}
\]
where
\[
 \widetilde d_{\mathrm n}^m
 =-d_\kappa^m
 +\bigl(
 \Sop_{\widehat v_{\mathrm n}^m,\ell^m}
 -\Sop_{\widehat\kappa^{m+1},\ell^{m+1}}
 \bigr)\widehat v_{\mathrm n}^{m+1}.
\]
The coefficient estimates, \eqref{up:eq:improved-increments},
\eqref{up:eq:geometric-increment}, \eqref{up:eq:curvature-defect},
\eqref{up:eq:length-defect}, and \eqref{time:scalar-diffusion} imply
\[
\begin{aligned}
 \|\widetilde d_{\mathrm n}^m\|_{L^2}
 &\le \|d_\kappa^m\|_{L^2}
 +C\left(
 \|\widehat\kappa^{m+1}+\widehat v_{\mathrm n}^m\|_{H^1}
 +|\ell^{m+1}-\ell^m|
 \right)\|\widehat v_{\mathrm n}^{m+1}\|_{H^2}\le C\tau.
\end{aligned}
\]
Moreover,
\[
\begin{aligned}
 \|\widehat v_{\mathrm n}^{m+1}-\widehat v_{\mathrm n}^m\|_{L^2}
 &\le
 \|\widehat v_{\mathrm n}^{m+1}+\widehat\kappa^{m+1}\|_{L^2}
 +\|\widehat\kappa^{m+1}+\widehat v_{\mathrm n}^m\|_{L^2}\le C\tau
\end{aligned}
\]
by \eqref{up:eq:improved-increments} and
\eqref{up:eq:geometric-increment}. Since
\[
 \Sop_{\widehat v_{\mathrm n}^m,\ell^m}
 =-(\ell^m)^{-2}\partial_\xi^2
   +(\widehat v_{\mathrm n}^m)^2
   +\Kop_{\widehat v_{\mathrm n}^m}^m,
\]
comparison with \eqref{time:scalar-diffusion} gives
\[
 d_{\mathrm n}^m
 =\widetilde d_{\mathrm n}^m
 -\left((\widehat v_{\mathrm n}^m)^2
       +\Kop_{\widehat v_{\mathrm n}^m}^m\right)
  (\widehat v_{\mathrm n}^{m+1}-\widehat v_{\mathrm n}^m).
\]
The uniform operator bounds therefore yield
\begin{equation}\label{up:eq:diffusion}
 \|d_{\mathrm n}^m\|_{L^2}\le C\tau.
\end{equation}
Since $E^0=O(\tau)$, Lemma~\ref{time:prefix-error}, together with
\eqref{up:eq:length-defect} and \eqref{up:eq:diffusion}, gives
$\max_{m\le M}E^m\le C\tau$.
Recover the curvature using
\[
\begin{aligned}
 \widehat\kappa^{m+1}-\widehat\kappa(t_{m+1})=(\widehat\kappa^{m+1}+\widehat v_{\mathrm n}^m)
       -(\widehat v_{\mathrm n}^m+\widehat\kappa(t_m))+\widehat\kappa(t_m)-\widehat\kappa(t_{m+1}).
\end{aligned}
\]
Equation \eqref{up:eq:geometric-increment} proves the required
$H^\sigma$ curvature error.

It remains to compare the marked points and their unit tangents.
The position update and tangent normalization give
\begin{align}
 x^{m+1}(0)
 &=x^m(0)+\tau\bigl(v_{\mathrm n}^m(0)\mathbf n^m(0)
                    +v_{\mathrm t}^m(0)\mathbf t^m(0)\bigr),
 \label{time:point-update}\\
 \mathbf t^{m+1}(0)
 &=\frac{(\mathbf t^m+\tau\partial_s v^m)(0)}
          {|(\mathbf t^m+\tau\partial_s v^m)(0)|}.
 \label{time:angle-update}
\end{align}
For the exact marked point and its unit
tangent, \eqref{eq:exact-marked-point} and the Frenet identities give
\begin{subequations}\label{time:exact-pose}
\begin{align}
 \partial_t\widehat X(0,t)
 &=(-\widehat\kappa\widehat{\mathbf n}+\widehat\alpha\widehat{\mathbf t})(0,t),
 \label{time:exact-point}\\
 \partial_t\widehat{\mathbf t}(0,t)
 &=-(\partial_s\widehat\kappa+\widehat\kappa\widehat\alpha)(0,t)
                                  \widehat{\mathbf n}(0,t).
 \label{time:exact-angle}
\end{align}
\end{subequations}
Their initial point and unit tangent agree with $x^0(0)$ and $\mathbf t^0(0)$.

Since $\sigma>3/2$, the embedding
$H^\sigma(\Torus)\hookrightarrow C^1(\Torus)$ and the preceding
estimates give
\[
\begin{aligned}
 &\|\widehat v_{\mathrm n}^m+\widehat\kappa(t_m)\|_{C^1}
 +\|\widehat\kappa^m-\widehat\kappa(t_m)\|_{C^1}\\
 &\qquad\le C\left(
 \|\widehat v_{\mathrm n}^m+\widehat\kappa(t_m)\|_{H^\sigma}
 +\|\widehat\kappa^m-\widehat\kappa(t_m)\|_{H^\sigma}
 \right)
 \le C\tau.
\end{aligned}
\]
We claim that the map
\[
 (\eta,w,\ell)\longmapsto
 \Lop_{\eta,\ell}^{-1}\Dop_{\eta,\ell}w
\]
is locally Lipschitz into $H^{\sigma+1}$.
Since $\sigma-1>1/2$, $H^{\sigma-1}(\Torus)$ is an algebra, and
the formula in \eqref{time:LD} gives, for
$(\eta,\ell),(\widetilde\eta,\widetilde\ell)\in\mathcal U(R)$ and
$w,\widetilde w$ in a bounded subset of $H^\sigma(\Torus)$,
\[
\begin{aligned}
 \|\Dop_{\eta,\ell}w
       -\Dop_{\widetilde\eta,\widetilde\ell}\widetilde w
   \|_{H^{\sigma-1}}\le C_R\bigl(
 \|\eta-\widetilde\eta\|_{H^\sigma}
 +\|w-\widetilde w\|_{H^\sigma}
 +|\ell-\widetilde\ell|\bigr).
\end{aligned}
\]
Uniform coercivity from Lemma~\ref{time:operator-bounds} and periodic
elliptic regularity imply
$\|\Lop_{\eta,\ell}^{-1}h\|_{H^{\sigma+1}}
 \le C_R\|h\|_{H^{\sigma-1}}$.
The inverse-operator identity gives directly
\[
 \Lop_{\eta,\ell}^{-1}\Dop_{\eta,\ell}w
 -\Lop_{\widetilde\eta,\widetilde\ell}^{-1}
       \Dop_{\widetilde\eta,\widetilde\ell}\widetilde w=\Lop_{\eta,\ell}^{-1}
 \bigl(\Dop_{\eta,\ell}w
       -\Dop_{\widetilde\eta,\widetilde\ell}\widetilde w\bigr)+\Lop_{\eta,\ell}^{-1}
 \bigl(\Lop_{\widetilde\eta,\widetilde\ell}
       -\Lop_{\eta,\ell}\bigr)
 \Lop_{\widetilde\eta,\widetilde\ell}^{-1}
       \Dop_{\widetilde\eta,\widetilde\ell}\widetilde w.
\]
The preceding estimates and \eqref{time:LD} therefore yield
\[
 \|\Lop_{\eta,\ell}^{-1}\Dop_{\eta,\ell}w
 -\Lop_{\widetilde\eta,\widetilde\ell}^{-1}
       \Dop_{\widetilde\eta,\widetilde\ell}\widetilde w
 \|_{H^{\sigma+1}}
 \le C_R\bigl(
 \|\eta-\widetilde\eta\|_{H^\sigma}
 +\|w-\widetilde w\|_{H^\sigma}
 +|\ell-\widetilde\ell|\bigr).
\]
Applying this estimate to
$(\widehat\kappa^m,\widehat v_{\mathrm n}^m,\ell^m)$ and
$(\widehat\kappa(t_m),-\widehat\kappa(t_m),|\Gamma(t_m)|)$, and using
\eqref{time:normal-resolvent} and \eqref{time:alphaM}, gives
$\|\widehat v_{\mathrm t}^m-\widehat\alpha(t_m)\|_{H^{\sigma+1}}
\le C\tau$.
Since the discrete and exact arclength derivatives are
$(\ell^m)^{-1}\partial_\xi$ and
$|\Gamma(t_m)|^{-1}\partial_\xi$, respectively, the preceding
$C^1$ and length estimates, together with the product estimates, yield
\[
\begin{aligned}
 &|\widehat v_{\mathrm n}^m(0)+\widehat\kappa(0,t_m)|
 +|\widehat v_{\mathrm t}^m(0)-\widehat\alpha(0,t_m)|\\
 &\quad+\left|
 (\partial_s\widehat v_{\mathrm n}^m
       -\widehat\kappa^m\widehat v_{\mathrm t}^m)(0)
 +(\partial_s\widehat\kappa
       +\widehat\kappa\widehat\alpha)(0,t_m)
 \right|
 \le C\tau.
\end{aligned}
\]
The velocity derivative components are
uniformly bounded. Expanding the normalization in
\eqref{time:angle-update} and using \eqref{time:ABCq} yields
\[
 \mathbf t^{m+1}(0)-\mathbf t^m(0)
 =\tau(\partial_s v_{\mathrm n}^m-
                         \kappa^m v_{\mathrm t}^m)(0)\mathbf n^m(0)
   +O(\tau^2).
\]
Since $\widehat{\mathbf n}^m=J\widehat{\mathbf t}^m$ and
$\widehat{\mathbf n}=J\widehat{\mathbf t}$, the normal and tangent
errors at the marked point have the same norm.
Comparison of \eqref{time:point-update}--\eqref{time:angle-update}
with \eqref{time:exact-pose}, followed by discrete Gronwall, gives
\begin{equation}\label{time:pose-error}
 \max_{m\tau\le T}
 \bigl(|\widehat x^m(0)-\widehat X(0,t_m)|
       +|\widehat{\mathbf t}^m(0)-\widehat{\mathbf t}(0,t_m)|\bigr)
 \le C\tau.
\end{equation}

The tangent fields satisfy the Frenet equations
\begin{equation}\label{time:angle-reconstruction}
 \begin{aligned}
  \partial_\xi\widehat{\mathbf t}^m&=-\ell^m\widehat\kappa^m J\widehat{\mathbf t}^m,\\
  \partial_\xi\widehat{\mathbf t}(t_m)
    &=-|\Gamma(t_m)|\widehat\kappa(t_m)
                         J\widehat{\mathbf t}(t_m).
 \end{aligned}
\end{equation}
Subtracting these equations and applying Gronwall on $[0,1]$ bounds
the uniform tangent error by the tangent error at $\xi=0$ in
\eqref{time:pose-error} and the curvature and length errors.
The positions are recovered from
\begin{equation}\label{time:curve-reconstruction}
 \widehat x^m(\xi)=\widehat x^m(0)+\ell^m\int_0^\xi\widehat{\mathbf t}^m(\zeta)\,d\zeta.
\end{equation}
The corresponding position formula for $\widehat X(t_m)$ and the established
curvature and length errors give the $C^1$ estimate in
\eqref{up:eq:first-order}.
\end{proof}

\subsection{Higher derivatives and material coordinates}\label{reg:sec:material}

We propagate higher curvature derivatives and control coordinate
changes to supply the material regularity used in
Section~\ref{sec:consistency}.
We use the tame product estimate stated in the proof of
Lemma~\ref{up:lem:resolvent} at arbitrary fixed integer orders.
The tangential equation is the second identity in
\eqref{time:normal-resolvent}, with the operator definitions
\eqref{time:LD}. We combine it with the normal equation
\eqref{up:eq:scalar}, the
metric cancellations \eqref{up:eq:metric-cancel}, and the exact
curvature update \eqref{time:exact-curvature}.

\begin{lemma}\label{reg:lem:regularity}
For every fixed integer $r\ge3$,
\[
 \max_{m\le M}\|\widehat\kappa^m\|_{H^r}
 +\max_{m<M}(\|\widehat v_{\mathrm n}^m\|_{H^r}+\|\widehat v_{\mathrm t}^m\|_{H^{r+1}})\le C_r.
\]
For every fixed integer $r\ge1$,
\begin{equation}\label{reg:eq:material-high}
 \max_{m\le M}\|x^m\|_{W^{r+1,\infty}}
 +\max_{m<M}\biggl(
   \|v^m\|_{W^{r,\infty}}+\bigl\|| (x^m)'|v_{\mathrm n}^m\mathbf n^m\bigr\|_{W^{r,\infty}}
                  \biggr)\le C_r.
\end{equation}
Also, $0<c\le|(x^m)'|\le C$ for $0\le m\le M$.
All bounds are uniform in $m$ and sufficiently small $\tau$, with
the time-step threshold allowed to depend on $r$.
\end{lemma}

\begin{proof}
The $r=3$ case for $\widehat\kappa^m,\widehat v_{\mathrm n}^m$ follows from
Lemma~\ref{up:lem:resolvent}. The operator definitions
\eqref{time:LD} show that the right-hand side of the tangential
equation is then uniformly bounded in $H^2$.
Uniform coercivity from Lemma~\ref{time:operator-bounds} and periodic
elliptic regularity give $\|\widehat v_{\mathrm t}^m\|_{H^4}\le C$.
Induct on $r\ge4$.
The preceding level bounds $\widehat\kappa^m,\widehat v_{\mathrm n}^m$ in $H^{r-1}$
and $\widehat v_{\mathrm t}^m$ in $H^r$. Thus
\begin{equation}\label{reg:eq:P-high}
\begin{aligned}
 &\|(\widehat\kappa^m)^2\widehat v_{\mathrm n}^m+\Dop_{\widehat\kappa^m}^m\widehat v_{\mathrm t}^m\|_{H^{r-1}}
 \le C_r(1+\|\widehat\kappa^m\|_{H^r}).
\end{aligned}
\end{equation}
Only the highest derivative of
$(\partial_\xi\widehat\kappa^m)\widehat v_{\mathrm t}^m$ requires the current curvature
norm. Repeating the calculation leading to \eqref{up:eq:high-energy},
now with the derivative sum through order $r$, and integrating the
highest-order term once by parts gives
\[
 \frac12\bigl(\|\widehat v_{\mathrm n}^m\|_{H^r}^2-\|\widehat\kappa^m\|_{H^r}^2\bigr)
       +\frac12\|\widehat\kappa^m+\widehat v_{\mathrm n}^m\|_{H^r}^2+\frac{\tau}{2(\ell^m)^2}\|\partial_\xi\widehat v_{\mathrm n}^m\|_{H^r}^2
 \le C_r\tau(1+\|\widehat\kappa^m\|_{H^r}^2+\|\widehat v_{\mathrm n}^m\|_{H^r}^2).
\]
Apply the same argument used to derive
\eqref{up:eq:near-contraction} from \eqref{up:eq:high-energy}:
move the $C_r\tau\|\widehat v_{\mathrm n}^m\|_{H^r}^2$ term to the
left, use $\|\widehat\kappa^m\|_{L^2}\ge c_0$ to absorb the constant
term, and enlarge the generic constant $C_r$. This yields
\begin{equation}\label{reg:eq:high-near}
 \|\widehat v_{\mathrm n}^m\|_{H^r}\le(1+C_r\tau)\|\widehat\kappa^m\|_{H^r}.
\end{equation}

By \eqref{time:ABCq}, the lower-level estimates bound
$\widehat{\mathbf n}^m\cdot\partial_s\widehat v^m$ in $H^{r-2}$ and give an $H^{r-1}$ bound
$C_r(1+\|\widehat v_{\mathrm n}^m\|_{H^r})$.
Moreover, \eqref{time:ABCq} and periodicity give
\[
 \int_0^1\widehat{\mathbf t}^m\cdot\partial_s\widehat v^m\,d\xi
 =\int_0^1\widehat\kappa^m\widehat v_{\mathrm n}^m\,d\xi.
\]
The right-hand side is uniformly bounded. Hence the first identity in
\eqref{up:eq:metric-cancel} and Poincar\'e's inequality give the first
two estimates below, while the second identity gives the last two:
\[
\begin{aligned}
 \|\widehat{\mathbf t}^m\cdot\partial_s\widehat v^m\|_{H^{r-1}}&\le C_r,\\
 \|\widehat{\mathbf t}^m\cdot\partial_s\widehat v^m\|_{H^r}
   &\le C_r(1+\|\widehat v_{\mathrm n}^m\|_{H^r}),\\
 \bigl\||\widehat{\mathbf t}^m+\tau\partial_s\widehat v^m|-1\bigr\|_{H^{r-1}}
   &\le C_r\tau,\\
 \bigl\||\widehat{\mathbf t}^m+\tau\partial_s\widehat v^m|-1\bigr\|_{H^r}
   &\le C_r\tau(1+\|\widehat v_{\mathrm n}^m\|_{H^r}).
\end{aligned}
\]
For the metric estimates, the induction starts from
\eqref{up:eq:q-high} and
$|\widehat{\mathbf t}^m+\tau\partial_s\widehat v^m|\ge1/2$.
At each successive order, the reciprocal estimate uses only the
already controlled lower-order metric norm.

The reciprocal and product argument used in deriving
\eqref{up:eq:D-high}, now at orders $r-1$ and $r$, turns these bounds
into the corresponding bounds for the coefficient in
\eqref{time:exact-curvature}. Applying the product estimate there gives
\[
\begin{aligned}
 \|\widehat\kappa^{m+1}\circ\chi^{m+1}\circ(\chi^m)^{-1}\|_{H^r}
   &\le(1+C_r\tau)\|\widehat v_{\mathrm n}^m\|_{H^r}+C_r\tau,\\
 \|\widehat\kappa^{m+1}\circ\chi^{m+1}\circ(\chi^m)^{-1}\|_{H^{r-1}}
   &\le C_r.
\end{aligned}
\]
The lower-order metric bound gives
$\|\chi^{m+1}\circ(\chi^m)^{-1}-\id\|_{H^r}\le C_r\tau$.
Successive differentiation of the identity obtained by composing this
map with its inverse yields
\[
 \|\partial_\xi(\chi^m\circ(\chi^{m+1})^{-1})-1\|_\infty
 +\sum_{j=2}^{r-1}\|\partial_\xi^{j}(\chi^m\circ(\chi^{m+1})^{-1})\|_\infty+\|\partial_\xi^{r}(\chi^m\circ(\chi^{m+1})^{-1})\|_{L^2}
 \le C_r\tau.
\]
For each derivative of a composition, its leading term contains the
same derivative of the outer function and only first derivatives of
the inner map. Its norm changes by a factor $1+C_r\tau$.
Every other term contains a higher derivative of the inner map.
Apply this chain rule to the curvature composition above and the
map $\chi^m\circ(\chi^{m+1})^{-1}$. The other terms are $O(\tau)$ in
$L^2$, using the uniform $H^{r-1}$ bound for
$\widehat\kappa^{m+1}\circ\chi^{m+1}\circ(\chi^m)^{-1}$.
The term with the $r$th derivative of the inner map is controlled by
the $L^\infty$ norm of the first derivative of this curvature composition.
Summing in the same derivative-sum norm gives
\[
\begin{aligned}
 \|\widehat\kappa^{m+1}\|_{H^r}
 &\le(1+C_r\tau)
 \|\widehat\kappa^{m+1}\circ\chi^{m+1}\circ(\chi^m)^{-1}\|_{H^r}
 +C_r\tau\\
 &\le(1+C_r\tau)\|\widehat\kappa^m\|_{H^r}+C_r\tau,
\end{aligned}
\]
where the last inequality uses the first estimate above and
\eqref{reg:eq:high-near}. Since $\widehat\kappa^0$ is smooth and
$M\tau=T$, discrete Gronwall gives
$\max_{m\le M}\|\widehat\kappa^m\|_{H^r}\le C_r$.
Equation~\eqref{reg:eq:high-near} then bounds
$\widehat v_{\mathrm n}^m$ in $H^r$. The right-hand side of the
tangential equation is consequently bounded in $H^{r-1}$, and periodic
elliptic regularity gives
$\|\widehat v_{\mathrm t}^m\|_{H^{r+1}}\le C_r$.
This closes the induction.

For \eqref{reg:eq:material-high}, apply the preceding estimates
through order $\max(3,r+1)$.
Sobolev embedding and \eqref{time:arclength-change} give
\[
 \|\chi^{m+1}\circ(\chi^m)^{-1}-\id\|_{W^{r+1,\infty}}
 \le C_r\tau.
\]
Using
$\chi^{m+1}=[\chi^{m+1}\circ(\chi^m)^{-1}]\circ\chi^m$,
the derivative of $\chi^m$ is a product of successive factors between
$1-C_r\tau$ and $1+C_r\tau$. Since $\chi^0=\id$ and $M\tau=T$,
this bounds $(\chi^m)'$ above and below by positive constants.
At derivative orders $2\le j\le r+1$, the chain rule and the bounds already
proved at lower orders give
\[
 \|(\chi^{m+1})^{(j)}\|_\infty
 \le(1+C_r\tau)\|(\chi^m)^{(j)}\|_\infty+C_r\tau.
\]
Discrete Gronwall, applied successively in $j$, gives uniform bounds
for the derivatives of $\chi^m$ through order $r+1$.
The positive lower bound for $(\chi^m)'$ and successive differentiation
of $\chi^m\circ(\chi^m)^{-1}=\id$ give the same bounds for
$(\chi^m)^{-1}$.

Theorem~\ref{up:thm:time} bounds $\widehat x^m$ in $L^\infty$.
Together with $\partial_\xi\widehat x^m=\ell^m\widehat{\mathbf t}^m$
and the Frenet identities \eqref{eq:orientation}, the curvature bounds
control the required $\xi$ derivatives of $\widehat x^m$ and the
normalized frame. Finally, the definitions give
\[
\begin{aligned}
 x^m&=\widehat x^m\circ\chi^m,\\
 v^m&=\bigl(\widehat v_{\mathrm n}^m\widehat{\mathbf n}^m
             +\widehat v_{\mathrm t}^m\widehat{\mathbf t}^m\bigr)
             \circ\chi^m,\\
 |(x^m)'|v_{\mathrm n}^m\mathbf n^m
 &=\ell^m(\chi^m)'
   \bigl(\widehat v_{\mathrm n}^m\widehat{\mathbf n}^m\bigr)
             \circ\chi^m,\\
 |(x^m)'|&=\ell^m(\chi^m)'.
\end{aligned}
\]
The chain rule, the length bounds, and the coordinate bounds prove
\eqref{reg:eq:material-high} and $0<c\le|(x^m)'|\le C$.
\end{proof}

\section{Consistency and linearized stability}\label{sec:consistency}

We establish the consistency and linearized stability estimates used in
the nonlinear argument of Section~\ref{sec:nonlinear}.

\subsection{Coercivity and Gauss--Lobatto consistency}

For a regular finite element input $\widetilde x_h\in V_h^2$, let
$a_{\widetilde x_h}$, $b_{\widetilde x_h,h}$, $\mu_i$, and $\omega_i$
denote the quantities in \eqref{eq:stiffness} and
\eqref{eq:nodal-mass}--\eqref{eq:normal-form} with $x_h^m$ replaced by
$\widetilde x_h$. Thus $a_{x_h^m}=a_h^m$ and
$b_{x_h^m,h}=b_h^m$. For later use, set
\[
 p_i=\sum_{K\ni i}w_{K,i}(\widetilde x_h|_K)'(\rho_i).
\]
Equations~\eqref{eq:discrete-geometry} and
\eqref{eq:averaged-normal} give $\omega_i=Jp_i/\mu_i$.
We suppress the dependence of these quantities on the input.

For a smooth regular input $x:\Torus\to\R^2$, write
$\kappa,\mathbf n,\mathbf t$ for its geometric fields.
For $u,\phi\in H^1(\Torus;\R^2)$, define
\[
 a_x(u,\phi)=\int_\T |x'|^{-1}u'\cdot\phi',\qquad
 b_x(u,\phi)=\int_\T |x'|(u\cdot \mathbf n)(\phi\cdot \mathbf n).
\]
With this notation, the step operators introduced in
\eqref{eq:step-operators} are characterized by
\eqref{eq:position-eliminated} and \eqref{time:bgn}. Set
$v=(\F_\tau(x)-x)/\tau$.
For $x=x^m$, Theorem~\ref{time:theorem} and
Lemma~\ref{reg:lem:regularity} provide the uniform coefficient bounds
used below. The normal Gram matrices
$\int_\T |(x^m)'|\mathbf n^m\otimes\mathbf n^m\,d\rho$ are uniformly
positive definite. Indeed, the normals of a regular closed curve span
$\R^2$, and compactness together with \eqref{up:eq:first-order}
transfers the corresponding lower bound from the exact flow to $x^m$.

Consequently, there is a fixed $\W$ neighborhood of the parametrizations $x^m$
on which, uniformly in small $h$ and for every $u_h\in V_h^2$,
\begin{equation}\label{un:eq:coercivity}
 a_{\widetilde x_h}(u_h,u_h)+b_{\widetilde x_h,h}(u_h,u_h)\ge c\|u_h\|_{H^1}^2.
\end{equation}
Indeed, the stiffness controls $u_h$ minus its parameter mean. The mass
controls constant vectors through the positivity of
$\sum_i\mu_i\omega_i\otimes\omega_i$. This matrix is close to the
smooth normal Gram matrix when $\widetilde x_h$ is close to
$x^m$ in $\W$ and $h$ is small. The upper bounds for the forms and
Poincar\'e's inequality complete the estimate. The same argument
applies to the continuous form. Hence
$a_{\widetilde x_h}+\tau^{-1}b_{\widetilde x_h,h}$ and
$a_x+\tau^{-1}b_x$ are uniformly coercive for $0<\tau\le1$.
Section~\ref{sec:nonlinear} proves inductively that $x_h^m$ remains in
these input neighborhoods.

For $K\in\mathcal T_h$, set $h_K=|K|$, let $I_K$ denote interpolation
at its Gauss--Lobatto nodes, and write
\[
 Q_Kg=\sum_{i\in K}w_{K,i}g(\rho_i)
\]
for the corresponding quadrature. Both operators act componentwise on
vector and matrix fields.

The following estimate compares one finite element update with the
interpolated smooth update.
\begin{lemma}\label{un:lem:defect}
For any smooth input $x$ whose velocity $v$ satisfies
\eqref{reg:eq:material-high},
\begin{equation}\label{un:eq:defect}
 \|\F_{h,\tau}(I_hx)-I_h\F_\tau(x)\|_{H^1}\le Ch^{k+1}.
\end{equation}
The constant is uniform in $h,\tau$.
\end{lemma}
\begin{proof}
Let $\Pi_{k-1}$ be the unweighted
$L^2(K)$ projection onto $P^{k-1}(K)$, applied componentwise to vectors.
For a smooth scalar or vector function $f$ on $K$, interpolation
gives $\|(I_Kf-f)'\|_\infty\le Ch_K^k$, while the
Gauss--Lobatto nodes give
\begin{equation}\label{un:eq:GL}
 \|\Pi_{k-1}\partial_\rho(I_Kf-f)\|_{\infty,K}\le Ch_K^{k+1},
 \qquad |Q_K\partial_\rho(I_Kf-f)|\le Ch_K^{k+2},
\end{equation}
To prove this, let $T_{k+1}f$ be the degree $k+1$ Taylor polynomial of
$f$ on $K$. On the reference interval $[-1,1]$, the Gauss--Lobatto
nodal polynomial satisfies
\[
 \prod_{j=0}^{k}(\zeta-\zeta_j)
 =c_k(1-\zeta^2)\partial_\zeta P_k(\zeta),
 \qquad
 \partial_\zeta\bigl[(1-\zeta^2)\partial_\zeta P_k\bigr]
 =-k(k+1)P_k.
\]
After scaling back to $K$, it follows that
\[
 \begin{aligned}
 \partial_\rho(I_Kf-f)
 &=\partial_\rho(I_KT_{k+1}f-T_{k+1}f)
   +\partial_\rho\bigl[I_K(f-T_{k+1}f)-(f-T_{k+1}f)\bigr]\\
 &=c_{K,f}P_k\circ F_K^{-1}+r_K.
 \end{aligned}
\]
Here $F_K:[-1,1]\to K$ is the affine element map and
\[
 \|r_K\|_{\infty,K}
 \le Ch_K^{k+1}\|f\|_{W^{k+2,\infty}(K)}.
\]
Since
\[
 \Pi_{k-1}(P_k\circ F_K^{-1})=0,
 \qquad
 Q_K(P_k\circ F_K^{-1})
 =\int_KP_k\circ F_K^{-1}=0,
\]
fixed-degree projection stability and $\sum_{i\in K}w_{K,i}=h_K$ give
\[
 \|\Pi_{k-1}\partial_\rho(I_Kf-f)\|_{\infty,K}
 \le C\|r_K\|_{\infty,K},
 \qquad
 |Q_K\partial_\rho(I_Kf-f)|
 \le h_K\|r_K\|_{\infty,K}.
\]
This proves \eqref{un:eq:GL}, including $k=1$.

For a smooth vector function $f:K\to\R^2$, a smooth matrix
coefficient $B:K\to\R^{2\times2}$ and a test $\phi_h\in V_h^2$,
freezing $B$ elementwise therefore gives
\begin{align}
 \left|\int_K B\partial_\rho(I_Kf-f)\cdot\phi_h'\right|
 &\le Ch_K^{k+1}\|\phi_h'\|_{L^1(K)},\label{un:eq:GLflux}\\
 |Q_K(B\partial_\rho(I_Kf-f)\cdot\phi_h)|
 &\le Ch_K^{k+1}
       (\|\phi_h\|_{L^1(K)}+\|\phi_h'\|_{L^1(K)}).
 \label{un:eq:GLforce}
\end{align}
For the second estimate, subtract $B(a)\phi_h(a)$ at an endpoint
and use \eqref{un:eq:GL}, coefficient oscillation $O(h_K)$ and
$\|\phi_h-\phi_h(a)\|_\infty\le\|\phi_h'\|_{L^1(K)}$.
The endpoint value is controlled by
$h_K|\phi_h(a)|\le C(\|\phi_h\|_{L^1(K)}
+h_K\|\phi_h'\|_{L^1(K)})$.
For a smooth scalar load $f$ and a scalar test $\phi_h\in V_h$,
the same splitting of the test, followed by
Taylor subtraction of degree $k$ in its constant part and degree
$k-1$ in its oscillating part, gives
\begin{equation}\label{un:eq:GLload}
 |Q_K(f\phi_h)-\int_K f\phi_h|
 \le Ch_K^{k+1}(\|\phi_h\|_{L^1(K)}+\|\phi_h'\|_{L^1(K)}).
\end{equation}
All subtracted products are within the degree $2k-1$ exactness range.
This scalar load estimate is applied componentwise below.

Return to vector tests $\phi_h\in V_h^2$. Taylor expansion on each
element gives
\[
 \frac{[I_K\F_\tau(x)]'}{|(I_Kx)'|}
       -\frac{[\F_\tau(x)]'}{|x'|}=\frac{\partial_\rho[I_K\F_\tau(x)-\F_\tau(x)]}{|x'|}-\frac{[\F_\tau(x)]'}{|x'|^2}
           \big[\mathbf t\cdot\partial_\rho(I_Kx-x)\big]
           +O(h^{2k}).
\]
Test with $\phi_h'$ and apply \eqref{un:eq:GLflux}. Since $2k\ge k+1$,
the stiffness residual is bounded by $Ch^{k+1}\|\phi_h\|_{H^1}$.

For the mass first consider separate element normals. Expand the
smooth matrix-valued function $p\mapsto(Jp\otimes Jp)/|p|$ at $p=x'$.
Equations \eqref{un:eq:GLforce}--\eqref{un:eq:GLload} show that
\[
 \left|\sum_K Q_K\!\left(
 \frac{J(I_hx)'\otimes J(I_hx)'}{|(I_hx)'|}
 v\cdot\phi_h\right)-b_x(v,\phi_h)\right|
 \le Ch^{k+1}\|\phi_h\|_{H^1}.
\]
Here $I_hv=v$ at the quadrature nodes. At a global node, set
\[
 \mathbf{n}_{K,i}=
 \frac{J(I_hx|_K)'(\rho_i)}{|(I_hx|_K)'(\rho_i)|}.
\]
With $\mu_i$ and $\omega_i$ evaluated at $\widetilde x_h=I_hx$, the assembled
mass differs from this intermediate form by the exact covariance matrix
\begin{equation}\label{un:eq:covariance}
\begin{aligned}
 &\sum_{K\ni i}w_{K,i}|(I_hx|_K)'(\rho_i)|
       \mathbf n_{K,i}\otimes\mathbf n_{K,i}
       -\mu_i\omega_i\otimes\omega_i\\
 &\quad=\sum_{K\ni i}w_{K,i}|(I_hx|_K)'(\rho_i)|
       (\mathbf n_{K,i}-\omega_i)\otimes(\mathbf n_{K,i}-\omega_i).
\end{aligned}
\end{equation}
Every incident normal is $O(h^k)$ from the same smooth normal at
that node. Hence \eqref{un:eq:covariance} is $O(\mu_i h^{2k})$.
Positive-weight norm equivalence bounds its contribution by
$Ch^{2k}\|\phi_h\|_{H^1}$.

The time-semidiscrete equation now gives
\[
 |a_{I_hx}(I_h\F_\tau(x),\phi_h)+b_{I_hx,h}(I_hv,\phi_h)|
 \le Ch^{k+1}\|\phi_h\|_{H^1}.
\]
Subtract \eqref{eq:position-eliminated} and use \eqref{un:eq:coercivity}. This proves
\eqref{un:eq:defect} without a time-step loss, because $v$
has already been bounded uniformly.
\end{proof}

\subsection{The linearized time-semidiscrete step}

An adapted norm for the coupled normal and tangential equations gives
the stability factor $1+C\tau$ needed for iteration.

Fix a smooth input $x$. Within this subsection, scalar
Sobolev norms use arclength derivatives and measure, and $L^p$ norms
use $ds$. Scalar dual norms use the arclength pairing.
Use the normal and tangential velocity components introduced in
Section~\ref{sec:reference-maps}. We retain the notation
$\Lop_\kappa,\Dop_\kappa,\Kop_\kappa,\Sop_\kappa$ from
\eqref{time:LD}--\eqref{time:KS} for their arclength-coordinate
realizations associated with the input geometry $x$. When two inputs are
compared, their curvature subscripts distinguish the operators.
Lemma~\ref{time:operator-bounds}, elliptic regularity, and
Lemma~\ref{reg:lem:regularity} give the required uniform bounds
$\Lop_\kappa^{-1}\Dop_\kappa:H^1\to H^2$ and
$\kappa^2+\Kop_\kappa:H^1\to H^1$.

For $e\in H^1(\Torus;\R^2)$, define the adapted norm associated with $x$ by
\begin{equation}\label{un:eq:norm}
 \|e\|_x=\|e\cdot\mathbf n\|_{H^1(ds)}
       +\|e\cdot\mathbf t-\Lop_\kappa^{-1}\Dop_\kappa
                    (e\cdot\mathbf n)\|_{H^1(ds)}.
\end{equation}
The operators and the norm are evaluated at the indicated input $x$.
In particular, $\|\cdot\|_{x^j}$ uses the frame, arclength measure and
operators $\Lop_{\kappa^j}$ and $\Dop_{\kappa^j}$ of $x^j$,
$0\le j\le M$. The uniform operator and
material metric bounds make \eqref{un:eq:norm} uniformly equivalent to the vector
$H^1(\Torus)$ norm.

\begin{proposition}[Stability of the linearized time-semidiscrete step]
\label{un:prop:continuous-stability}
For the time-semidiscrete solution $x^m$ in Section~\ref{sec:upgrade},
$0\le m<M$, and sufficiently small $\tau$, the linearized update
extends from smooth directions to $H^1(\Torus;\R^2)$ and satisfies
\begin{equation}\label{un:eq:continuous-stability}
 \bigl\|D\F_\tau(x^m)e\bigr\|_{x^{m+1}}
 \le(1+C\tau)\|e\|_{x^m}
 \qquad\text{for all }e\in H^1(\Torus;\R^2).
\end{equation}
The constant is independent of $m$ and $\tau$.
\end{proposition}

\begin{proof}
Taking the normal component of \eqref{time:strong-bgn}, with $x^m$
replaced by $x$, and using \eqref{eq:orientation} gives
\begin{equation}\label{un:eq:cancel}
 \kappa+v_{\mathrm n}=\tau(\partial_s^2v\cdot\mathbf n).
\end{equation}
The velocity, curvature, and their needed derivatives are uniformly
bounded by \eqref{reg:eq:material-high}.

For a smooth vector direction $e:\Torus\to\R^2$, set
$u=D\F_\tau(x)[e]$. In this subsection $w=(u-e)/\tau$ denotes
the corresponding derivative of the vector velocity.
For $f,\phi\in H^1(\Torus;\R^2)$ held fixed, define the geometry
derivatives by
\[
\begin{aligned}
 a'_x[e](f,\phi)&=\left.\frac{d}{d\varepsilon}
       a_{x+\varepsilon e}(f,\phi)\right|_{\varepsilon=0},\\
 b'_x[e](f,\phi)&=\left.\frac{d}{d\varepsilon}
       b_{x+\varepsilon e}(f,\phi)\right|_{\varepsilon=0}.
\end{aligned}
\]
For finite element geometries and a direction $e_h\in V_h^2$,
$a'_{\widetilde x_h}[e_h]$ and $b'_{\widetilde x_h,h}[e_h]$ are
defined by the same derivatives of $a_{\widetilde x_h+\varepsilon e_h}$ and
$b_{\widetilde x_h+\varepsilon e_h,h}$.
Differentiating in the fixed material parameter and using the
arclength derivative and measure of $x$ gives
\begin{equation}\label{un:eq:formderiv}
\begin{aligned}
 a'_x[e](\F_\tau(x),\phi)
 &=-\int(\mathbf t\cdot\partial_s e)\partial_s\F_\tau(x)\cdot\partial_s\phi,\\
 b'_x[e](v,\phi)
 &=\int v_{\mathrm n}
       \big[(\mathbf t\cdot\partial_s e)(\phi\cdot\mathbf n)
          -(\mathbf n\cdot\partial_s e)(\phi\cdot\mathbf t)\big]-\int v_{\mathrm t}(\mathbf n\cdot\partial_s e)(\phi\cdot\mathbf n).
\end{aligned}
\end{equation}
The linearized step therefore satisfies, for every
$\phi\in H^1(\Torus;\R^2)$,
\begin{equation}\label{un:eq:linearized}
 b_x(w,\phi)+a_x(u,\phi)
       +a'_x[e](\F_\tau(x),\phi)+b'_x[e](v,\phi)=0.
\end{equation}
Resolve the input direction and velocity derivative in the frame
$(\mathbf n,\mathbf t)$ of $x$:
\[
\begin{aligned}
 e_{\mathrm n}=e\cdot\mathbf n,\qquad &e_{\mathrm t}=e\cdot\mathbf t,\\
 w_{\mathrm n}=w\cdot\mathbf n,\qquad &w_{\mathrm t}=w\cdot\mathbf t.
\end{aligned}
\]
Thus $\partial_s e\cdot\mathbf t=\partial_s e_{\mathrm t}+\kappa e_{\mathrm n}$ and
$\partial_s e\cdot\mathbf n=\partial_s e_{\mathrm n}-\kappa e_{\mathrm t}$.
Define the scalar distribution $R_xe$ by its action on every
$\psi\in H^1(\Torus)$:
\[
 \begin{aligned}
 \langle R_xe,\psi\rangle
 ={}&\int(\partial_s^2v\cdot\mathbf n)(\partial_s e\cdot\mathbf n)\psi+\int(\partial_s e\cdot\mathbf t)
       \big[(\partial_s v\cdot\mathbf t)\partial_s\psi
             -\kappa(\partial_s v\cdot\mathbf n)\psi\big].
\end{aligned}
\]
Tangential testing in \eqref{un:eq:linearized} and cancellation
\eqref{un:eq:cancel} give
\begin{equation}\label{un:eq:tangent}
 \Lop_\kappa w_{\mathrm t}-\Dop_\kappa w_{\mathrm n}=R_xe.
\end{equation}
In detail, the terms from the normal stiffness at $x$ and the mass
variation are $-\kappa(\partial_s e\cdot\mathbf n)\psi$ and
$-v_{\mathrm n}(\partial_s e\cdot\mathbf n)\psi$. Their sum is
$-\tau(\partial_s^2v\cdot\mathbf n)(\partial_s e\cdot\mathbf n)\psi$. Division by $\tau$ therefore leaves a bounded
coefficient. Since $R_x:H^1(\Torus;\R^2)\to H^{-1}(\Torus)$
is uniformly bounded, so is
$\Lop_\kappa^{-1}R_x:H^1(\Torus;\R^2)\to H^1(\Torus)$, and
\begin{equation}\label{un:eq:eta}
 w_{\mathrm t}=\Lop_\kappa^{-1}\Dop_\kappa w_{\mathrm n}
                 +\Lop_\kappa^{-1}R_xe.
\end{equation}
For the output direction $u=D\F_\tau(x)[e]$, likewise define
$u_{\mathrm n}=u\cdot\mathbf n$ and $u_{\mathrm t}=u\cdot\mathbf t$.
Set
\begin{equation}\label{un:eq:J}
\begin{aligned}
 J_xe={}&-(\partial_s^2v\cdot\mathbf n)\partial_s e_{\mathrm t}
          -\partial_s\big[(\partial_s v\cdot\mathbf n)(\partial_s e\cdot\mathbf t)\big]
          +\kappa(\partial_s v\cdot\mathbf t)(\partial_s e\cdot\mathbf t)
          -\Dop_\kappa\Lop_\kappa^{-1}R_xe.
\end{aligned}
\end{equation}
Normal testing, followed by $u=e+\tau w$ and \eqref{un:eq:eta}, gives
\begin{equation}\label{un:eq:normal}
\begin{aligned}
 (I+\tau\Sop_\kappa)u_{\mathrm n}
 ={}&e_{\mathrm n}+\tau\big[(\kappa^2+\Kop_\kappa)e_{\mathrm n}
                         +v_{\mathrm t}\partial_s e_{\mathrm n}
                         -\kappa v_{\mathrm n}e_{\mathrm n}\big]-\tau(\partial_s\kappa+\kappa v_{\mathrm t})e_{\mathrm t}+\tau^2J_xe.
\end{aligned}
\end{equation}
Formula \eqref{un:eq:J} defines a uniformly bounded map
$H^1(\Torus;\R^2)\to H^{-1}(\Torus)$.
Using $u=e+\tau w$, equation \eqref{un:eq:eta} becomes
\begin{equation}\label{un:eq:gamma}
 u_{\mathrm t}-\Lop_\kappa^{-1}\Dop_\kappa u_{\mathrm n}
 =e_{\mathrm t}-\Lop_\kappa^{-1}\Dop_\kappa e_{\mathrm n}
      +\tau\Lop_\kappa^{-1}R_xe.
\end{equation}

For $f\in H^1(\Torus)$ and small $\tau$, we prove
\begin{gather}
 \|(I+\tau\Sop_\kappa)^{-1}\|_{H^1\to H^1}\le1+C\tau,
 \qquad
 \|(I+\tau\Sop_\kappa)^{-1}\|_{H^{-1}\to H^1}\le C/\tau,
 \label{un:eq:resolvent}\\
 \|(I+\tau\Sop_\kappa)^{-1}(f+\tau v_{\mathrm t}\partial_s f)\|_{H^1}
 \le(1+C\tau)\|f\|_{H^1}.\label{un:eq:drift}
\end{gather}
For $-\partial_s^2$ the first two estimates follow from its
periodic Fourier multipliers. The bounded perturbation
$\kappa^2+\Kop_\kappa$ is
absorbed using the heat-resolvent identity.

For \eqref{un:eq:drift}, first solve
$g-\tau\partial_s^2g=f+\tau v_{\mathrm t}\partial_s f$. Test with $g$ and use
polarization. Splitting $v_{\mathrm t}\partial_s f=v_{\mathrm t}\partial_s g+v_{\mathrm t}\partial_s(f-g)$ and integrating by
parts gives
\[
 \tfrac12(\|g\|_2^2-\|f\|_2^2+\|g-f\|_2^2)
 +\tau\|\partial_s g\|_2^2
 \le C\tau\|g\|_2^2
 +C\tau\|g-f\|_2(\|\partial_s g\|_2+\|g\|_2).
\]
Young's inequality absorbs part of $\|g-f\|_2^2$ and, for small
$\tau$, the resulting $C\tau^2\|\partial_s g\|_2^2$, and yields
\[
 (1-C\tau)\|g\|_2^2
 +\tfrac12\|g-f\|_2^2
 +\tau\|\partial_sg\|_2^2
 \le \|f\|_2^2.
\]
Hence $\|g\|_2\le(1+C\tau)\|f\|_2$. For smooth $f$, let $g_1$
solve the same equation with $f$ replaced by $\partial_sf$. Then
\[
\begin{aligned}
 (I-\tau\partial_s^2)\partial_sg
 &=\partial_sf+\tau v_{\mathrm t}\partial_s^2f
      +\tau(\partial_sv_{\mathrm t})\partial_sf,\\
 \partial_sg
 &=g_1+\tau(I-\tau\partial_s^2)^{-1}
          \bigl((\partial_sv_{\mathrm t})\partial_sf\bigr),
\end{aligned}
\]
and the $L^2$ contraction of the heat resolvent gives
\[
 \|\partial_sg\|_2
 \le(1+C\tau)\|\partial_sf\|_2.
\]
The two squared estimates, followed by density, show that
\[
 \|(I-\tau\partial_s^2)^{-1}
       (f+\tau v_{\mathrm t}\partial_sf)\|_{H^1}
 \le(1+C\tau)\|f\|_{H^1}.
\]
If
$\widetilde g=(I+\tau\Sop_\kappa)^{-1}
 (f+\tau v_{\mathrm t}\partial_sf)$, then
\[
 \widetilde g
 =(I-\tau\partial_s^2)^{-1}(f+\tau v_{\mathrm t}\partial_sf)
 -\tau(I-\tau\partial_s^2)^{-1}
       (\kappa^2+\Kop_\kappa)\widetilde g.
\]
Consequently,
\[
 \|\widetilde g\|_{H^1}
 \le(1+C\tau)\|f\|_{H^1}
      +C\tau\|\widetilde g\|_{H^1},
\]
and absorption proves \eqref{un:eq:drift}.

Substitute
$e_{\mathrm t}=(e_{\mathrm t}-\Lop_\kappa^{-1}\Dop_\kappa e_{\mathrm n})
 +\Lop_\kappa^{-1}\Dop_\kappa e_{\mathrm n}$ in
\eqref{un:eq:normal}. Since
$J_x:H^1(\Torus;\R^2)\to H^{-1}(\Torus)$,
\eqref{un:eq:resolvent} gives
\[
 \tau^2\|(I+\tau\Sop_\kappa)^{-1}J_xe\|_{H^1}
 \le C\tau\|e\|_{H^1}.
\]
Equations \eqref{un:eq:normal} and \eqref{un:eq:gamma}, together with
\eqref{un:eq:resolvent}--\eqref{un:eq:drift}, therefore give
\[
\begin{aligned}
 \|u_{\mathrm n}\|_{H^1}
 &\le(1+C\tau)\|e_{\mathrm n}\|_{H^1}
   +C\tau\|e_{\mathrm t}-\Lop_\kappa^{-1}
                         \Dop_\kappa e_{\mathrm n}\|_{H^1},\\
 \|u_{\mathrm t}-\Lop_\kappa^{-1}\Dop_\kappa u_{\mathrm n}\|_{H^1}
 &\le(1+C\tau)
       \|e_{\mathrm t}-\Lop_\kappa^{-1}
                         \Dop_\kappa e_{\mathrm n}\|_{H^1}
       +C\tau\|e_{\mathrm n}\|_{H^1}.
\end{aligned}
\]
Summing these inequalities yields
\begin{equation}\label{un:eq:fixed-geometry-stability}
 \|u\|_x\le(1+C\tau)\|e\|_x.
\end{equation}

It remains to compare $\|\cdot\|_{x^m}$ and $\|\cdot\|_{x^{m+1}}$.
For each level $j\in\{m,m+1\}$ and smooth periodic scalar function $f$,
their pullbacks to the common material coordinate are
\[
\begin{aligned}
 \Lop_{\kappa^j}f
 &=-\frac1{|(x^j)'|}\partial_\rho
       \left(\frac{f'}{|(x^j)'|}\right)
       +(\kappa^j)^2f,\\
 \Dop_{\kappa^j}f
 &=\frac{\partial_\rho\kappa^jf
             +2\kappa^jf'}{|(x^j)'|}.
\end{aligned}
\]
Use \eqref{un:eq:norm} at each input $x^j$.
The material velocity bounds in \eqref{reg:eq:material-high}
show that the smooth coefficients, frames and arclength weights
change by $O(\tau)$.
For the operator in the second component of \eqref{un:eq:norm}, use
the exact identity
\[
\begin{aligned}
 &\Lop_{\kappa^{m+1}}^{-1}\Dop_{\kappa^{m+1}}
   -\Lop_{\kappa^m}^{-1}\Dop_{\kappa^m}\\
 &\quad=\Lop_{\kappa^{m+1}}^{-1}
       (\Dop_{\kappa^{m+1}}-\Dop_{\kappa^m})
   +\Lop_{\kappa^{m+1}}^{-1}
       (\Lop_{\kappa^m}-\Lop_{\kappa^{m+1}})
       \Lop_{\kappa^m}^{-1}\Dop_{\kappa^m}.
\end{aligned}
\]
It has $H^1\to H^1$ norm $O(\tau)$ by uniform elliptic bounds.
Consequently,
$\|e\|_{x^{m+1}}\le(1+C\tau)\|e\|_{x^m}$ for
$e\in H^1(\Torus;\R^2)$. Combining this norm comparison with
\eqref{un:eq:fixed-geometry-stability} at $x=x^m$ proves
\eqref{un:eq:continuous-stability}.
Uniform norm equivalence is used only in terms already multiplied by
$O(\tau)$. The bound extends the linearized update from smooth
directions to $H^1$ by density.
\end{proof}

\subsection{Comparison of the linearized steps}

Unqualified Sobolev and $L^p$ norms below are taken in the material
coordinate on $\Torus$, whereas $\|\cdot\|_x$ denotes the
arclength-based norm \eqref{un:eq:norm}.
\begin{lemma}\label{un:lem:derivative}
Let $x=x^m$, $0\le m<M$, be the time-semidiscrete solution,
$e_h\in V_h^2$, and $u=D\F_\tau(x)e_h$.
For $0<h\le\tau\le1$,
\begin{equation}\label{un:eq:derivative}
 \|D\F_{h,\tau}(I_hx)e_h-I_hD\F_\tau(x)e_h\|_{H^1}
 \le Ch\tau^{-1}\|e_h\|_{H^1}.
\end{equation}
The interpolation error satisfies
\begin{equation}\label{un:eq:derivinterp}
 \|u-I_hu\|_{H^1}+h^{-1}\|u-I_hu\|_2
 \le Ch\tau^{-1}\|e_h\|_{H^1}.
\end{equation}
\end{lemma}
\begin{proof}
Replace the velocity derivative $w$ in \eqref{un:eq:linearized} by
$(u-e_h)/\tau$. Using the stiffness variation in
\eqref{un:eq:formderiv}, extract the leading directional term by setting
\[
 z=u-\frac{e_h\cdot x'}{|x'|^2}[\F_\tau(x)]'.
\]
The term containing $e_h'$ in
$a_x\big((e_h\cdot x')[\F_\tau(x)]'/|x'|^2,\phi\big)$
cancels the stiffness-variation term
$a'_x[e_h](\F_\tau(x),\phi)$, leaving
\[
\begin{aligned}
 (a_x+\tau^{-1}b_x)(z,\phi)
 ={}&\tau^{-1}b_x\left(e_h-
       \frac{e_h\cdot x'}{|x'|^2}[\F_\tau(x)]',\phi\right)-b'_x[e_h](v,\phi)\\
    &-\int\frac1{|x'|}\partial_\rho\left(
       \frac{[\F_\tau(x)]'\otimes x'}{|x'|^2}\right)e_h\cdot\phi'.
\end{aligned}
\]
The time-semidiscrete stability estimate
\eqref{un:eq:continuous-stability} gives $\|u\|_{H^1}\le C\|e_h\|_{H^1}$.
The subtracted term has a smooth coefficient, so
$\|z\|_{H^1}\le C\|e_h\|_{H^1}$.
The strong equation for $z$ has principal part
$-(|x'|^{-1}z')'$ and $L^2$ forcing consisting of smooth coefficients
times $e_h,e_h',\tau^{-1}e_h,\tau^{-1}z$, together with the derivative
of the coefficient-weighted $e_h$ term in the last integral above.
Thus
\begin{equation}\label{un:eq:zH2}
 \|z\|_{H^2}\le C\tau^{-1}\|e_h\|_{H^1}.
\end{equation}
No second derivative of $e_h$ is invoked.

Smooth-product superapproximation gives
\[
\begin{aligned}
 &\left\|I_h\left(\frac{e_h\cdot x'}{|x'|^2}[\F_\tau(x)]'\right)
          -\frac{e_h\cdot x'}{|x'|^2}[\F_\tau(x)]'\right\|_{H^1}\\
 &\quad+h^{-1}\left\|I_h\left(
           \frac{e_h\cdot x'}{|x'|^2}[\F_\tau(x)]'\right)
          -\frac{e_h\cdot x'}{|x'|^2}[\F_\tau(x)]'\right\|_2
 \le Ch\|e_h\|_{H^1}.
\end{aligned}
\]
Indeed, on each element $K$, write
\[
 B_x=\frac{[\F_\tau(x)]'\otimes x'}{|x'|^2},
 \qquad
 \overline e_K=|K|^{-1}\int_K e_h,
\]
and fix $\rho_K\in K$. Since the constant matrix
$B_x(\rho_K)$ preserves $P^k(K)^2$,
\[
\begin{aligned}
 I_K(B_xe_h)-B_xe_h
 ={}&I_K(B_x\overline e_K)-B_x\overline e_K\\
 &+I_K\bigl((B_x-B_x(\rho_K))(e_h-\overline e_K)\bigr)
      -(B_x-B_x(\rho_K))(e_h-\overline e_K).
\end{aligned}
\]
The local interpolation, Poincar\'e and polynomial inverse estimates give
\[
\begin{gathered}
 \|B_x-B_x(\rho_K)\|_{L^\infty(K)}\le Ch_K,
 \qquad
 \|e_h-\overline e_K\|_{L^2(K)}
       \le Ch_K\|e_h'\|_{L^2(K)},\\
 \|I_K(B_xe_h)-B_xe_h\|_{L^2(K)}
 \le Ch_K^2\bigl(\|\overline e_K\|_{L^2(K)}
                         +\|e_h'\|_{L^2(K)}\bigr),\\
 \|\partial_\rho[I_K(B_xe_h)-B_xe_h]\|_{L^2(K)}
 \le Ch_K\bigl(\|\overline e_K\|_{L^2(K)}
                         +\|e_h'\|_{L^2(K)}\bigr).
\end{gathered}
\]
Summing over the quasi-uniform mesh proves the displayed
superapproximation estimate. Moreover, \eqref{un:eq:zH2} gives
\[
 \|z-I_hz\|_{H^1}+h^{-1}\|z-I_hz\|_2
 \le Ch\|z\|_{H^2}
 \le Ch\tau^{-1}\|e_h\|_{H^1}.
\]
Together with $u=z+B_xe_h$ and $\tau\le1$, this proves
\eqref{un:eq:derivinterp}.

Finally, \eqref{un:eq:continuous-stability},
\eqref{un:eq:derivinterp}, and $h\le\tau$ give directly
\[
 \|I_hu\|_{H^1}
 \le\|u\|_{H^1}+\|I_hu-u\|_{H^1}
 \le C(1+h\tau^{-1})\|e_h\|_{H^1}
 \le C\|e_h\|_{H^1}.
\]

For the fully discrete update at $I_hx$,
Lemma~\ref{un:lem:defect}, an inverse inequality, and
$\F_\tau(x)=x+\tau v$ give
\begin{equation}\label{un:eq:primal}
\begin{aligned}
 \|\partial_\rho[\F_{h,\tau}(I_hx)-I_h\F_\tau(x)]\|_\infty
 &\le Ch^{k+1/2},\\
 \frac{\F_{h,\tau}(I_hx)-I_hx}{\tau}
 &=I_hv+\frac{\F_{h,\tau}(I_hx)-I_h\F_\tau(x)}{\tau}.
\end{aligned}
\end{equation}
For the same direction $e_h$, set $u_h=D\F_{h,\tau}(I_hx)e_h$.
Differentiating the actual assembled discrete equation gives,
for every $\phi_h\in V_h^2$,
\[
\begin{aligned}
 &(a_{I_hx}+\tau^{-1}b_{I_hx,h})(u_h,\phi_h)\\
 &\quad=\tau^{-1}b_{I_hx,h}(e_h,\phi_h)
         -a'_{I_hx}[e_h](\F_{h,\tau}(I_hx),\phi_h)-b'_{I_hx,h}[e_h]\left(
           \frac{\F_{h,\tau}(I_hx)-I_hx}{\tau},\phi_h\right).
\end{aligned}
\]
We compare this equation with \eqref{un:eq:linearized}, after
substituting $(u-e_h)/\tau$ for its velocity derivative. For arbitrary
$f_h,g_h\in V_h^2$,
\begin{align}
 |a_{I_hx}(f_h,g_h)-a_x(f_h,g_h)|
 &\le Ch^k\|f_h\|_{H^1}\|g_h\|_{H^1},\nonumber\\
 |b_{I_hx,h}(f_h,g_h)-b_x(f_h,g_h)|
 &\le C(h^k+h^2)\|f_h\|_{H^1}\|g_h\|_{H^1}.
 \label{un:eq:formcomparison}
\end{align}
The stiffness estimate and the $O(h^k)$ geometry part of the mass
estimate follow from the interpolation bounds used in the proof of
Lemma~\ref{un:lem:defect}. For the remaining mass quadrature error,
the elementwise constant-subtraction argument in
\eqref{un:eq:GLflux}--\eqref{un:eq:GLload} gives the $h^2$ term.
For $k\ge2$, this uses $k+1\le2k-1$. For $k=1$, the trapezoidal
remainder and the elementwise affinity of $f_h,g_h$ give the same bound.

For the differentiated mass, evaluate $p_i$ and $\mu_i$ at
$\widetilde x_h=I_hx$ and differentiate the assembled matrix
$C_i=(Jp_i)(Jp_i)^T/\mu_i$ before comparing with continuous forms.
Since $Dp_i[e_h]=\sum_{K\ni i}w_{K,i}(e_h|_K)'(\rho_i)$, its derivative is
\begin{equation}\label{un:eq:massvariation}
\begin{aligned}
 DC_i[e_h]={}&\frac{(JDp_i[e_h])(Jp_i)^T+(Jp_i)(JDp_i[e_h])^T}{\mu_i}\\
 &-\frac{(Jp_i)(Jp_i)^T}{\mu_i^2}
       \sum_{K\ni i}w_{K,i}
       \frac{(I_hx|_K)'(\rho_i)}{|(I_hx|_K)'(\rho_i)|}
       \cdot (e_h|_K)'(\rho_i).
\end{aligned}
\end{equation}
Replacing the interpolated geometry coefficients by their smooth
nodal values costs $Ch^k\|e_h\|_{H^1}\|\phi_h\|_{H^1}$.
At those smooth values \eqref{un:eq:massvariation} is exactly the
elementwise quadrature of the continuous mass variation. Its smooth
coefficient multiplies $e_h'\phi_h$, of degree at most $2k-1$.
Freezing that coefficient therefore gives
\begin{equation}\label{un:eq:bprimecomp}
 |b'_{I_hx,h}[e_h](I_hv,\phi_h)-b'_x[e_h](v,\phi_h)|
 \le Ch\|e_h\|_{H^1}\|\phi_h\|_{H^1}.
\end{equation}
Exact stiffness and smooth interpolation similarly give an $h^k$
bound for $a'_{I_hx}[e_h](I_h\F_\tau(x),\phi_h)-a'_x[e_h](\F_\tau(x),\phi_h)$.
Finally, Lemma~\ref{un:lem:defect} and \eqref{un:eq:primal} yield
the replacement errors
\begin{align*}
 \left|b'_{I_hx,h}[e_h]\left(
 \frac{\F_{h,\tau}(I_hx)-I_h\F_\tau(x)}{\tau},\phi_h\right)\right| & \le Ch^{k+1}\tau^{-1}\|e_h\|_{H^1}\|\phi_h\|_{H^1},\\
 \left|a'_{I_hx}[e_h](
	\F_{h,\tau}(I_hx)-I_h\F_\tau(x),\phi_h)\right| & \le Ch^{k+1/2}\|e_h\|_{H^1}\|\phi_h\|_{H^1}.
\end{align*}
The first uses $\|\F_{h,\tau}(I_hx)-I_h\F_\tau(x)\|_\infty\le Ch^{k+1}$ and
weighted Cauchy--Schwarz for $e_h'$ in \eqref{un:eq:massvariation}.
The second uses
$\|\partial_\rho[\F_{h,\tau}(I_hx)-I_h\F_\tau(x)]\|_\infty$.

Combining these bounds, the residual of $I_hu$ in the discrete
derivative equation is at most
\[
 C ( h\tau^{-1}+h^2\tau^{-2}+(h^k+h^2)\tau^{-1}
       +h+h^{k+1}\tau^{-1}+h^{k+1/2} )
 \|e_h\|_{H^1}\|\phi_h\|_{H^1}.
\]
For $h\le\tau\le1$ it is bounded by
$Ch\tau^{-1}\|e_h\|_{H^1}\|\phi_h\|_{H^1}$.
Coercivity proves \eqref{un:eq:derivative}.
\end{proof}

Combining \eqref{un:eq:continuous-stability} with Lemma
\ref{un:lem:derivative}, using uniform norm equivalence only on the
consistency errors, gives
\begin{equation}\label{un:eq:discrete-stability}
 \bigl\|D\F_{h,\tau}(I_hx^{m})e_h\bigr\|_{x^{m+1}}
 \le(1+C\tau+Ch\tau^{-1})\|e_h\|_{x^{m}}.
\end{equation}

\section{The fully discrete error estimate}\label{sec:nonlinear}

For fixed $k\ge1$, we bound the nonlinear remainder at the smooth
interpolant and use Section~\ref{sec:consistency} to complete the
induction for Theorem~\ref{un:thm:main}. Constants may depend on $k$
but are independent of $h,\tau,m$.

\subsection{The nonlinear remainder}

\begin{lemma}\label{un:lem:remainder}
Let $x=x^m$, $0\le m<M$, be the time-semidiscrete solution,
$e_h\in V_h^2$, and
$0<h\le\tau\le1$. If the segment
$I_hx+\varepsilon e_h$, $0\le\varepsilon\le1$, lies in the fixed
input neighborhood of \eqref{un:eq:coercivity}, the remainder
\[
 R_h(e_h)=\F_{h,\tau}(I_hx+e_h)-\F_{h,\tau}(I_hx)-D\F_{h,\tau}(I_hx)e_h
\]
satisfies
\begin{equation}\label{un:eq:remainder}
 \|R_h(e_h)\|_{H^1}
 \le C(h^{-1/2}+\tau^{-1})\|e_h\|_{H^1}^2.
\end{equation}
For $h\le\tau^2$, the right-hand side is at most
$Ch^{-1/2}\|e_h\|_{H^1}^2$.
\end{lemma}
\begin{proof}
Throughout this proof, $u_h=D\F_{h,\tau}(I_hx)e_h$ is the vector
update derivative at $I_hx$. The vector velocity is $v=(\F_\tau(x)-x)/\tau$.
Lemmas~\ref{un:lem:defect} and~\ref{un:lem:derivative}, together
with the inverse inequality and \eqref{un:eq:continuous-stability},
give
\begin{equation}\label{un:eq:reference-bounds}
\begin{gathered}
 \|\F_{h,\tau}(I_hx)\|_{\W}\le C,\qquad
 \|\F_{h,\tau}(I_hx)-I_hx\|_\infty\le C\tau,\\
 \|u_h\|_{H^1}\le C\|e_h\|_{H^1}.
\end{gathered}
\end{equation}
Indeed, the one-step defect satisfies
\[
\begin{aligned}
 \|\F_{h,\tau}(I_hx)-I_h\F_\tau(x)\|_{H^1}&\le Ch^{k+1},\\
 \|\F_{h,\tau}(I_hx)-I_h\F_\tau(x)\|_{\W}&\le Ch^{k+1/2}.
\end{aligned}
\]
Use $h^{k+1}\le\tau$ and the identity
\[
 \F_{h,\tau}(I_hx)-I_hx
 =\tau I_hv+\F_{h,\tau}(I_hx)-I_h\F_\tau(x).
\]

For the fixed input pair $I_hx$ and $I_hx+e_h$, define the form
increments by
\[
 \delta a=a_{I_hx+e_h}-a_{I_hx},\qquad
 \delta b_h=b_{I_hx+e_h,h}-b_{I_hx,h}.
\]
Their first-order remainders are
$\delta^2a=\delta a-a'_{I_hx}[e_h]$ and
$\delta^2b_h=\delta b_h-b'_{I_hx,h}[e_h]$.
Subtracting the two step equations and the derivative equation gives,
for every $\phi_h\in V_h^2$, the exact identity
\begin{align}
 &(a_{I_hx+e_h}+\tau^{-1}b_{I_hx+e_h,h})(R_h(e_h),\phi_h)\nonumber\\
 &\quad=-\delta a(u_h,\phi_h)
      -\tau^{-1}\delta b_h(u_h-e_h,\phi_h)\nonumber\\
 &\qquad-\delta^2a(\F_{h,\tau}(I_hx),\phi_h)
      -\tau^{-1}\delta^2b_h(\F_{h,\tau}(I_hx)-I_hx,\phi_h).
 \label{un:eq:remainder-identity}
\end{align}

The first two derivatives of $p\mapsto|p|^{-1}$
are bounded on the input neighborhood. The one-dimensional inverse
estimates for fixed-degree polynomials give
\begin{align*}
 |\delta a(u_h,\phi_h)|
 &\le C\|e_h'\|_\infty\|u_h\|_{H^1}\|\phi_h\|_{H^1}
 \le Ch^{-1/2}\|e_h\|_{H^1}^2\|\phi_h\|_{H^1},\\
 |\delta^2a(\F_{h,\tau}(I_hx),\phi_h)|
 &\le C\|e_h'\|_{L^4}^2\|[\F_{h,\tau}(I_hx)]'\|_\infty\|\phi_h'\|_2
 \le Ch^{-1/2}\|e_h\|_{H^1}^2\|\phi_h\|_{H^1}.
\end{align*}

For the assembled normal mass, use the matrices
$C_i=(Jp_i)(Jp_i)^T/\mu_i$ from \eqref{un:eq:massvariation}, now
evaluated at the variable input $\widetilde x_h=I_hx+\varepsilon e_h$,
$0\le\varepsilon\le1$. Define the nodal weights and weighted derivative
magnitudes by
\[
 m_i=\sum_{K\ni i}w_{K,i},\qquad
 E_i=\left(m_i^{-1}\sum_{K\ni i}w_{K,i}|(e_h|_K)'(\rho_i)|^2\right)^{1/2}.
\]
Positive weights and uniform speed bounds along the input segment
give
\[
 \frac{|Dp_i[e_h]|+|D\mu_i[e_h]|}{m_i}\le CE_i,\qquad
 \frac{|D^2\mu_i[e_h,e_h]|}{m_i}\le CE_i^2.
\]
The first estimate is weighted Cauchy--Schwarz. The second follows
by twice differentiating $|(\widetilde x_h)'|$. Since $p_i$ is linear in $\widetilde x_h$,
differentiation of $C_i$ yields
\[
 |DC_i[e_h]|\le Cm_iE_i,\qquad
 |D^2C_i[e_h,e_h]|\le Cm_iE_i^2.
\]
Moreover,
\begin{equation}\label{un:eq:quadrature-energy}
 \sum_i m_iE_i^2
 =\sum_{K\in\mathcal T_h}Q_K(|e_h'|^2)
 =\|e_h'\|_2^2,
\end{equation}
because $|e_h'|^2$ has degree at most $2k-2$, within the
Gauss--Lobatto exactness range. Integrating the matrix derivatives
along the segment, using \eqref{un:eq:quadrature-energy} and
$\sum_i m_i=1$, gives, for arbitrary $f_h,\phi_h\in V_h^2$,
\begin{align*}
 |\delta b_h(f_h,\phi_h)|
 &\le C\|e_h\|_{H^1}\|f_h\|_\infty\|\phi_h\|_\infty,\\
 |\delta^2b_h(f_h,\phi_h)|
 &\le C\|e_h\|_{H^1}^2\|f_h\|_\infty\|\phi_h\|_\infty.
\end{align*}
Consequently \eqref{un:eq:reference-bounds} and
$H^1\hookrightarrow L^\infty$ bound the two mass terms in
\eqref{un:eq:remainder-identity} by
$C(\tau^{-1}+1)\|e_h\|_{H^1}^2\|\phi_h\|_{H^1}$.
Coercivity of $(a_{I_hx+e_h}+\tau^{-1}b_{I_hx+e_h,h})$ proves \eqref{un:eq:remainder}.
\end{proof}

\subsection{Proof of Theorem~\ref{un:thm:main}}

\begin{proof}
For $0\le m\le M$, let $e_h^m=x_h^m-I_hx^m$.
We estimate this position error in the adapted norm \eqref{un:eq:norm}.
The initial error is exactly zero.
Choose $c\le1$ and $\tau_0\le1$. As long as the input lies in the
fixed neighborhood, Lemmas~\ref{un:lem:defect} and
\ref{un:lem:remainder}, and \eqref{un:eq:discrete-stability}, give
\begin{equation}\label{un:eq:recurrence}
\begin{aligned}
 \|e_h^{m+1}\|_{x^{m+1}}
 &\le(1+C_0\tau)\|e_h^m\|_{x^m}+C_1h^{k+1}\\
 &\quad+C_2h^{-1/2}\|e_h^m\|_{x^m}^2,
 \qquad \|e_h^0\|_{x^0}=0.
\end{aligned}
\end{equation}
since $h/\tau\le c\tau$. The constants are fixed uniformly for
$0<c\le1$ and are independent of the induction bound chosen below.

Let $C_T$ be a Gronwall constant for the linear recurrence
with coefficient $1+(C_0+1)\tau$ and forcing $C_1h^{k+1}$,
and fix $K=2C_T+1$. We close the induction at the scale
\begin{equation}\label{un:eq:stopping}
 \|e_h^m\|_{x^m}\le K h^{k+1}/\tau.
\end{equation}
Decrease $\tau_0$ so that
\[
 C_2Kc^{3/2}\tau_0\le1.
\]
For every $k\ge1$ and $h\le1$, $h^{k+1/2}\le h^{3/2}$.
Thus, whenever \eqref{un:eq:stopping} holds, the nonlinear
coefficient satisfies
\[
 C_2h^{-1/2}\|e_h^m\|_{x^m}
 \le C_2K h^{k+1/2}/\tau
 \le C_2K h^{3/2}/\tau
 \le C_2Kc^{3/2}\tau^2\le\tau.
\]
Consequently,
\[
 C_2h^{-1/2}\|e_h^m\|_{x^m}^2
 \le\tau\|e_h^m\|_{x^m},
\]
so the quadratic term is included in the linear factor
$1+(C_0+1)\tau$ in \eqref{un:eq:recurrence}.
The inverse inequality and interpolation also give
\[
 \|x_h^m-x^m\|_{\W}
 \le C\{h^k+K h^{k+1/2}/\tau\}
 \le C(1+K\sqrt c)h^k.
\]
For sufficiently small $\tau_0$, this places the entire segment from
$I_hx^m$ to $x_h^m$ in the fixed neighborhood. Hence the next step
exists uniquely, and \eqref{un:eq:recurrence} implies
\begin{equation}\label{un:eq:gronwall}
 \|e_h^m\|_{x^m}\le C_T h^{k+1}/\tau
\end{equation}
through that next index. This strictly improves
\eqref{un:eq:stopping}, so the induction reaches $M$.
Uniform norm equivalence gives the supercloseness estimate
\begin{equation}\label{un:eq:superclose}
 \max_{m\le M}\|x_h^m-I_hx^{m}\|_{H^1}\le C h^{k+1}/\tau.
\end{equation}
Interpolation and one inverse estimate yield
\[
 \|x_h^m-x^m\|_{\W}
 \le C\{h^k+h^{-1/2}h^{k+1}/\tau\}
 =Ch^k(1+\sqrt h/\tau)\le Ch^k.
\]
Sobolev embedding and interpolation also give $\max_{m\le M}\|x_h^m-x^{m}\|_{L^\infty}
\le C h^{k+1}/\tau$.

The spatial $\W$ proximity proves regularity and embeddedness of
every finite element curve. On short parameter intervals its
elementwise tangents have a positive component along a fixed nearby
smooth tangent. On the complementary pairs of parameters, the
time-semidiscrete curves $\Gamma^m$ retain the uniform separation
of the exact curve family. Both properties
persist under the displayed small $\W$ perturbation. Positive nodal
mass recovers the scalar curvature uniquely from the eliminated
position equation.

Finally, Lemma~\ref{reg:lem:regularity} and
$\widehat x_h^m-\widehat x^m=(x_h^m-x^m)\circ(\chi^m)^{-1}$ give
\[
 \|\widehat x_h^m-\widehat x^m\|_{\W}
 \le C\|x_h^m-x^m\|_{\W}\le Ch^k.
\]
Combining this with \eqref{up:eq:first-order} gives
\eqref{un:eq:total}. Taking the images
of the matched maps proves the Hausdorff estimate. The matched
$H^1$, $L^2$ and $L^\infty$ bounds follow from the
$W^{1,\infty}$ bound on the fixed parameter domain.
\end{proof}
\section*{Statements and declarations}

{\small
\raggedright
\bibliographystyle{abbrv}
\bibliography{ref}
}
\end{document}